\documentclass[11pt]{scrartcl}
\usepackage[utf8]{inputenc}
\RequirePackage[T1]{fontenc}
\usepackage{lmodern}
\usepackage{alphabeta}

\usepackage{config}
\usepackage{titling}
\usepackage{xspace}
\usepackage{textpos}

\usepackage[a4paper, top=3cm, bottom=2cm, left=2.5cm, right=2.5cm, includefoot]{geometry}

\usepackage{soul}

\usepackage{tikz}
\usepackage{tikzmacros}

\usetikzlibrary{shapes,decorations}
\usetikzlibrary{decorations.markings}
\usetikzlibrary{arrows}
\usetikzlibrary{arrows}
\usetikzlibrary{shapes.geometric}
\usetikzlibrary{calc}
\usetikzlibrary{automata,positioning,arrows.meta}
\usetikzlibrary{decorations.markings}
\usetikzlibrary{decorations.pathmorphing}
\usetikzlibrary{matrix}
\usetikzlibrary{shadows}
\usetikzlibrary{backgrounds}

\newdimen\arrowsize
\newlength{\arrowlength}
\newlength{\arrowangle}
\newlength{\arrowthickness}

 \pgfarrowsdeclare{slim}{slim}
 {
 \arrowsize=0.7pt
 \advance\arrowsize by .5\pgflinewidth
 \pgfarrowsleftextend{-4\arrowsize-.5\pgflinewidth}
 \pgfarrowsrightextend{.5\pgflinewidth}
 }
 {
 \advance\arrowsize by .5\pgflinewidth
 \pgfsetdash{}{0pt}
 \pgfsetroundjoin
 \pgfsetroundcap
 \pgfpathmoveto{\pgfpointpolar{\arrowangle}{\arrowlength}}
 \pgfpathlineto{\pgfpointorigin}
 \pgfusepathqstroke
 \pgfpathmoveto{\pgfpointpolar{\arrowangle}{\arrowlength}}
 \pgfpathcurveto{\pgfpointpolar{\arrowangle*1.04}{\arrowthickness+0.5*(\arrowlength-\arrowthickness)}}%
 {\pgfpointpolar{\arrowangle*1.1}{\arrowthickness+0.3*(\arrowlength-\arrowthickness)}}
 {\pgfpointpolar{180}{\arrowthickness}}
 \pgfpathlineto{\pgfpointorigin}
 \pgfusepathqfill
 \pgfpathmoveto{\pgfpointpolar{-\arrowangle}{\arrowlength}}
 \pgfpathlineto{\pgfpointorigin}
 \pgfusepathqstroke
 \pgfpathmoveto{\pgfpointpolar{-\arrowangle}{\arrowlength}}
 \pgfpathcurveto{\pgfpointpolar{-\arrowangle*1.04}{\arrowthickness+0.5*(\arrowlength-\arrowthickness)}}%
 {\pgfpointpolar{-\arrowangle*1.1}{\arrowthickness+0.3*(\arrowlength-\arrowthickness)}}
 {\pgfpointpolar{180}{\arrowthickness}}
 \pgfpathlineto{\pgfpointorigin}
 \pgfusepathqfill
 }

\tikzstyle{vertex}=[circle,inner sep=1,minimum size =2mm,semithick,fill=orange!10, draw=black]
\tikzstyle{point}=[circle,inner sep=1,fill=black, draw=black]
\tikzstyle{brace}=[thin,decorate,decoration=brace]
\tikzstyle{ie}=[thin,dashed,gray]
\tikzstyle{edge}=[draw=black]
\tikzstyle{arc}=[draw=black, >=Stealth, ->]

\newcommand{\hh}{\end{document}}

\usepackage[inline]{enumitem}

\hypersetup{
colorlinks=true,
linkcolor=AO!65!black,
citecolor=AO!65!black,
urlcolor=AppleGreen!65!black,
bookmarksopen=true,
bookmarksnumbered,
bookmarksopenlevel=2,
bookmarksdepth=3
}

\title{Generating strongly $2$-connected digraphs}
\predate{}
\date{}
\postdate{}

\preauthor{}
\DeclareRobustCommand{\authorthing}{
    \begin{center}
	    Meike Hatzel\thanks{\href{mailto:research@meikehatzel.com}{research@meikehatzel.com}, Meike Hatzel's research was supported by the Federal Ministry of Education and
        Research (BMBF) and by a fellowship within the IFI programme of the German Academic Exchange Service (DAAD) and by the Institute for Basic Science (IBS-R029-C1).} \\ 
		{\small Discrete Mathematics Group, Institute for Basic Science (IBS), Daejeon, South Korea}\\
		\medskip
	    Stephan Kreutzer\thanks{\href{mailto:stephan.kreutzer@tu-berlin.de}{stephan.kreutzer@tu-berlin.de}} \\
		{\small Technical University of Berlin, Germany}\\
		\medskip
        Evangelos Protopapas\thanks{\href{mailto:vagelis.protopapas@gmail.com}{vagelis.protopapas@gmail.com}, Evangelos Protopapas was supported by the French-German Collaboration ANR/DFG Project UTMA (ANR-20-CE92-0027).} \\
		{\small LIRMM, University of Montpellier, CNRS, Montpellier, France}\\
		\medskip
        Florian Reich\thanks{\href{mailto:florian.reich@uni-hamburg.de}{florian.reich@uni-hamburg.de}, Florian Reich was supported by the Studienstiftung des deutschen Volkes.} \\
		{\small University of Hamburg, Germany}\\
		\medskip
        Giannos Stamoulis\thanks{\href{mailto:giannos.stamoulis@mimuw.edu.pl}{giannos.stamoulis@mimuw.edu.pl}, Giannos Stamoulis was supported by the project BOBR, which is funded by the European Research Council (ERC) under the European Union’s Horizon 2020 research and innovation programme with grant agreement No.~948057. During the research stage of this work, G.~S.~was affiliated with LIRMM, University of Montpellier, CNRS, Montpellier, France, and supported by the ANR project ESIGMA (ANR-17-CE23-0010) and the French-German Collaboration ANR/DFG Project UTMA (ANR-20-CE92-0027).} \\
		{\small Institute of Informatics, University of Warsaw, Poland}\\
		\medskip
	    Sebastian Wiederrecht\thanks{\href{mailto:wiederrecht@kaist.ac.kr}{wiederrecht@kaist.ac.kr}}\\
		{\small School of Computing, KAIST, Daejeon, South Korea}\\
\end{center}}
\author{\authorthing}
\postauthor{}

\newif\ifcomment
\commentfalse
\commenttrue

\begin{document}
\maketitle

\begin{abstract}
\noindent We prove that there exist four operations such that given any two strongly $2$-connected digraphs $H$ and $D$ where $H$ is a butterfly-minor of $D$, there exists a sequence $D_0,\dots, D_n$ where $D_0=H$, $D_n=D$ and for every $0\leq i\leq n-1$, $D_i$ is a strongly $2$-connected butterfly-minor of $D_{i+1}$ which is obtained by a single application of one of the four operations.

As a consequence of this theorem, we obtain that every strongly $2$-connected digraph can be generated from a concise family of strongly $2$-connected digraphs by using these four operations.
\smallskip

\noindent\textbf{Keywords.} braces, strongly $2$-connected digraphs, butterfly-minor, splitter theorem
\end{abstract}

\begin{textblock}{20}(12.2,0.7)
   \includegraphics[width=80px]{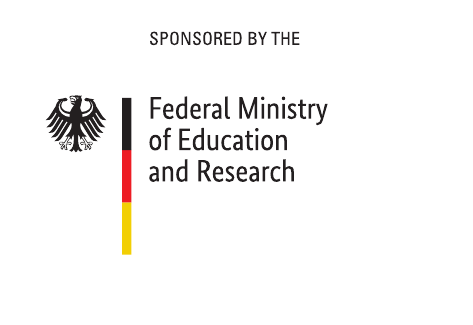}%
\end{textblock}
\begin{textblock}{20}(12.55, 1.7)
\includegraphics[width=40px]{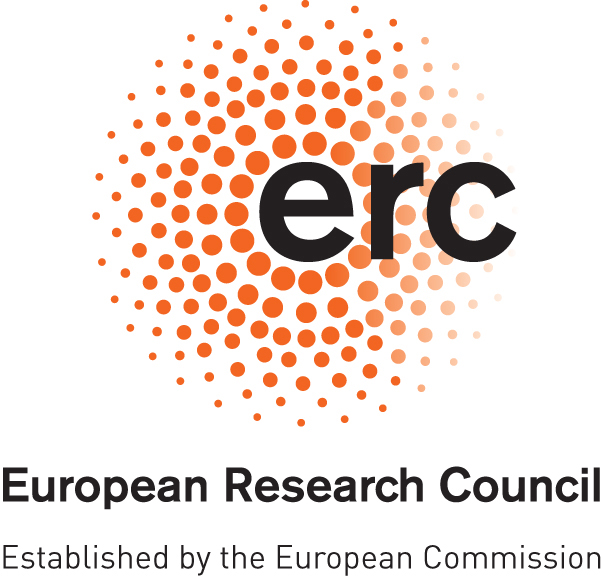}%
\end{textblock}
\begin{textblock}{20}(12.55, 2.5)
\includegraphics[width=40px]{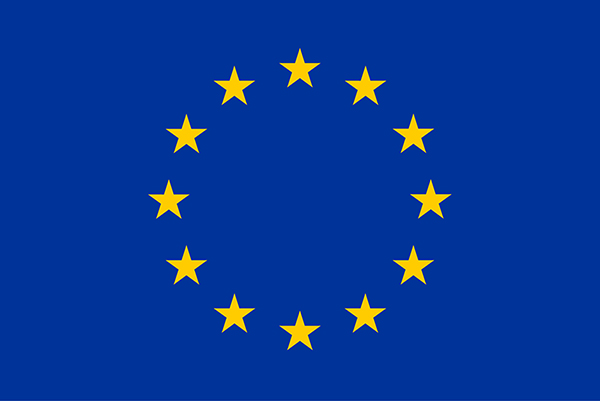}%
\end{textblock}

\thispagestyle{empty}

\newpage
    \thispagestyle{empty}
	\tableofcontents
	
\newpage
\setcounter{page}{1}

\section{Introduction}
\label{sec:intro}
Generation theorems for graphs respecting certain containment relations and connectivity conditions have played a central role in the study of said containment relations.
Among the best-known results in this direction is the \emph{ear decomposition} of $2$-connected graphs \cite{Whitney1932}.
Given a graph $G$ and a subgraph $H\subseteq G,$ we say that a path $P\subseteq G$ is an \emph{$H$-ear} in $G$ if it shares exactly its first and last vertex with $H$ and is otherwise disjoint from $H$.
Whitney's theorem on ear-decompositions says that for every pair of $2$-connected graphs $H$ and $G$, where $H$ is a proper subgraph of $G$, there exists a sequence
\begin{equation*}
    G_0 \subseteq G_1 \subseteq \dots \subseteq G_{\ell}
\end{equation*}
such that $G_0=H$, $G_{\ell}=G$, and, for all $i\in[\ell]$, the graph $G_i$ is obtained from the graph $G_{i-1}$ by adding a single $G_{i-1}$-ear.

The consequences of this theorem are two-fold.
The first is that every $2$-connected subgraph of a $2$-connected graph can be reached by iteratively removing paths while maintaining $2$-connectivity for all intermediate graphs.
The second one is the generative result.
Since every $2$-connected graph contains a cycle as a subgraph and every proper subgraph of a cycle is no longer $2$-connected, the cycles represent a minimal, or basic, family of $2$-connected graphs.
Moreover, Whitney's theorem asserts that every $2$-connected graph can be constructed by starting with a cycle and iteratively adding ears.

It is a straightforward observation that Whitney's theorem also holds for digraphs.
Given a strongly connected digraph $H$, we say that a directed path $P$ is a \emph{directed $H$-ear} if $P$ is an $H$-ear in the underlying undirected graph.
We say that the graph $D \coloneqq H\cup P$ is obtained from $H$ by a \emph{directed-ear-addition}.
Let $\mathcal{B}_1$ be the set of all directed cycles and $\mathcal{D}_1$ be the class of all digraphs that can be generated from members of $\mathcal{B}_1$ by a sequence of directed-ear-additions.
By using Whitney's theorem, one may observe that $\mathcal{D}_1$ is exactly the class of strongly connected digraphs.

A possible reformulation of both results, the directed and the undirected one, can be stated using the \textsl{topological minor relation}.
A (di)graph $D$ is a \emph{subdivision} of a (di)graph $H$ if it is isomorphic to a (di)graph $H'$ which is obtained from $H$ by replacing the edges of $H$ with (directed) paths of length at least one that are internally vertex-disjoint from $H$, and, in case $H$ is a digraph, are directed in the same way as the edges they replace.
If $D$ is a subdivision of $H$ with exactly one additional vertex, then we say that $D$ is obtained by \emph{subdividing an edge} of $H$.
Since subdividing edges clearly preserves $2$-connectivity (and strong connectivity for digraphs), a reformulation of Whitney's theorem states that
\begin{enumerate}
    \item every $2$-connected graph can be obtained from $K_3$ by a sequence of edge-subdivisions and ear-additions, and
    \item every strongly connected digraph can be obtained from the directed two-cycle by a sequence of edge-subdivisions and ear-additions.
\end{enumerate}

What about higher notions of connectivity?
Let $H$ be a $3$-connected graph and $G$ be some graph with a minimum degree of at least three such that there exists an edge $e\in E(G)$ for which $G-e$ is a subdivision of $H$.
In this case, we say that $G$ is obtained from $H$ by a \emph{$3$-augmentation}.
It was proven by Barnette and Gr\"unbaum and, independently, Titov in their PhD thesis \cite{BarnetteG1969,titov1975constructive} that the class of $3$-connected graphs is exactly the class of graphs that can be generated from $K_4$ by sequences of $3$-augmentations.
Notice that if $G$ can be obtained from a $3$-connected graph $H$ by a sequence of $3$-augmentations, then $H$ must be a minor of $G$.

Let $H$ be a graph and $x\in \V{H}$ such that $N(x)$ is the disjoint union of two sets $X_1$ and $X_2$, both of size at least two.
We say that the graph
\begin{equation*}
    G\coloneqq (H-x)+x_1+x_2+\{x_1,x_2\}+\{ \{x_i,y\} \mid \text{for all }i\in[2]\text{ and all }y \in X_i \} 
\end{equation*}
is obtained from $H$ by \emph{$3$-expansion}.
If $G$ contains an edge $e\in E(G)$ such that $G-e=H$, we say that $G$ is obtained from $H$ by \emph{edge-addition}.

A seminal result by Tutte \cite{Tutte1961} states that the class of all graphs that can be generated from the set of wheels by sequences of $3$-expansions and edge-additions is exactly the set of $3$-connected graphs.
Seymour \cite{seymour1980splitter} strengthened Tutte's generation theorem by proving that, if $G$ and $H$ are 3-connected graphs such that $H$ is a proper minor of $G$, then there exists a 3-connected graph $K$ such that $K$ is a minor of $G$ and one of the following statements holds:
\begin{itemize}
    \item $K$ can be obtained from $H$ by 3-expansion or edge addition, or
    \item $H$ and $K$ are wheels and $|\V{H}| + 1 = |V(K)|$.
\end{itemize}

\subsection{Our result: Generating strongly 2-connected digraphs}\label{subsec:ourresult}
In this paper, we adapt ideas from the work of McCuaig \cite{mccuaig2001bracegeneration} on matching minors to the concept of butterfly-minors and prove a generation theorem for the class of strongly $2$-connected digraphs.
That is, we provide four families of operations, which allow us to generate the family $\mathcal{D}_2$ of all strongly $2$-connected digraphs starting from a concise base set $\mathcal{B}_2$.

The four operations are as follows (see~\cref{subsec:augmentations} for details):
\begin{enumerate}
    \item[(a)] The \textbf{basic augmentation} is an addition of a single edge, possibly after ``splitting'' one or two vertices,
    \item[(b)] the \textbf{chain augmentation} is the operation of ``weaving'' a directed path through an already existing edge $e$, possibly after ``splitting'' the endpoints of $e$, 
    \item[(c)] the \textbf{collarette augmentation} is the operation of subdividing two edges $e_1$ and $e_2$ forming a directed two-cycle, 
    introducing an edge from one of the new subdivision vertices, $x$, to the other subdivision vertex, $y$, and finally ``weaving'' a directed path from $y$ to $x$ through the edge $(x,y)$, and 
    \item[(d)] the \textbf{bracelet augmentation} is the operation of ``piercing'' a directed path $P$ of length two from an in-neighbour of a vertex $u$ with in- and out-degree exactly two, to an out-neighbour of $u$ through some edge which is incident with $u$ but not with both endpoints of $P$.
\end{enumerate}
See \cref{fig:augmentations1} for an illustration of these operations.
We provide formal definitions for the four types of augmentations and the operations contained therein in~\cref{subsec:augmentations}.
In~\cref{subsec:augment_discussion} we also comment on the fact that the chain augmentation allows the introduction of an unbounded number of vertices in a single step, which does not appear in similar generation theorems based on minor-like containment relations.

\begin{figure}[!ht]
    \centering
    \input{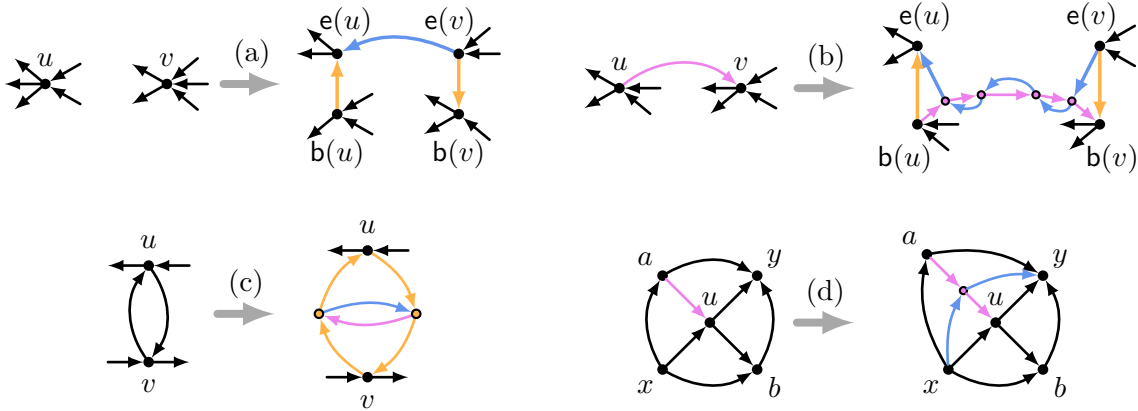}
    \caption{The four types of augmentations: (a) basic augmentation, (b) chain augmentation, (c) collarette augmentation, and (d) bracelet augmentation.}
    \label{fig:augmentations1}
\end{figure}

An edge $(u,v)$ in a digraph is said to be \emph{butterfly-contractible} if it is the only outgoing edge of $u$ or the only incoming edge of $v$.
A digraph $H$ is obtained from a digraph $D$ by \emph{butterfly-contraction} if there exists a butterfly-contractible edge $(u,v)$ in $D$ such that $H$ is isomorphic to the graph obtained from $D$ by identifying $u$ and $v$ into a single vertex and deleting all resulting loops and parallel edges.

A digraph $H$ is said to be a \emph{butterfly-minor} of a digraph $D$ if $H$ is isomorphic to a graph $H'$ that can be obtained from $D$ by a sequence of edge-deletions, vertex-deletions, and butterfly-contractions.

To state our main theorems, the remaining piece is the base set $\mathcal{B}_2$.
Let $G$ be a graph.
We denote by $\Bidirected{G}{}$ the digraph obtained from $G$ by replacing every edge $\{u, v\}\in E(G)$ with the two edges $(u,v)$ and $(v,u)$ forming a directed two-cycle, also known as a \emph{digon}.
The digraph $A_4$ is the strongly connected digraph on four vertices obtained from a directed four-cycle by connecting each pair of anti-diametrical vertices with a digon.
Then $\mathcal{B}_2$ is the set consisting of $A_4$ and all bidirected cycles $\Bidirected{C}{k}$, $k\geq 3$.
See~\cref{fig:base_class} for an illustration.

\begin{figure}[!ht]
    \centering
    \begin{tikzpicture}
		
		\node (D1) [v:ghost] {};
        \node (D2) [v:ghost,position=0:25mm from D1] {};
		\node (D3) [v:ghost,position=0:23mm from D2] {};
		\node (D4) [v:ghost,position=0:26mm from D3] {};
        \node (D5) [v:ghost,position=0:29mm from D4] {};
		\node (D6) [v:ghost,position=0:20mm from D5] {$\mathbf{\dots}$};
		
        \node (C1) [v:ghost,position=0:0mm from D1] {
			
			\begin{tikzpicture}[scale=0.8]
			
			\pgfdeclarelayer{background}
			\pgfdeclarelayer{foreground}
			
			\pgfsetlayers{background,main,foreground}
			
			\begin{pgfonlayer}{main}
			
			\node (C) [] {};
			
			\node (v1) [v:main,position=90:12mm from C] {};
			\node (v2) [v:main,position=270:12mm from C] {};
			\node (v3) [v:main,position=180:6mm from C] {};
			\node (v4) [v:main,position=0:6mm from C] {};

            \draw [e:main,->] (90:12mm) arc (90:270:12mm);
            \draw [e:main] (270:12mm) arc (270:360:12mm);
            \draw [e:main,->] (0:12mm) arc (0:90:12mm);
   
			\draw (v3) [e:main,->,bend right=15] to (v4);
			\draw (v4) [e:main,->,bend right=15] to (v3);
			
			\draw (v1) [e:main,->,bend right=15] to (v3);
			\draw (v3) [e:main,->,bend right=15] to (v2);
			\draw (v2) [e:main,->,bend right=15] to (v4);
			\draw (v4) [e:main,->,bend right=15] to (v1);
			\end{pgfonlayer}
			
			\begin{pgfonlayer}{background}
			\end{pgfonlayer}
   
			\begin{pgfonlayer}{foreground}
			\end{pgfonlayer}
			\end{tikzpicture}
			
		};

		\node (C2) [v:ghost,position=0:0mm from D2] {
			
			\begin{tikzpicture}[scale=0.8]
			
			\pgfdeclarelayer{background}
			\pgfdeclarelayer{foreground}
			
			\pgfsetlayers{background,main,foreground}
			
			\begin{pgfonlayer}{main}
			
			\node (C) [] {};
			
			\node (v1) [v:main,position=90:10mm from C] {};
			\node (v2) [v:main,position=210:10mm from C] {};
			\node (v3) [v:main,position=330:10mm from C] {};
			
			\draw (v1) [e:main,->,bend right=15] to (v2);
			\draw (v2) [e:main,->,bend right=15] to (v3);
			\draw (v3) [e:main,->,bend right=15] to (v1);
			
			\draw (v1) [e:main,->,bend right=15] to (v3);
			\draw (v2) [e:main,->,bend right=15] to (v1);
			\draw (v3) [e:main,->,bend right=15] to (v2);
			\end{pgfonlayer}
			
			\begin{pgfonlayer}{background}
			\end{pgfonlayer}
   
			\begin{pgfonlayer}{foreground}
			\end{pgfonlayer}
			\end{tikzpicture}
			
		};

    \node (C3) [v:ghost,position=0:0mm from D3] {
			
			\begin{tikzpicture}[scale=0.8]
			
			\pgfdeclarelayer{background}
			\pgfdeclarelayer{foreground}
			
			\pgfsetlayers{background,main,foreground}
			
			\begin{pgfonlayer}{main}
			
			\node (C) [] {};
			
			\node (v1) [v:main,position=90:11.5mm from C] {};
			\node (v2) [v:main,position=180:11.5mm from C] {};
			\node (v3) [v:main,position=270:11.5mm from C] {};
			\node (v4) [v:main,position=0:11.5mm from C] {};
   
			\draw (v1) [e:main,->,bend right=15] to (v2);
			\draw (v2) [e:main,->,bend left=15] to (v3);
			\draw (v3) [e:main,->,bend right=15] to (v4);
			\draw (v4) [e:main,->,bend left=15] to (v1);
			
			\draw (v2) [e:main,->,bend right=15] to (v1);
			\draw (v3) [e:main,->,bend left=15] to (v2);
			\draw (v4) [e:main,->,bend right=15] to (v3);
			\draw (v1) [e:main,->,bend left=15] to (v4);
			
			\end{pgfonlayer}
			
			\begin{pgfonlayer}{background}
			\end{pgfonlayer}	
			
			\begin{pgfonlayer}{foreground}
			\end{pgfonlayer}
			\end{tikzpicture}
		};
		
		\node (C4) [v:ghost,position=0:0mm from D4] {
			
			\begin{tikzpicture}[scale=0.8]
			
			\pgfdeclarelayer{background}
			\pgfdeclarelayer{foreground}
			
			\pgfsetlayers{background,main,foreground}
			
			\begin{pgfonlayer}{main}
			
			\node (C) [] {};
			
			\node (v1) [v:main,position=90:13mm from C] {};
			\node (v2) [v:main,position=162:13mm from C] {};
			\node (v3) [v:main,position=234:13mm from C] {};
			\node (v4) [v:main,position=306:13mm from C] {};
			\node (v5) [v:main,position=18:13mm from C] {};
			
			\draw (v1) [e:main,->,bend right=15] to (v2);
			\draw (v2) [e:main,->,bend right=15] to (v3);
			\draw (v3) [e:main,->,bend right=15] to (v4);
			\draw (v4) [e:main,->,bend right=15] to (v5);
			\draw (v5) [e:main,->,bend right=15] to (v1);
			
			\draw (v1) [e:main,->,bend right=15] to (v5);
			\draw (v2) [e:main,->,bend right=15] to (v1);
			\draw (v3) [e:main,->,bend right=15] to (v2);
			\draw (v4) [e:main,->,bend right=15] to (v3);
			\draw (v5) [e:main,->,bend right=15] to (v4);
			
			\end{pgfonlayer}
			
			\begin{pgfonlayer}{background}
			\end{pgfonlayer}	
			
			\begin{pgfonlayer}{foreground}
			\end{pgfonlayer}
			\end{tikzpicture}
			
		};
		
		\node (C5) [v:ghost,position=0:0mm from D5] {
			
			\begin{tikzpicture}[scale=0.8]
			
			\pgfdeclarelayer{background}
			\pgfdeclarelayer{foreground}
			
			\pgfsetlayers{background,main,foreground}
			
			\begin{pgfonlayer}{main}
			
			\node (C) [] {};
			
			\node (v1) [v:main,position=90:15mm from C] {};
			\node (v2) [v:main,position=150:15mm from C] {};
			\node (v3) [v:main,position=210:15mm from C] {};
			\node (v4) [v:main,position=270:15mm from C] {};
			\node (v5) [v:main,position=330:15mm from C] {};
			\node (v6) [v:main,position=30.1:15mm from C] {};
			
			\draw (v1) [e:main,->,bend right=15] to (v2);
			\draw (v2) [e:main,->,bend left=15] to (v3);
			\draw (v3) [e:main,->,bend right=15] to (v4);
			\draw (v4) [e:main,->,bend left=15] to (v5);
			\draw (v5) [e:main,->,bend right=15] to (v6);
			\draw (v6) [e:main,->,bend left=15] to (v1);
			
			\draw (v2) [e:main,->,bend right=15] to (v1);
			\draw (v3) [e:main,->,bend left=15] to (v2);
			\draw (v4) [e:main,->,bend right=15] to (v3);
			\draw (v5) [e:main,->,bend left=15] to (v4);
			\draw (v6) [e:main,->,bend right=15] to (v5);
			\draw (v1) [e:main,->,bend left=15] to (v6);
			
			\end{pgfonlayer}
			
			\begin{pgfonlayer}{background}
			\end{pgfonlayer}	
			
			\begin{pgfonlayer}{foreground}
			\end{pgfonlayer}
			\end{tikzpicture}
		};
		
	\end{tikzpicture}
    \caption{The \emph{base class} $\mathcal{B}_{2}$: The digraph $A_4$ (far left) and the family of bidirected cycles (everything else).}
    \label{fig:base_class}
\end{figure}
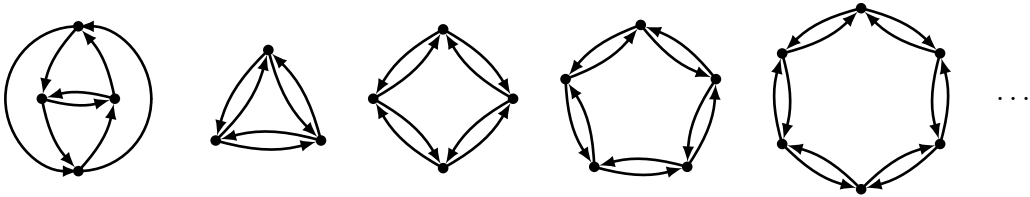

\begin{theorem}\label{thm:mainthm1}
    The class $\mathcal{D}_2$ of digraphs that can be generated from the graphs in $\mathcal{B}_2$ by sequences of basic, chain, collarette, and bracelet augmentations is exactly the class of all strongly $2$-connected digraphs.
    \end{theorem}

\begin{restatable}{theorem}{directedSplitterTheorem}
    \label{thm:mainthm2}
    If $D$ and $H$ are strongly $2$-connected digraphs and $H$ is a butterfly-minor of $D$ then there exists a sequence $D_0, D_1, \dots, D_{n-1}, D_n$ such that
\begin{itemize}
    \item $D_0 = H$, $D_n = D$,
    \item $D_i$ is a strongly 2-connected digraph for $0 \leq i \leq n$, and
    \item for every $i\in[n]$, $D_{i-1}$ is a butterfly-minor of $D_i$ and $D_i$ is obtained from $D_{i-1}$ by a basic, chain, collarette, or bracelet augmentation.
\end{itemize}
\end{restatable}

Wiederrecht~\cite{wiederrecht2020dtwone} proved that every strongly $2$-connected digraph contains a member of $\mathcal{B}_2$ as a butterfly-minor.

\begin{theorem}[\cite{wiederrecht2020dtwone}]
    \label{thm:base_class_minor}
    Every strongly 2-connected digraph contains a member of $\mathcal{B}_2$ as a butterfly-minor.
\end{theorem}

Additionally, one can observe that our four types of augmentations preserve strong $2$-connectivity.

\begin{restatable}{theorem}{augmentationAgainStronglyTwoConnected}
    \label{thm:all_constructed_graphs_are_strongly_2_connected}
    Let $D$ be a strongly $2$-connected digraph and let $D^{\star}$ be obtained from $D$ by a basic, chain, collarette, or bracelet augmentation.
    Then $D^{\star}$ is strongly $2$-connected.
\end{restatable}

This allows us to establish that \cref{thm:mainthm2} directly implies \cref{thm:mainthm1} as follows.

\begin{proof}[Proof of \cref{thm:mainthm1}]
    Let $D$ be a strongly $2$-connected digraph.
    By \cref{thm:base_class_minor}, it contains a digraph $H$ from $\mathcal{B}_2$ as butterfly-minor.
    Now \cref{thm:mainthm2} yields the desired construction sequence of $D$ from $H$ using only our four types of augmentations.
\end{proof}

\subsection{Splitter theorems}\label{subsec:intro_butterfly}

Seymour's theorem on $3$-connected graphs, as described above, is often called a \textsl{splitter theorem}.
That is because it can be used to \textsl{split} different graph classes or, in other words, to show that two graph classes are disjoint.
A striking example is discussed in \cite{mccuaig2001bracegeneration}.
Such splitter theorems do not only exist for the minor relation, or for graphs even.
In fact, Seymour's theorem on $3$-connected graphs was originally proven for $3$-connected regular matroids.

Another setting where such theorems exist is the vertex minor\footnote{A \emph{local complementation} of a vertex $v$ in a graph $G$ is the operation of replacing the graph $G[N(v)]$ with its complement in $G$. A graph $H$ is a \emph{vertex minor} of a graph $G$ if $H$ can be obtained from $G$ by a sequence of local complementations and vertex deletions.} relation.
Geelen and Oum \cite{GeelenO2009} proved that if $G$ is sufficiently large compared to $H$ and contains $H$ as a vertex minor, then there exists a vertex which can be removed in at least two of three possible ways while maintaining the presence of $H$ as a vertex minor.
Moreover, if $G$ and $H$ are both \emph{prime}\footnote{A notion of higher connectivity for vertex minors. A graph is prime if its vertex set cannot be partitioned into two sets, both of size at least two, such that the edge cut between the two sets induces a complete bipartite graph.} and $H$ is a proper vertex minor of $G$ there exists a sequence $G_0, G_1,\dots, G_n$ where $G_0=H$, $G_n=G$, and $G_{i-1}$ is a prime graph obtained from $G_i$ by either deleting a vertex or locally complementing and then deleting a vertex \cite{MR2694410}.

Finally, we have to mention the splitter theorems for \textsl{matching minors} by McCuaig \cite{mccuaig2001bracegeneration} for bipartite graphs and by Norin and Thomas \cite{NorinT2007} for non-bipartite graphs.
Two major results on the matching minor relation were its inception in Little's characterisation of bipartite Pfaffian graphs \cite{Little1975} and the \textsl{tight cut decomposition} of Lovász \cite{Lovasz1987}.
Moreover, a theorem of Seymour and Thomassen \cite{SeymourT1987}, which later turned out to be equivalent to Little's result, led to the discovery of the butterfly-minor relation.
On their way to solving the recognition problem for bipartite Pfaffian graphs, also known as \emph{P\'olya's Permanent Problem}, McCuaig proved a splitter theorem for \emph{braces}\footnote{Lovász's tight cut decomposition asserts that every graph with a perfect matching can be decomposed into a unique list of bipartite graphs, \emph{braces}, and non-bipartite graphs \emph{bricks}. Moreover, a bipartite graph $B$ with a perfect matching contains a brace $H$ as a matching minor if and only if a brace of $B$ contains $H$ as a matching minor.}.
It is worth pointing out that McCuaig claims a generation result for strongly $2$-connected digraphs similar to ours in their paper \cite{mccuaig2001bracegeneration}\footnote{page 127, reference [17] in the paper}.
However, the reference points to an unpublished manuscript, and despite all efforts, we were not able to reach McCuaig or find this manuscript.

\section{Preliminaries}
\label{sec:preliminaries}
In this paper, we work with finite and simple digraphs, that is, we do not allow loops or parallel edges.
Given two integers $a,b\in\Z$ we denote the set $\{z\in\Z \mid a\leq z\leq b\}$ by $[a,b].$
In case $a>b$, the set $[a,b]$ is empty.
For an integer $p\geq 1,$ we set $[p]\coloneqq [1,p]$.

Given two digraphs $D_1$ and $D_2$, we denote by $D_1\cap D_2$ the graph $(V(D_1)\cap V(D_2),E(D_1)\cap E(D_2))$ and by $D_1\cup D_2$ the digraph $(V(D_1)\cup V(D_2),E(D_1)\cup E(D_2))$.
For a vertex set $S\subseteq \V{D}$ of a digraph $D$ we denote by $D[S]$ the digraph $(S, E(D)\cap \{(u,v)\mid u,v\in S\})$ and by $D-S$ the digraph $D[\V{D}\setminus S]$.
In case $S=\{ v\}$ is a singleton, we drop the set notation and write $D-v$ instead of $D-\{ v\}$.
Finally, for digraphs $D$ and $H$ we define $D \graphminus H$ as $D - \V{H}$.
Similarly, for a set $F\subseteq \E{D}$ of edges we denote by $D-F$ the digraph $(\V{D},\E{D}\setminus F)$ and if $F=\{ e\}$ we write $D-e$ instead of $D-\{ e\}$.
If $(u,v)$ is a directed edge with $(u,v)\notin \E{D}$ but $u,v\in \V{D}$, we write $D+(u,v)$ for the digraph $D\cup (\{ u,v\},\{ (u,v)\})$.
We call this operation \emph{adding an edge}.

Let $D$ be a digraph.
For a vertex $v\in \V{D}$ we denote by $\OutNeighbours{D}{v}$ the \emph{out-neighbourhood} of $v$ in $D$, that is the set of vertices $x\in \V{D}$ such that $(v,x)\in \E{D}$.
Similarly, $\InNeighbours{D}{v}$ denotes the \emph{in-neighbourhood} of $v$ in $D$, that is the set of vertices $x\in \V{D}$ with $(x,v)\in \E{D}$.
In case the digraph $D$ is understood from the context, we drop the subscript and write $\OutNeighbours{}{v}$ and $\InNeighbours{}{v}$ respectively.

A \emph{(directed) path} $P$ of \emph{length $k$} in $D$ is a sequence of distinct vertices $v_1,\dots,v_{k+1}$ such that $(v_{i},v_{i+1}) \in E({D})$ for all $i\in[k]$.
The vertex $v_1$ is called the \emph{start-vertex} of $P$, denoted $\Start{P}$, and the vertex $v_{k+1}$ is the \emph{end-vertex} of $P$, denoted $\End{P}$.
We also refer to $P$ as a $\Start{P}$-$\End{P}$-path.
We often identify $P$ with the digraph $(\{v_1,\dots,v_{k+1}\},\{(v_i,v_{i+1})\mid i\in[k]\})$, and thus can use notation such as $\V{P}$ and $\E{P}$ to refer to the vertex and edge set of $P$ respectively.
Every directed path $P$ induces an order $<_{P}$ on $\V{P}$, that is, for $u,v \in \V{P}$ we write $u <_{P} v$ if $u$ occurs on the path $P$ before $v$. 

The subpath of $P$ starting in a vertex $a \in \V{P}$ and ending at a vertex $b \in \V{P}$ is denoted by $aPb$.
We also write $aP$ for $aP\End{P}$ and $Pa$ for $\Start{P}Pa$.
We say that two paths $P$ and $Q$ are \emph{internally disjoint} if $\V{P}\cap V(Q)\subseteq \{\Start{P},\End{P},\Start{Q},\End{Q} \}$.
Given two internally disjoint paths $P$ and $Q$ where the end-vertex of $P$ is the same as the start-vertex of $Q$, we denote the concatenation of $P$ and $Q$ (i.e.~the $\Start{P}$-$\End{Q}$-path $P\cup Q$) by $P\cdot Q$.
When concatenating two subpaths, we write $PaQ$ instead of $Pa \cdot aQ$.

The \emph{underlying undirected graph} of $D$ is the undirected graph $G$ with vertex set $\V{D}$ and edge set $\{ \{x, y\} \mid \{(x,y),(y,x)\}\cap \E{D}\neq\emptyset \}$.
The digraph $D$ is said to be \emph{weakly connected} if its underlying undirected graph is connected.
We say that $D$ is \emph{strongly connected} if for every pair of vertices $x,y\in \V{D}$ there exists a directed $x$-$y$-path and a directed $y$-$x$-path in $D$.
For a positive integer $k$, $D$ is \emph{strongly $k$-connected} if $|\V{D}|\geq k+1$ and $D-S$ is strongly connected for all sets $S\subseteq \V{D}$ with $|S|\leq k-1$.

\paragraph{Butterfly-minor models.}

A \emph{butterfly-model}, or just \emph{model}, of $H$ in $D$ is a function $\mu$ which assigns to every $x  \in \V{H} \cup \E{H}$ a subgraph $\mu(x) \subseteq D$ such that 
\begin{itemize}
    \item for $u \neq v \in \V{H}$, $\mu(u) \cap \mu(v) = \emptyset$,
    \item each $\mu(v)$ is the union of an in-branching and an out-branching which are disjoint except for their common root, which we refer to as the \emph{root} of $\mu(v)$, and
    \item for each $e = (s, t) \in \E{H}$, $\mu(e)$ is a path in $D$ starting at a vertex of the out-branching of $\mu(s)$ and ending at a vertex of the in-branching of $\mu(t)$ which is internally vertex disjoint from $\mu(v)$ and $\mu(e')$ for all $v \in \V{H}$ and $e' \in \E{H} \setminus \{ e\}$. 
\end{itemize}
  
We use $\mu(H)$ to denote the union of all $\mu(x)$, where $x\in \V{H}\cup \E{H}$.

These models yield a more constructive way of considering the butterfly-minor relation as a digraph $H$ is a butterfly-minor of $D$ if, and only if, it has a butterfly-model in $D.$

\begin{lemma}[{\cite{amiri2016erdos}}]
  Let $H, D$ be digraphs.
  $H$ is a butterfly-minor of $D$ if, and only if, $H$ has a butterfly-model in $D$.
\end{lemma}

\section{The directed splitter theorem}\label{sec:formalintro}
In this section, we present the formal definitions necessary to state our main results rigorously and provide an overview of our proof strategy.

Let $H$ and $D$ be digraphs.
If $H$ is a butterfly-minor of $D$ then there exists some minimal subgraph $H'$ such that $H$ can be obtained from $H'$ by only performing butterfly-contractions.
Our strategy towards proving \cref{thm:mainthm1,thm:mainthm2} is to fix such a subgraph $H'$ and then manipulate it to find a larger digraph $H''$ as a butterfly-minor.
To do this, we need the following notions.

\paragraph{In- and out-trees and stars.}

Let $H$ and $D$ be digraphs and $\mu$ be a butterfly-model of $H$ in $D$, then we refer to the subgraph $\mu(H)$ as an \emph{$H$-expansion}.
Now fix a vertex $v\in \V{H}$ and consider the digraph $\mu(v)$.
Let $\intree{v}$ be the minimal in-arborescence in $\mu(v)$ containing all vertices of $\mu(v)$ that have in-degree at least two with respect to $\mu(H)$.
Let $\outtree{v}$ be the minimal out-arborescence in $\mu(v)$ containing all vertices of $\mu(v)$ that have out-degree at least two with respect to $\mu(H)$. %, see \cref{fig:trees}.
Next, let $\instar{v}$ be the maximal in-arborescence containing $\intree{v}$ such that all vertices but its root have out-degree one with respect to $\mu(H)$.
Similarly, let $\outstar{v}$ be the maximal out-arborescence containing $\outtree{v}$ such that all vertices but its root have in-degree one with respect to $\mu(H)$.
See \cref{fig:stars_and_trees} for an illustration of these definitions.

Note that for every $(u,v) \in \E{H}$, if the subgraph $\outstar{u} \cap \instar{v}$ is non-empty, then it is a directed path.
For every edge $(u,v) \in \E{H}$ we call the path induced by the edge-set
\begin{equation*}
    \Set{e \in \E{\mu(H)}~|~\Head{e}\in \V{\instar{v}},\Tail{e} \in \V{\outstar{u}}}
\end{equation*}
the \emph{$u$-$v$-bridge of $\mu(H)$}, or simply \emph{$u$-$v$-bridge} if the expansion is clear from the context.
We remark that digraphs of the form $\outstar{u} \cup \instar{v} \cup \Set{e~|~\Head{e}\in \V{\instar{v}},\Tail{e}\in \V{\outstar{u}}}$ consist precisely of $\outstar{u} \cup \instar{v}$ and the $u$-$v$-bridge.

Also, we call the directed path $\outstar{v} \cap \instar{v}$ the \emph{\redpath} of $v$ in $\mu$ and denote it by $\rpath{v}$.
Notice that $\rpath{v}$ contains the root of $\mu(v)$.

The \emph{branchset} of a vertex $v \in \V{H}$ is defined as $\outtree{v}\cup\intree{v}\cup\rpath{v}$.
Please note that the objects $\outstar{v}$, $\instar{v}$, $\rpath{v}$, $\intree{v}$, and $\outtree{v}$ are defined with respect to a fixed butterfly-model $\mu$.
Usually, the $H$-expansion $\mu(H)$ can be understood from the context.
In case ambiguity arises, we make this explicit by replacing $v$ with $v,\mu(H)$ in the subscript, i.e.~$\outstar{v,\mu(H)}$.

Let us proceed by formally introducing our four augmentations which were sketched in \cref{subsec:ourresult}.
Before we can do so, we need two preliminary operations.

\paragraph{Splits.}

Let $D$ be a digraph and $v$ a vertex in it with $\OutNeighbours{D}{v} = N_b \disjointUnion N_e$ such that $|N_b| \geq 1$ and $\Abs{N_e} \geq 2$.
We obtain a digraph $D'$ from $D$ by removing $v$ and replacing it with two vertices $\base{v}$, the \emph{base}, and $\exposed{v}$, the \emph{exposed vertex}, with the addition of the following edges.
For all $x \in \InNeighbours{D}{v}$ we add the edge $(x,\base{v})$, for all vertices $x \in N_b$ we add the edge $(\base{v},x)$, for all vertices $x \in N_e$ we add the edge $(\exposed{v},x)$, and, finally, we add the single edge $(\base{v},\exposed{v})$.
The digraph $D'$ is said to be obtained by an \emph{out-split} at $v$, see \cref{fig:out-split} for an illustration.

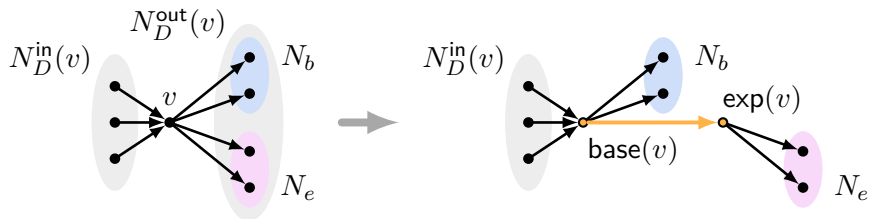
\begin{figure}[!ht]
    \centering
    \begin{tikzpicture}
        
		\node (D1) [v:ghost] {};
        \node (D2) [v:ghost,position=0:50mm from D1] {};
		
        \node (C1) [v:ghost,position=0:0mm from D1] {
			
			\begin{tikzpicture}[scale=0.8]
			
			\pgfdeclarelayer{background}
			\pgfdeclarelayer{foreground}
			
			\pgfsetlayers{background,main,foreground}
			
			\begin{pgfonlayer}{main}
			
                \node (C) [] {};
                \node (L) [position=180:28mm from C] {};
                \node (R) [position=0:40mm from C] {};
                \node (T_L) [position=180:5mm from C] {};
                \node (T_R) [position=0:5mm from C] {};

                \node (label) [position=90:5mm from C] {};
                \draw [e:main,->,line width=3pt,color=Gray] (T_L) to (T_R);
    
                \node (v) [v:main,position=180:5mm from L] {};
                \node (v_label) [v:ghost,position=90:4mm from v] {$v$};

                \node (in_b_L) [v:main,position=180:9mm from v] {};
                \node (in_a_L) [v:main,position=90:6mm from in_b_L] {};
                \node (in_c_L) [v:main,position=270:6mm from in_b_L] {};
                \node (in_label_L) [v:ghost,position=135:15mm from in_b_L] {$\InNeighbours{D}{v}$};

                \node (out_b_1_L) [v:main,position=20:14mm from v] {};
                \node (out_b_2_L) [v:main,position=90:6mm from out_b_1_L] {};
                \node (out_e_1_L) [v:main,position=340:14mm from v] {};
                \node (out_e_2_L) [v:main,position=270:6mm from out_e_1_L] {};
                \node (out_label_L) [v:ghost,position=155:13mm from out_b_2_L] {$\OutNeighbours{D}{v}$};
                \node (Nb_L) [v:ghost,position=0:8mm from out_b_2_L] {$N_b$};
                \node (Ne_L) [v:ghost,position=0:8mm from out_e_2_L] {$N_e$};

                \draw [e:main,->] (in_a_L) to (v);
                \draw [e:main,->] (in_b_L) to (v);
                \draw [e:main,->] (in_c_L) to (v);

                \draw [e:main,->] (v) to (out_b_1_L);
                \draw [e:main,->] (v) to (out_b_2_L);
                \draw [e:main,->] (v) to (out_e_1_L);
                \draw [e:main,->] (v) to (out_e_2_L);

                \node (v_base) [v:main,position=180:5mm from R,fill=PastelOrange] {};
                \node (v_base_label) [v:ghost,position=330:9.5mm from v_base] {$\base{v}$};
                \node (v_exposed) [v:main,position=0:23mm from v_base,fill=PastelOrange] {};
                \node (v_exposed_label) [v:ghost,position=30:7.5mm from v_exposed] {$\exposed{v}$};

                \node (in_b_R) [v:main,position=180:9mm from v_base] {};
                \node (in_a_R) [v:main,position=90:6mm from in_b_R] {};
                \node (in_c_R) [v:main,position=270:6mm from in_b_R] {};
                \node (in_label_R) [v:ghost,position=135:15mm from in_b_R] {$\InNeighbours{D}{v}$};

                \node (out_b_1_R) [v:main,position=20:14mm from v_base] {};
                \node (out_b_2_R) [v:main,position=90:6mm from out_b_1_R] {};
                \node (out_e_1_R) [v:main,position=340:14mm from v_exposed] {};
                \node (out_e_2_R) [v:main,position=270:6mm from out_e_1_R] {};
                \node (Nb_R) [v:ghost,position=0:8mm from out_b_2_R] {$N_b$};
                \node (Ne_R) [v:ghost,position=0:8mm from out_e_2_R] {$N_e$};

                \draw [e:main,->] (in_a_R) to (v_base);
                \draw [e:main,->] (in_b_R) to (v_base);
                \draw [e:main,->] (in_c_R) to (v_base);

                \draw [e:main,->] (v_base) to (out_b_1_R);
                \draw [e:main,->] (v_base) to (out_b_2_R);
                \draw [e:main,->] (v_exposed) to (out_e_1_R);
                \draw [e:main,->] (v_exposed) to (out_e_2_R);

                \draw [e:main,->,color=PastelOrange,line width=1.5pt] (v_base) to (v_exposed);

            \end{pgfonlayer}
			
			\begin{pgfonlayer}{background}

                \node [ellipse, fill=Gray, draw=Gray,minimum width=1.8cm,minimum height=0.6cm,color=Gray,opacity=0.2,rotate=90,position=0:0mm from in_b_L] {};
                \node [ellipse, fill=Gray, draw=Gray,minimum width=2.6cm,minimum height=0.9cm,color=Gray,opacity=0.2,rotate=90,position=0:13mm from v] {};
                \node [ellipse, fill=white, draw=white,minimum width=10mm,minimum height=5mm,color=white,rotate=90,position=90:3mm from out_b_1_L] {};
                \node [ellipse, fill=white, draw=white,minimum width=10mm,minimum height=5mm,color=white,rotate=90,position=270:3mm from out_e_1_L] {};
                \node [ellipse, fill=CornflowerBlue, draw=CornflowerBlue,minimum width=10mm,minimum height=5mm,color=CornflowerBlue,opacity=0.3,rotate=90,position=90:3mm from out_b_1_L] {};
                \node [ellipse, fill=LavenderMagenta, draw=LavenderMagenta,minimum width=10mm,minimum height=5mm,color=LavenderMagenta,opacity=0.3,rotate=90,position=270:3mm from out_e_1_L] {};
                
                \node [ellipse, fill=Gray, draw=Gray,minimum width=1.8cm,minimum height=0.6cm,color=Gray,opacity=0.2,rotate=90,position=0:0mm from in_b_R] {};
                \node [ellipse, fill=CornflowerBlue, draw=CornflowerBlue,minimum width=10mm,minimum height=5mm,color=CornflowerBlue,opacity=0.3,rotate=90,position=90:3mm from out_b_1_R] {};
                \node [ellipse, fill=LavenderMagenta, draw=LavenderMagenta,minimum width=10mm,minimum height=5mm,color=LavenderMagenta,opacity=0.3,rotate=90,position=270:3mm from out_e_1_R] {};

			\end{pgfonlayer}

			\begin{pgfonlayer}{foreground}
			\end{pgfonlayer}
			\end{tikzpicture}
			
		};

    \end{tikzpicture}
    \caption{An out-split at the vertex $v$ with $\OutNeighbours{D}{v} = N_b \disjointUnion N_e$.
        We require $|N_e| \geq 2$ and $|N_b| \geq 1$.
        Notice that the edge $(\base{v},\exposed{v})$ is butterfly-contractible.}
    \label{fig:out-split}
\end{figure}

Similarly, let $D$ be a digraph and $v$ a vertex in it with $\InNeighbours{D}{v} = N_b \disjointUnion N_e$ such that $|N_b| \geq 1$ and $|N_e| \geq 2$.
We obtain a digraph $D'$ from $D$ by removing $v$ and replacing it with two vertices $\base{v}$, the \emph{base}, and $\exposed{v}$, the \emph{exposed vertex}, with the addition of the following edges.
For all $x \in \OutNeighbours{D}{v}$ we add the edge $(\base{v},x)$, for all vertices $x \in N_b$ we add the edge $(x,\base{v})$, for all vertices $x \in N_e$ we add the edge $(x,\exposed{v})$, and, finally, we add the single edge $(\exposed{v},\base{v})$.
The digraph $D'$ is said to be obtained by an \emph{in-split} at $v$.

Note that if the original digraph $D$ is strongly 2-connected, then after an out- or in-split the exposed vertex is the only vertex in the obtained digraph $D'$ that does not have at least two out-neighbours and at least two in-neighbours.
Moreover, the unique edge joining the base and the exposed vertex of $v$ in $D'$ is either the only incoming edge or the only outgoing edge of the exposed vertex and thus butterfly-contractible.
Contracting this edge restores the digraph $D$.
Hence, $D$ is a butterfly-minor of every digraph $D'$ obtained from $D$ via out- or in-splitting a vertex.

\paragraph{Chains.}

Let $D$ be a digraph and $e = (u,v) \in \E{D}$.
Let $Q = (u=v_1, v_2, \dots, v_l=v)$ be a directed $u$-$v$-path of length at least $2$ which is internally disjoint from $D$.
Now let $x$ and $y$ be two vertices, not necessarily distinct, of $D$ and $P=(x,v_{l-1},\dots,v_2,y)$.
We say that the digraph $D' \coloneqq \Brace{D - e} \cup P \cup Q$ is obtained by \emph{adding an $x$-$y$-chain through $(u,v)$}. 
We then also say that the paths $P$ and $Q$ form a \emph{chain} in $D'$.
The \emph{length} of a chain is the number of internal vertices in $Q$.
See \cref{fig:adding_a_chain} for an illustration.

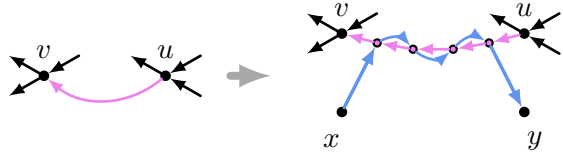
\begin{figure}[!ht]
    \centering
    \begin{tikzpicture}[scale=0.8]
			
        \pgfdeclarelayer{background}
        \pgfdeclarelayer{foreground}
        
        \pgfsetlayers{background,main,foreground}
        
        \begin{pgfonlayer}{main}
        
            \node (C) [] {};
            \node (L) [position=180:24mm from C] {};
            \node (RP) [position=0:30mm from C] {}; 
            \node (R) [position=90:7mm from RP] {};
            \node (T_L) [position=180:5mm from C] {};
            \node (T_R) [position=0:5mm from C] {};

            \node (u_L) [v:main,position=180:10mm from L] {};
            \node (v_L) [v:main,position=0:10mm from L] {};

            \node (u_label_L) [v:ghost] [position=90:4mm from u_L] {$v$};
            \node (v_label_L) [v:ghost] [position=90:4mm from v_L] {$u$};

            \node (u_out_1_L) [v:ghost,position=150:7mm from u_L] {};
            \node (u_out_2_L) [v:ghost,position=210:7mm from u_L] {};
            \node (u_in_1_L) [v:ghost,position=30:7mm from u_L] {};
            
            \node (v_out_1_L) [v:ghost,position=150:7mm from v_L] {};
            \node (v_in_1_L) [v:ghost,position=30:7mm from v_L] {};            
            \node (v_in_2_L) [v:ghost,position=330:7mm from v_L] {};

            \draw [e:main,->] (u_L) to (u_out_1_L);
            \draw [e:main,->] (u_L) to (u_out_2_L);
            \draw [e:main,->] (u_in_1_L) to (u_L);
            
            \draw [e:main,->] (v_L) to (v_out_1_L);
            \draw [e:main,->] (v_in_1_L) to (v_L);
            \draw [e:main,->] (v_in_2_L) to (v_L);

            \draw [e:main,->,bend left=40,color=LavenderMagenta] (v_L) to (u_L);
            
            \draw [e:main,->,line width=3pt,color=Gray] (T_L) to (T_R);

            \node (u_R) [v:main,position=180:15mm from R] {};
            \node (v_R) [v:main,position=0:15mm from R] {};

            \node (u_label_R) [v:ghost] [position=90:4mm from u_R] {$v$};
            \node (v_label_R) [v:ghost] [position=90:4mm from v_R] {$u$};

            \node (uP_R) [v:main,position=270:13mm from u_R] {};
            \node (vP_R) [v:main,position=270:13mm from v_R] {};

            \node (x_1) [v:main,position=345:6mm from u_R, fill=LavenderMagenta,scale=0.85] {};
            \node (x_2) [v:main,position=350:6mm from x_1, fill=LavenderMagenta,scale=0.85] {};
            \node (x_4) [v:main,position=195:6mm from v_R, fill=LavenderMagenta,scale=0.85] {};
            \node (x_3) [v:main,position=190:6mm from x_4,, fill=LavenderMagenta,scale=0.85] {};

            \node (u_label_R) [v:ghost] [position=250:5mm from uP_R] {$x$};
            \node (v_label_R) [v:ghost] [position=290:5mm from vP_R] {$y$};

            \node (u_out_1_R) [v:ghost,position=150:7mm from u_R] {};
            \node (u_out_2_R) [v:ghost,position=210:7mm from u_R] {};
            \node (u_in_1_R) [v:ghost,position=30:7mm from u_R] {};
            
            \node (v_out_1_R) [v:ghost,position=150:7mm from v_R] {};
            \node (v_in_1_R) [v:ghost,position=30:7mm from v_R] {};            
            \node (v_in_2_R) [v:ghost,position=330:7mm from v_R] {};

            \draw [e:main,->] (u_R) to (u_out_1_R);
            \draw [e:main,->] (u_R) to (u_out_2_R);
            \draw [e:main,->] (u_in_1_R) to (u_R);
            
            \draw [e:main,->] (v_R) to (v_out_1_R);
            \draw [e:main,->] (v_in_1_R) to (v_R);
            \draw [e:main,->] (v_in_2_R) to (v_R);

            \draw [e:main,->,color=LavenderMagenta] (v_R) to (x_4);
            \draw [e:main,->,color=LavenderMagenta] (x_4) to (x_3);
            \draw [e:main,->,color=LavenderMagenta] (x_3) to (x_2);
            \draw [e:main,->,color=LavenderMagenta] (x_2) to (x_1);
            \draw [e:main,->,color=LavenderMagenta] (x_1) to (u_R);

            \draw [e:main,->,color=CornflowerBlue,line width=1.3pt] (uP_R) to (x_1);
            \draw [e:main,->,color=CornflowerBlue,bend left=40] (x_1) to (x_2);
            \draw [e:main,->,color=CornflowerBlue,bend right=40] (x_2) to (x_3);
            \draw [e:main,->,color=CornflowerBlue,bend left=40] (x_3) to (x_4);
            \draw [e:main,->,color=CornflowerBlue,line width=1.3pt] (x_4) to (vP_R);

        \end{pgfonlayer}
        
        \begin{pgfonlayer}{background}
        \end{pgfonlayer}

        \begin{pgfonlayer}{foreground}
        \end{pgfonlayer}
        \end{tikzpicture}
    \caption{Adding an $x$-$y$-chain of length four through the edge $(u,v)$.}
    \label{fig:adding_a_chain}
\end{figure}

The operation of adding an $x$-$y$-chain through $(u,v)$ can be described as ``weaving'' a path through the edge $(u,v)$.
This operation allows us to replace a single edge with a bidirected path of arbitrary length.
Notice that removing the edges of $P$
leaves a digraph that is obtained by subdividing the edge $(u,v)$.
It follows that any digraph $D'$ that is obtained from $D$ by adding an $x$-$y$-chain through $(u,v)$ contains $D$ as a butterfly-minor.

\subsection{The four types of augmentations}\label{subsec:augmentations}

We are now ready to describe our four types of augmentations.
An example for each of the four types of augmentations can be found in \cref{fig:augmentations1}.

Let $H, H', D$ be digraphs.
For $\lambda \in \{\text{basic, chain, collarette, bracelet}\}$, $H'$ is a \emph{$\lambda$~$H$-augmentation}, for short an \emph{$H$-augmentation}, if it is obtained from $H$ by performing a $\lambda$~augmentation as defined below.
We say that $D$ \emph{admits} a $\lambda$ $H$-augmentation if $D$ contains a butterfly-minor that is a $\lambda$ $H$-augmentation.

Let $D$ be a digraph and let $u, v$ be distinct vertices of $D$.
We define a digraph $D'$ obtained from $D$ as follows.
\begin{itemize}
    \item setting $D'\coloneqq D$, or
    \item performing an out-split at $u$, or
    \item performing an in-split at $v$, or
    \item performing an out-split at $u$ and an in-split at $v$.
\end{itemize}
For both $x \in \{u,v\}$ we define
\begin{align*}
    \mathsf{b}(x) &{}\coloneqq \begin{cases}  \base{x} & \text{in case we performed a split at $x$}\\
    x & \text{otherwise, and}
    \end{cases}\\
    \mathsf{e}(x) &{}\coloneqq \begin{cases}  \exposed{x} & \text{in case we performed a split at $x$}\\
    x & \text{otherwise.}
    \end{cases}
\end{align*}

\paragraph{Basic augmentation:}
We say that a digraph $D^*$ is obtained from $D$ by a \emph{basic augmentation} if either
\begin{itemize}
    \item $D^*\coloneqq D' + (\mathsf{e}(v), \mathsf{e}(u))$ for $(\mathsf{e}(v), \mathsf{e}(u)) \notin E(D')$, or
    \item $D^*$ is obtained from $D$ by performing firstly an out-split at some vertex $w \in V(D)$, secondly an in-split at $\base{w}$ and finally adding the edge $(\exposed{\base{w}}, \exposed{w})$.
\end{itemize}

\paragraph{Chain augmentation:}
Let $e$ be an edge of $D'$ with $\Tail{e} \in \{\mathsf{b}(u), \mathsf{e}(u)\}$ and $\Head{e} \in \{\mathsf{b}(v), \mathsf{e}(v)\}$.
We say that a digraph $D^*$ obtained from $D'$ by adding an $\mathsf{e}(v)$-$\mathsf{e}(u)$-chain through the edge $e$ is obtained from $D$ by a \emph{chain augmentation}.

\paragraph{Collarette augmentation:}
Let $u$ and $v$ be distinct vertices of $D$ such that $(v,u),(u,v)\in \E{D}$.
Let $D''$ be the digraph obtained from $D$ by subdividing the edge $(v,u)$ with a vertex $x$ and subdividing the edge $(u,v)$ with a vertex $y$.
We say that a digraph $D^*$ is obtained from $D$ by a \emph{collarette augmentation} if it is obtained from $D''$ by first adding the edge $(x,y)$ and then a $y$-$x$-chain through $(x,y)$.

\paragraph{Bracelet augmentation:}
Finally, let $w\in \V{D}$ be a vertex with $\InNeighbours{}{w}=\{x,a\}$ and $\OutNeighbours{}{w}=\{ y,b\}$ such that
\begin{itemize}
    \item either $a=y$ or $(x,a), (a,y) \in \E{H}$, and
    \item either $x=b$ or $(x,b), (b,y) \in \E{H}$.
\end{itemize}
We say that a digraph $D^*$ is obtained from $D$ by a \emph{bracelet augmentation} if either
\begin{enumerate}
    \item $D^*$ is obtained from $D$ by subdividing the edge $(a,w)$ with a vertex $z$ and adding the edges $(x,z)$ and $(z,y)$, or
    \item $D^*$ is obtained from $D$ by subdividing the edge $(w,b)$ with a vertex $z$ and adding the edges $(x,z)$ and $(z,y)$.
\end{enumerate}

\subsection{A note on the types of augmentations}\label{subsec:augment_discussion}

One of the four types of augmentations appearing in our main theorems allows the introduction of an unbounded number of vertices in a single step.
This is in stark contrast to similar generation theorems based on minor-like containment relations and appears to be more reminiscent of the subgraph relation that plays a role in Whitney's theorem on ear decompositions~\cite{Whitney1932}.
This is due to two reasons in particular.
The first one is rooted in the following simple observation:
A graph $G$ is $2$-connected if and only if $\Bidirected{G}{}$ is strongly $2$-connected.
Thus, \cref{thm:mainthm1,thm:mainthm2} must necessarily contain Whitney's result on ear-decompositions for $2$-connected undirected graphs.
To see that this is indeed true observe that the addition of an undirected $H$-ear $P$ can be emulated in two steps.
Let $x$ and $y$ be the endpoints of $P$.
\begin{enumerate}
    \item If $\{x, y\}\in \E{H}$, then first use a chain augmentation to ``weave'' a directed $x$-$y$-path of length $|P|-1$ through the edge $(y,x) \in E(\Bidirected{H}{})$ and then add the edge $(y, x)$ back in using a basic augmentation.
    \item If $\{x, y\} \not\in \E{H}$, then first add the edge $(y, x)$ using a basic augmentation and then ``weave'' in a directed $x$-$y$-path of length $|P|-1$ through the edge $(y,x)$.
\end{enumerate}
The second reason can be found in the existence of infinite anti-chains in the butterfly-minor relation.

\paragraph{Butterfly-minors and infinite anti-chains.}

It is a well-established fact that the butterfly-minor relation is not a well-quasi-order on the family of digraphs.
In fact, there even exist infinite anti-chains purely consisting of strongly $2$-connected digraphs.
For this reason, it appears unlikely that a generation theorem for strongly $2$-connected digraphs exists that is compatible with the butterfly-minor relation and only uses a finite number of local operations.
However, it appears that the difficulties arising from the existence of strongly $2$-connected butterfly-minor anti-chains for a generation theorem boil down to a single, well-describable phenomenon.
Let $P$ be a path on at least three vertices.
Observe that deleting any edge from $\Bidirected{P}{}$ results in a digraph that is not strongly connected.
However, the only way to create a butterfly-contractible edge in $\Bidirected{P}{}$ is by deleting some edge.
Let $Q$ be another path with $|\E{P}|-|E(Q)|=1$, and for each $R\in\{ P, Q\}$ let $H_R$ be the graph obtained from $R$ by introducing, for each endpoint $x$ of $R$, a four-cycle that has exactly the vertex $x$ in common with the rest of the graph.
As a result, $\Bidirected{H}{P}$ does not contain $\Bidirected{H}{Q}$ as a butterfly-minor, while $H_Q$ can be obtained from $H_P$ by contracting a single edge of $P$.
See \cref{fig:antichain} for an illustration.

\begin{figure}[!ht]
    \centering
    \begin{tikzpicture}
		
		\node (D1) [v:ghost] {};
        \node (D2) [v:ghost,position=0:50mm from D1] {};
		
        \node (C1) [v:ghost,position=0:0mm from D1] {
			
			\begin{tikzpicture}[scale=0.8]
			
			\pgfdeclarelayer{background}
			\pgfdeclarelayer{foreground}
			
			\pgfsetlayers{background,main,foreground}
			
			\begin{pgfonlayer}{main}
			
			\node (C) [] {};

            \node (a) [v:main] {};
            \node (b) [v:main,position=0:9mm from a] {};
            \node (d) [v:main,position=0:9mm from b] {};
            
            \node(x_1) [v:main, position=135:9mm from a] {};
            \node(y_1) [v:main, position=225:9mm from x_1] {};
            \node(z_1) [v:main, position=315:9mm from y_1] {};

            \node(x_2) [v:main, position=315:9mm from d] {};
            \node(y_2) [v:main, position=45:9mm from x_2] {};
            \node(z_2) [v:main, position=135:9mm from y_2] {};

            \draw[e:main,->,bend left=25] (a) to (b);
            \draw[e:main,->,bend left=25] (b) to (a);
            \draw[e:main,->,bend right=25] (b) to (d);
            \draw[e:main,->,bend right=25] (d) to (b);

            \draw[e:main,->,bend right=25] (a) to (x_1);
            \draw[e:main,->,bend left=25] (x_1) to (y_1);
            \draw[e:main,->,bend right=25] (y_1) to (z_1);
            \draw[e:main,->,bend left=25] (z_1) to (a);

            \draw[e:main,->,bend right=25] (x_1) to (a);
            \draw[e:main,->,bend left=25] (y_1) to (x_1);
            \draw[e:main,->,bend right=25] (z_1) to (y_1);
            \draw[e:main,->,bend left=25] (a) to (z_1);

            \draw[e:main,->,bend right=25] (d) to (x_2);
            \draw[e:main,->,bend left=25] (x_2) to (y_2);
            \draw[e:main,->,bend right=25] (y_2) to (z_2);
            \draw[e:main,->,bend left=25] (z_2) to (d);

            \draw[e:main,->,bend right=25] (x_2) to (d);
            \draw[e:main,->,bend left=25] (y_2) to (x_2);
            \draw[e:main,->,bend right=25] (z_2) to (y_2);
            \draw[e:main,->,bend left=25] (d) to (z_2);

            \end{pgfonlayer}
			
			\begin{pgfonlayer}{background}
			\end{pgfonlayer}
   
			\begin{pgfonlayer}{foreground}
			\end{pgfonlayer}
			\end{tikzpicture}
			
		};
  
		\node (C2) [v:ghost,position=0:0mm from D2] {
			
			\begin{tikzpicture}[scale=0.8]
			
			\pgfdeclarelayer{background}
			\pgfdeclarelayer{foreground}
			
			\pgfsetlayers{background,main,foreground}
			
			\begin{pgfonlayer}{main}
			
			\node (C) [] {};

            \node (a) [v:main] {};
            \node (b) [v:main,position=0:9mm from a] {};
            \node (c) [v:main,position=0:9mm from b] {};
            \node (d) [v:main,position=0:9mm from c] {};
            
            \node(x_1) [v:main, position=135:9mm from a] {};
            \node(y_1) [v:main, position=225:9mm from x_1] {};
            \node(z_1) [v:main, position=315:9mm from y_1] {};

            \node(x_2) [v:main, position=315:9mm from d] {};
            \node(y_2) [v:main, position=45:9mm from x_2] {};
            \node(z_2) [v:main, position=135:9mm from y_2] {};

            \draw[e:main,->,bend left=25] (a) to (b);
            \draw[e:main,->,bend left=25] (b) to (a);
            \draw[e:main,->,bend right=25] (b) to (c);
            \draw[e:main,->,bend right=25] (c) to (b);
            \draw[e:main,->,bend left=25] (c) to (d);
            \draw[e:main,->,bend left=25] (d) to (c);

            \draw[e:main,->,bend right=25] (a) to (x_1);
            \draw[e:main,->,bend left=25] (x_1) to (y_1);
            \draw[e:main,->,bend right=25] (y_1) to (z_1);
            \draw[e:main,->,bend left=25] (z_1) to (a);

            \draw[e:main,->,bend right=25] (x_1) to (a);
            \draw[e:main,->,bend left=25] (y_1) to (x_1);
            \draw[e:main,->,bend right=25] (z_1) to (y_1);
            \draw[e:main,->,bend left=25] (a) to (z_1);

            \draw[e:main,->,bend right=25] (d) to (x_2);
            \draw[e:main,->,bend left=25] (x_2) to (y_2);
            \draw[e:main,->,bend right=25] (y_2) to (z_2);
            \draw[e:main,->,bend left=25] (z_2) to (d);

            \draw[e:main,->,bend right=25] (x_2) to (d);
            \draw[e:main,->,bend left=25] (y_2) to (x_2);
            \draw[e:main,->,bend right=25] (z_2) to (y_2);
            \draw[e:main,->,bend left=25] (d) to (z_2);

            \end{pgfonlayer}
			
			\begin{pgfonlayer}{background}
			\end{pgfonlayer}
   
			\begin{pgfonlayer}{foreground}
			\end{pgfonlayer}
			\end{tikzpicture}
			
		};

	\end{tikzpicture}
    \caption{Two members of an infinite anti-chain of the butterfly-minor relation: the digraph $\Bidirected{H}{Q}$ on the left and the digraph $\Bidirected{H}{P}$ on the right.
    While $H_P$ contains $H_Q$ as a minor, $\Bidirected{H}{Q}$ and $\Bidirected{H}{P}$ are incomparable under the butterfly-minor relation.}
    \label{fig:antichain}
\end{figure}
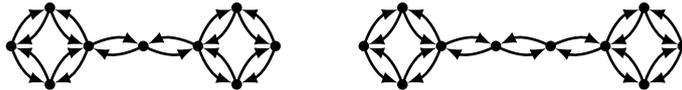
Notice that our base family $\mathcal{B}_2$ is also an infinite anti-chain of the butterfly-minor relation.
We mention the example above because it illustrates one of the major deficiencies of the butterfly-minor relation:
butterfly-minors are unable to reduce the length of ``bidirected paths'', that is, digraphs which are obtained from an undirected path by replacing every edge with a digon while maintaining structures connected to the endpoints of such a path.
This is precisely the reason why our four types of augmentations must contain the operation of ``weaving'' a directed path through another edge.

In the presence of \cref{thm:mainthm1,thm:mainthm2}, it appears that these ``bidirected paths'' could be the only obstruction to digraphs being well-quasi-ordered by butterfly-minors.
A result by Muzi~\cite{muzi2017paths}
shows that every butterfly-minor-closed class of digraphs excluding an \emph{antidirected path}\footnote{A digraph is an antidirected path if its underlying undirected graph is a path and, by reversing the direction of every other edge one obtains a directed path.} is well-quasi-ordered under the butterfly-minor relation.
The family of antidirected paths is a well-known infinite anti-chain for the butterfly-minor relation consisting exclusively of acyclic digraphs.
Notice that every bidirected path contains two edge-disjoint antidirected paths as subgraphs.

\subsection{Augmentations preserve strong \texorpdfstring{$2$}{2}-connectivity}
\label{sec:construction_str_2_conn}

Before we delve into the more involved part of the proof for our main theorem, we establish that the four types of augmentations defined above indeed preserve strong $2$-connectivity.

\augmentationAgainStronglyTwoConnected*

\begin{proof}
    \setcounter{claimcounter}{0}

    We have that $D^{*}$ is a $\lambda$~$D$-augmentation where $\lambda \in \Set{\text{basic, chain, collarette, bracelet}}$.
    
    First, we consider the very special case that $\Abs{V(D)} = 3$, i.e.~$D$ is isomorphic to $\Bidirected{K}{3}$.
    Note that no split can be performed on $D$ since any vertex of $D$ has in- and out-degree $2$.
    Since $D$ is complete, no basic augmentation can be applied to $D$.
    If $\lambda = \text{chain}$, $\mathsf{e}(u) = \mathsf{b}(u)$ and $\mathsf{e}(v) = \mathsf{b}(v)$ and thus $D^*$ is strongly $2$-connected.
    Similarly, if $\lambda \in \Set{\text{collarette, bracelet}}$, $D^*$ is strongly $2$-connected.

    Thus, we can assume that $\Abs{V(D)} \geq 4$.
    We have to show that for arbitrary distinct vertices $x, y, z \in V(D^*)$, there is an $x$-$y$-path in $D^{*} - z$.
    
    Depending on the $D$-augmentation performed,
    we define a set of vertices $W \subseteq V(D^{*})$ which
    are ``related'' to the augmentation.
    We ensure that for every vertex $w \in W$
    there exist two $w$-$(V(D^{*}) \setminus W)$-paths that are disjoint apart from $w$ as well as two $(V(D^{*}) \setminus W)$-$w$-paths that are disjoint apart from $w$.
    Further, we show that:
    \begin{claim}\label{claim:strong_2conn} Let $x, y \in V(D^{*}) \setminus (W \cup \{z\})$. Then, there exists an $x$-$y$-path in $D^{*} - z$.
    \end{claim}
    
    Assuming the stated property of $W$ and \cref{claim:strong_2conn}, we can show how to find an $x$-$y$ path in the case that at least one of $x$ or $y$ is in $W$.
    Without loss of generality, let $x \in W$.
    By the previous observation for the vertices in $W$, there exist two $x$-$(V(D^*) \setminus W)$-paths that are disjoint apart from $x$ and therefore there is an $x$-$(V(D^*) \setminus W)$-path $P$ in $D^* - z$.
    Similarly, if $y \in W$, then there exists a $(V(D^*) \setminus W)$-$y$-path $Q$ in $D^* - z$.
    If $y \not\in W$, we set $Q \coloneqq  \{ y \}$.
    Then, by \cref{claim:strong_2conn}, there exists a path $R$ in $D^* - z$ starting at $\End{P}$ and ending in $\Start{Q}$.
    Thus, $P \cup R \cup Q$ yields an $x$-$y$-walk in $D^* - z$, which gives us the required $x$-$y$-path in $D^{*} - z$.

    Hence, it suffices to appropriately define the set $W$ and prove \cref{claim:strong_2conn} for each possible type $\lambda$.
    
    \paragraph{Consider the case that $\lambda = \text{basic}$.}
    Let $u, v \in \V{D}$ be the unique vertices such that the edge added by the augmentation corresponds to the edge $(\mathsf{e}(v), \mathsf{e}(u)) \in E(D^{*})$.
    Let $W\coloneqq  \{\mathsf{b}(u), \mathsf{e}(u), \mathsf{b}(v), \mathsf{e}(v) \}$.
    Firstly, observe that the in- and out-degree conditions on $u$ and $v$ guarantee that the vertices in $W$ have the required property.

    \begin{claimproof}[Proof of \cref{claim:strong_2conn}]
        Let $\hat z\coloneqq  u$ if $z \in \{ \mathsf{b}(u), \mathsf{e}(u) \}$, let $\hat z\coloneqq  v$ if $z \in \{ \mathsf{b}(v), \mathsf{e}(v) \}$, or otherwise $\hat z \coloneqq  z$.
        Then there is an $x$-$y$-path in $D - \hat z$ as $D$ is strongly 2-connected.
        By possibly inserting the edges $(\mathsf{b}(u), \mathsf{e}(u))$ and/or $(\mathsf{e}(v), \mathsf{b}(v))$ this path extends to an $x$-$y$-path in $D^* - z$.
    \end{claimproof}
    
    \paragraph{Consider the case that $\lambda = \text{chain}$.}
    Let $u, v \in \V{D}$ be the unique vertices such that $D^{*}$ is obtained from $D$ by the addition of an $\mathsf{e}(v)$-$\mathsf{e}(u)$-chain through one of the edges in $\Set{(\mathsf{b}({u}),\mathsf{b}({v})),(\mathsf{b}({u}),\mathsf{e}({v})),(\mathsf{e}({u}),\mathsf{b}({v})),(\mathsf{e}({u}),\mathsf{b}({v})),(\mathsf{e}(u),\mathsf{e}(v))}$.
    Then, define $W$ as the union of $\{\mathsf{b}(u), \mathsf{e}(u), \mathsf{b}(v), \mathsf{e}(v) \}$ and the set of internal vertices of the chain.
    Observe that the in- and out-degree conditions on $u$ and $v$ and the definition of a chain guarantee that the vertices in $W$ have the required property.

    \begin{claimproof}[Proof of \cref{claim:strong_2conn}]
    Let $\hat z\coloneqq  u$ if $z \in W \setminus \{ \mathsf{b}(v), \mathsf{e}(v) \}$, let $\hat z\coloneqq  v$ if $z \in \{ \mathsf{b}(v), \mathsf{e}(v) \}$, or otherwise $\hat z \coloneqq z$.
    Then there is an $x$-$y$-path in $D - \hat z$ as $D$ is strongly 2-connected.
    By possibly inserting the edges $(\mathsf{b}(u), \mathsf{e}(u))$ and/or $(\mathsf{e}(v), \mathsf{b}(v))$, or replacing the edge $(u, v)$ with the $\mathsf{b}(u)$-$\mathsf{b}(v)$-path that is part of the chain, this path extends to an $x$-$y$-path in $D^* - z$.
    \end{claimproof}
        
    \paragraph{Consider the case that $\lambda = \text{collarette}$.}
    Let $u, v \in \V{D}$ be the unique vertices such that $D^{*}$ is the digraph obtained from $D$ by subdividing the edge $(v,u)$ with a vertex $a$, subdividing the edge $(u,v)$ with a vertex $b$, adding the edge $(a,b)$, and then adding a $b$-$a$-chain through $(a,b)$.
    Then, define $W$ as the union of $\Set{u, v, a, b}$ and the set of internal vertices of the chain.

    Observe that, since $D$ is strongly $2$-connected and $\Abs{V(D)} \geq 4$, we can guarantee that the vertices in $W$ have the required property.
    Indeed, assume it is false. 
    In particular, assume that there is a vertex $w \in W$ such that there are no two paths $w$-$(V(D^{*}) \setminus W)$ which are disjoint apart from $w$.
    This in fact implies that in $D$ both $u$ and $v$ have the same unique out-neighbour, which contradicts the fact that $D$ is strongly $2$-connected.
    The other case is symmetrical.

    \begin{claimproof}[Proof of \cref{claim:strong_2conn}]
    First, observe that if $z$ is any of the internal vertices of the added $b$-$a$-chain then we immediately conclude.
    Indeed, any $x$-$y$-path in $D$ can be extended to a $x$-$y$ path in $D^{*}$ avoiding such a $z$.
    Otherwise, let $\hat z \coloneqq  u$ if $z = a$, let $\hat z \coloneqq  v$ if $z = b$, or otherwise $\hat z \coloneqq  z$.
    Then there is an $x$-$y$-path in $D - \hat z$ as $D$ is strongly 2-connected.
    By possibly replacing the edge $(u, v)$ with the path $(u, b), (b, v)$ and possibly replacing the edge $(v, u)$ with the path $(v, a), (a, u)$, this path extends to an $x$-$y$-path in $D^* - z$.
    \end{claimproof}

    \paragraph{Consider the case that $\lambda = \text{bracelet}$.}
    Let $w \in \V{D}$ be the unique vertex with $\InNeighbours{}{w}=\{u,a\}$ and $\OutNeighbours{}{w}=\{v,b\}$, where either $a=v$ or $(u,a), (a,v) \in \E{D}$, and either $u=b$ or $(u,b), (b,v) \in \E{D}$, such that $D^*$ is obtained from $D$ by
    subdividing either the edge $(a,w)$ or the edge $(w, b)$ with a vertex $w'$ and adding the edges $(u,w')$ and $(w',v)$.
    We define $W = \Set{w, w'}$ and observe that the in- and out-degrees as well as the way the vertices in $\Set{ a, v, u, b, w }$ are connected, guarantees that the vertices in $W$ have the required property.

    \begin{claimproof}[Proof of \cref{claim:strong_2conn}]
    Let $\hat z \coloneqq w$ if $z = w'$, or otherwise $\hat z \coloneqq z$.
    Then there is an $x$-$y$-path in $D - \hat z$ as $D$ is strongly 2-connected.
    By possibly replacing the edges $(a, w)$ or $(w, b)$ by the paths $(a,w'), (w', w)$, or $(w, w'), (w', b)$ this path extends to an $x$-$y$-path in $D^* - z$.
    \end{claimproof}

This concludes the proof.
\end{proof}

Let $\Inverted{D}$ be the digraph obtained from a digraph $D$ by inverting all edges.
In some proof steps, it is necessary to distinguish cases where some vertex has large in- or large out-degree.
In these situations it is convenient to only treat one of the two cases and complete the other by using a symmetry argument.
In our case this symmetry argument would be to say that inverting all edges in the involved digraphs yields the situation of the case we treated explicitly.
The purpose of the next lemma is to give us the right to make such an argument.

\begin{lemma} \label{lem:inverted_augmentation}
    Let $D$, $H$, and $K$ be strongly $2$-connected digraphs such that $H$ and $K$ are butterfly-minors of $D$ and $K$ is an $H$-augmentation.
    Then $\Inverted{H}$ and $\Inverted{K}$ are butterfly-minors of $\Inverted{D}$ and $\Inverted{K}$ is an $\Inverted{H}$-augmentation.
\end{lemma}

\begin{proof}
   First notice that for every digraph $J$ and every directed walk $W$ in $J$ we have that $\Inverted{W}$ is also a directed walk in $\Inverted{J}$.
   This has the following important consequences.
    \begin{itemize}
        \item if $C$ is a directed cycle in $J$, then $\Inverted{C}$ is a directed cycle in $\Inverted{J}$,
        \item if $\{P_1,\dots,P_{\ell} \}$ is a collection of pairwise vertex-disjoint paths in $J$, then $\{\Inverted{P_1},\dots,\Inverted{P_{\ell}} \}$ is a collection of pairwise vertex-disjoint paths in $\Inverted{J}$ such that $\Start{P_i}=\End{\Inverted{P_i}}$ and $\End{P_i}=\Start{\Inverted{P_i}}$.
    \end{itemize}
    From this list of observations it follows immediately that $\Inverted{D}$, $\Inverted{H}$, and $\Inverted{K}$ are strongly $2$-connected.

    By definition, a digraph $F$ is a butterfly-minor of a digraph $J$ if it can be obtained from $J$ by a sequence of vertex deletions, edge deletions, and butterfly-contractions.
    Let $(u,v)$ be an edge in $J$.
    It is butterfly-contractible if and only if it is either the only edge of $J$ with head $v$ or the only edge of $J$ with tail $u$.
    Notice that in this case we get that $(v,u)$ is either the only edge of $\Inverted{J}$ with tail $v$ or the only edge of $\Inverted{J}$ with head $u$.
    Therefore, an edge $(u,v)$ of $J$ is butterfly-contractible if and only if $(v,u)$ is butterfly-contractible in $\Inverted{J}$.

    We now prove by induction on $(|E(J)|-|E(F)|)+(|V(J)|-|V(F)|)$ that $F$ being a butterfly-minor of $J$ implies that $\Inverted{F}$ is a butterfly-minor of $\Inverted{J}$.
    In case $(|E(J)|-|E(F)|)+(|V(J)|-|V(F)|)=0$ we have $F=J$ and thus $\Inverted{F}=\Inverted{J}$ and we are done.
    
    Now suppose there is some $x\in V(J)\cup E(J)$ such that $J-x$ contains $F$ as a butterfly-minor.
    Then we may use the induction hypothesis and obtain that $\Inverted{F}$ is a butterfly-minor of $\Inverted{J}-x$ and thus, $\Inverted{F}$ is a butterfly-minor of $\Inverted{J}$.

    Therefore, we may assume that no such $x$ exists and $F$ can be obtained from $J$ purely by butterfly-contractions.
    Let $(u,v)\in E(J)$ be some butterfly-contractible edge and let $J'$ be obtained from $J$ by contracting $(u,v)$ and let $x$ be the resulting vertex.
    By our discussion above, $(v,u)$ is butterfly-contractible in $\Inverted{J}$.
    Let $\Inverted{J}^{\star}$
    be obtained by contracting $(v,u)$ in $\Inverted{J}$ and let $y$ be the resulting vertex.
    Notice that $\Inverted{J}^{\star}-y=\Inverted{J'}-x$.
    Now, if some edge $(z,w)$, $z\in\{ u,v\}$, $w\in V(J)\setminus\{ u,v\}$, exists in $\Inverted{J}$, then $(w,z)$ exists in $J$, and while $y$ is the tail of the resulting edge in $\Inverted{J}^{\star}$, $x$ is the head of the resulting edge in $J'$, hence $x$ is the head of the resulting edge in $\Inverted{J'}$.
    The same holds true (with the words ``head'' and ``tail'' swapped) if $(w,z)\in E(\Inverted{J})$.
    Therefore, we obtain $\Inverted{J'}=\Inverted{J}^{\star}$.
    As $J'$ still contains $F$ as a butterfly-minor, $\Inverted{J'}$ is a butterfly-minor of $\Inverted{J}$, and $(|E(J')|-|E(F')|)+(|V(J')|-|V(F')|)<(|E(J)|-|E(F)|)+(|V(J)|-|V(F)|)$, the claim now follows from the induction hypothesis.

    Thus, it follows that both $\Inverted{H}$ and $\Inverted{K}$ are butterfly-minors of $\Inverted{D}$.

    What remains to discuss is that $\Inverted{K}$ is indeed an $\Inverted{H}$-augmentation.
    To see that this is true let $u$ and $v$ be the (up to two) vertices involved in the $H$-augmentation that yields $K$.
    If an out-split is performed at $u$ this means $u$ has out-degree at least three in $H$ and if an in-split is performed at $v$ this means $v$ has in-degree at least three in $H$.
    Observe that in $\Inverted{H}$, $v$ has out-degree at least three and $u$ has in-degree at least three.
    Notice that this allows us to perform the corresponding $\Inverted{H}$-augmentation with the roles of the vertices $u$ and $v$ swapped.
    Hence, $\Inverted{K}$ is indeed an $\Inverted{H}$-augmentation.
\end{proof}

\subsection{Outline of the proof}
\label{sec:proofoutline}

We prove \cref{thm:mainthm2} by induction on the difference between the number of edges in $D$ and the number of edges in $H$.
The base of this induction is given when $|\E{D}|=|\E{H}|$ as this implies $D \cong H$.
So $|\E{D}|>|\E{H}|$ may be assumed.
From this point onward, it suffices to show that there exists a strongly $2$-connected graph $K$ such that $K$ is a butterfly-minor of $D$ and an $H$-augmentation.
Since the latter implies $|E(K)|>|\E{H}|$, \cref{thm:mainthm2} then follows by induction.
That means that most of the proof for~\cref{thm:mainthm2} implies the following corollary.

\begin{corollary}\label{thm:main}
        Let $D$ and $H$ be strongly $2$-connected digraphs and $H'\subsetneq D$ be an $H$-expansion.
        Then there exists a strongly $2$-connected digraph $K$ such that $D$ contains a $K$-expansion $K'$ as a subgraph and $K$ is an $H$-augmentation.
\end{corollary}

So the essence of the discussion is to facilitate a single step in the sequence promised by the conclusion of \cref{thm:mainthm2}.
See~\cref{fig:generating} for an illustration of this process.

\begin{figure}[!ht]
    \centering
    \begin{tikzpicture}[scale=0.8]
        
        \node (H) [v:ghost,inner sep=3pt] {$D_0$};
        \node (K) [v:ghost,position=0:30mm from H,inner sep=3pt] {$D_1$};
        \node (HP) [v:ghost,position=90:24mm from H,inner sep=3pt] {$D_0'$};
        \node (KP) [v:ghost,position=90:24mm from K,inner sep=3pt] {$D_1'$};
        \node (D2) [v:ghost,position=0:30mm from K,inner sep=3pt] {$D_2$};
        \node (D2P) [v:ghost,position=90:24mm from D2,inner sep=3pt] {$D_2'$};
        \node (D3a) [v:ghost,position=0:9mm from D2] {};
        \node (D3b) [v:ghost,position=0:8mm from D3a,inner sep=5pt] {$\cdots$};
        \node (D3bP) [v:ghost,position=90:24mm from D3b,inner sep=5pt] {$\cdots$};
        \node (D3c) [v:ghost,position=0:8mm from D3b] {};
        \node (Dn-1) [v:ghost,position=0:12mm from D3c,inner sep=3pt] {$D_{n-1}$};
        \node (Dn-1P) [v:ghost,position=90:24mm from Dn-1,inner sep=3pt] {$D_{n-1}'$};
        \node (Dn) [v:ghost,position=0:30mm from Dn-1,inner sep=3pt] {$D_n$};
        \node (DnP) [v:ghost,position=90:24mm from Dn,inner sep=3pt] {$D_n'$};
        \node (D) [v:ghost,position=90:30mm from DnP,inner sep=3pt,ellipse,minimum width=2.5pt, minimum height=1.7pt] {$D\in\mathcal{D}_2$};

        \node (labelalignment) [v:ghost,position=90:12mm from DnP] {};
        \node (subset_n) [v:ghost,position=180:4mm from labelalignment] {\small $=$};
        \node (subset_n-1) [v:ghost,position=180:20mm from subset_n] {\small $\subsetneq$};
        \node (subset_3) [v:ghost,position=180:12.5mm from subset_n-1] {\small $\cdots$};
        \node (subset_2) [v:ghost,position=180:18.5mm from subset_3] {\small $\subsetneq$};
        \node (subset_1) [v:ghost,position=180:29.5mm from subset_2] {\small $\subsetneq$};
        \node (subset_0) [v:ghost,position=180:27.2mm from subset_1] {\small $\subsetneq$};

        \node (Hintermediate) [v:ghost,position=90:9mm from H] {};
        \node (H_expansion_label) [v:ghost,position=180:5.5mm from Hintermediate] {\small exp.};
        \node (Kintermediate) [v:ghost,position=90:9mm from K] {};
        \node (K_expansion_label) [v:ghost,position=180:5.5mm from Kintermediate] {\small exp.};
        \node (D2intermediate) [v:ghost,position=90:9mm from D2] {};
        \node (D2_expansion_label) [v:ghost,position=180:5.5mm from D2intermediate] {\small exp.};
        \node (Dn-1intermediate) [v:ghost,position=90:9mm from Dn-1] {};
        \node (Dn-1_expansion_label) [v:ghost,position=180:5.5mm from Dn-1intermediate] {\small exp.};
        \node (Dnintermediate) [v:ghost,position=90:9mm from Dn] {};
        \node (Dn_expansion_label) [v:ghost,position=180:4mm from Dnintermediate] {\small $\cong$};
        \node (HKintermediate) [v:ghost,position=0:15mm from H] {};
        \node (HK_augmentation_label) [v:ghost,position=270:5mm from HKintermediate] {\small aug.};
        \node (KD2intermediate) [v:ghost,position=0:15mm from K] {};
        \node (KD2_augmentation_label) [v:ghost,position=270:5mm from KD2intermediate] {\small aug.};
        \node (D2Dn-1intermediate) [v:ghost,position=0:0mm from D3b] {};
        \node (D2Dn-1_augmentation_label) [v:ghost,position=270:5mm from D2Dn-1intermediate] {\small aug.};
        \node (Dn-1Dnintermediate) [v:ghost,position=0:15mm from Dn-1] {};
        \node (Dn-1Dn_augmentation_label) [v:ghost,position=270:5mm from Dn-1Dnintermediate] {\small aug.};

        \draw [e:main,->,line width=1.4pt] (H) to (HP);
        \draw [e:main,->,bend left=15] (HP) to (D);
        \draw [e:main,->,out=30,in=186,color=PastelOrange,line width=1.5pt] (KP) to (D);
        \draw [e:main,->,bend left=10] (D2P) to (D);
        \draw [e:main,->,bend left=5] (Dn-1P) to (D);
        \draw [e:main,->] (DnP) to (D);
        \draw [e:main,->,color=PastelOrange,line width=1.5pt] (H) to (K);
        \draw [e:main,->,color=PastelOrange,line width=1.5pt] (K) to (KP);
        \draw [e:main,->,line width=1.1pt] (K) to (D2);
        \draw [e:main,line width=1.1pt] (D2) to (D3a);
        \draw [e:main,->,line width=1.1pt] (D3c) to (Dn-1);
        \draw [e:main,->,line width=1.1pt] (Dn-1) to (Dn);
        \draw [e:main,->,line width=1.1pt] (D2) to (D2P);
        \draw [e:main,->,line width=1.1pt] (Dn-1) to (Dn-1P);
        \draw [e:main,->,line width=1.1pt] (Dn) to (DnP);

        \node (HH) [v:ghost,inner sep=3pt,position=270:5mm from H] {\begin{tikzpicture} \node (H) [rotate=90] {$=$}; \end{tikzpicture}};
        \node (HHH) [v:ghost,inner sep=3pt,position=270:5mm from HH] {$H$};
        \node (KK) [v:ghost,inner sep=3pt,position=270:5mm from K] {\begin{tikzpicture} \node (K) [rotate=90] {$=$}; \end{tikzpicture}};
        \node (KKK) [v:ghost,inner sep=3pt,position=270:5mm from KK] {$K$};

    \end{tikzpicture}
    \caption{A diagram of the generation sequence of a strongly $2$-connected digraph $D$ from a strongly $2$-connected butterfly-minor $H$ of $D$.
    The step from $D_0=H$ to $D_1=K$ is the step isolated in~\cref{thm:main}.
    This is the step that is made explicit in the inductive proof of~\cref{thm:mainthm2}.}
    \label{fig:generating}
\end{figure}
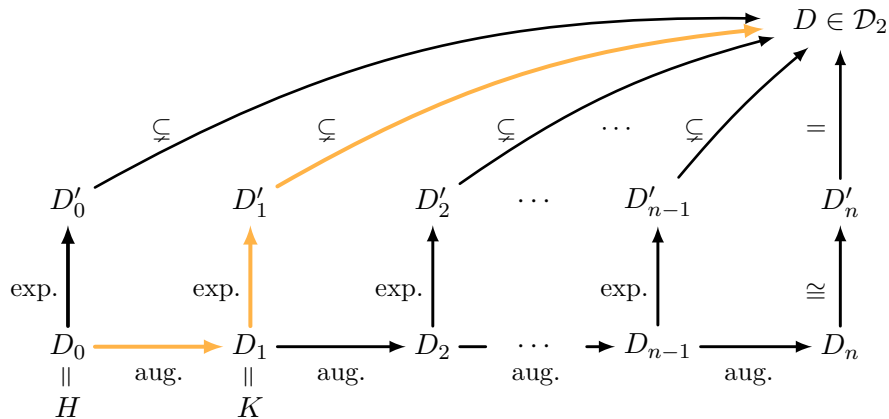

Let $H'$ be some $H$-expansion in $D$ and let $P$ be a path which is internally disjoint from $H'$ but which has both endpoints in $H'$.
Such a path is called an \emph{$H'$-\earpath} later on.
The engine of our proof is an analysis of how such $H'$-\earpaths behave with respect to $H'$.
In~\cref{sec:earpaths}, we provide the formal definition together with a full classification of $H'$-\earpaths into the three categories \emph{augmenting}, \emph{\bad}, and \emph{switching}.
The properties of \emph{switching} paths are also discussed here:
Switching paths are paths that attach to $H'$ in a way that makes them somewhat ``parallel'' to the $H$-expansion.
This allows to replace (or switch) a path within $H'$ with the switching path to obtain a different $H'$-augmentation.
See~\cref{lem:expansion_stays_expansion_after_switching} for the details.
Switching paths do not allow the construction of an $H$-augmentation on their own; however, in most cases, they can be used to ``optimise'' the $H$-expansion $H'$, which happens in various ways throughout \cref{sec:obtaining_the_augmentations,sec:escapes,sec:mainproof}.

In \cref{sec:obtaining_the_augmentations}, we show that \emph{augmenting} $H'$-\earpaths immediately give rise to an $H$-augmentation $K$ such that $D$ contains a $K$-expansion.
However, augmenting paths are the only kind for which this is so clear.
We also show in this section that whenever we encounter a path $P$ which is an $H'$-\earpath itself, but a path that ``weaves\footnote{We call this property to be \emph{laced} with another path and describe the exact situation in \cref{thm:add_chain}.}'' through some path in $H'$ representing an edge of $H$, we may also find our $H$-augmentation $K$ together with its $K$-expansion $K'$ in $D$.

We have now dealt with both the good (augmenting) and the ugly (switching).
Therefore, what is left is the \bad.
\Bad paths are, essentially, such $H'$-\earpaths that form directed cycles together with the branchset of a single vertex of $H$ or the edge between two vertices of $H$.
In most cases, a \bad path represents one of the intermediate edges of some path that was introduced as part of a chain augmentation (recall \cref{fig:adding_a_chain} for example).
Any such intermediate edge of a chain identifies a vertex of $H'$ that witnesses $H'$ not to be strongly $2$-connected itself (see \cref{subsec:exist_escapes} for a formal discussion of this situation).
Since $D$ itself is strongly $2$-connected, this yields a special kind of \bad path which we call an \emph{escape}.
The purpose of~\cref{sec:escapes} is to identify sufficient conditions for an escape to exist and to show that whenever we find an escape, we also find some path that has the properties required by~\cref{thm:add_chain} eventually leading to an $H$-augmentation $K$ together with its expansion $K'$ in $D$.

Finally, we combine the tools obtained in the previous sections in~\cref{sec:mainproof} to provide a full proof of~\cref{thm:mainthm2}.

\section{Expansions and \earpaths}
\label{sec:expansions}
\label{sec:earpaths}

This section is concerned with some basic properties of directed paths.
In particular those paths that start and end in some prescribed subgraph $H$ and are otherwise disjoint from $H$.

\paragraph{\Earpaths.}

Let $D$ be a digraph and $J$ be a subdigraph of $D$.
An \emph{\earpath} of $J$, or \emph{$J$-\earpath}, is either an edge $(u,v)\in \E{D}$ with $u,v\in V(J)$ but $(u,v)\notin E(J)$, or a directed path in $D$ that starts and ends in $J$ but is internally disjoint from $J$.

\subsection{Different types of \earpaths}

We distinguish between different types of \earpaths for butterfly-models.
In the following let $D$ and $H$ be digraphs and and $H'$ be an $H$-expansion in $D$.

\paragraph{Switching paths.}
An $H'$-\earpath $P$ is called \emph{switching}, or a \emph{switching path of $H'$}, if there is $(u,v) \in \E{H}$ such that
\begin{enumerate}[label=\textbf{[S\arabic*]},itemindent=1em]
\item \label{item:switching_models_1} $\Start{P} \in \V{\outstar{u} \cup \intree{v} \cup \rpath{v}}$ and $\End{P} \in \V{\rpath{u} \cup \outtree{u} \cup \instar{v}}$,
\item \label{item:switching_stars_2} either $\Start{P} \in \V{\outstar{u}}$, or $\End{P} \in \V{\instar{v}}$, and
\item \label{item:switching_acyclic_3} $\outstar{u} \cup \instar{v} \cup \Set{e~|~\Head{e}\in \V{\instar{v}},\Tail{e}\in \V{\outstar{u}}} \cup P$ is acyclic.
\end{enumerate}
Further we call $P$ \emph{parallel-switching} if there is a $\Start{P}$-$\End{P}$-path in $\outstar{u} \cup \instar{v} \cup \Set{e~|~\Head{e}\in \V{\instar{v}},\Tail{e}\in \V{\outstar{u}}}$.
Otherwise, $P$ is called \emph{non-parallel-switching}.
See~\cref{fig:switching_paths} for an illustration.

\begin{figure}[!ht]
    \centering
    \includegraphics[scale=1.7]{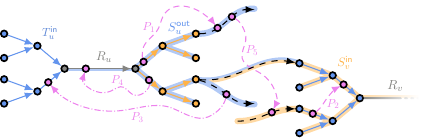}
    \caption{The models of two vertices $u$ and $v$.
    	The trees \textcolor{CornflowerBlue}{$\intree{u}$} and \textcolor{CornflowerBlue}{$\intree{v}$} are depicted in \textcolor{CornflowerBlue}{blue}, while the tree \textcolor{PastelOrange}{$\outtree{u}$} is depicted in \textcolor{PastelOrange}{yellow}.
    	The star \textcolor{CornflowerBlue}{$\outstar{u}$} is marked in \textcolor{CornflowerBlue}{blue} and the star \textcolor{PastelOrange}{$\instar{v}$} is marked in \textcolor{PastelOrange}{yellow}.
    	The two \textcolor{DarkGray}{\redpaths} are depicted in \textcolor{DarkGray}{grey}.
    	Additionally, the \textcolor{LavenderMagenta}{pink} dashed paths are \textcolor{LavenderMagenta}{switching} paths, of these \textcolor{LavenderMagenta}{$P_1$} is parallel, while \textcolor{LavenderMagenta}{$P_2$} and \textcolor{LavenderMagenta}{$P_5$} are non-parallel.
    	The \textcolor{LavenderMagenta}{pink} dash-dotted paths, \textcolor{LavenderMagenta}{$P_3$} and \textcolor{LavenderMagenta}{$P_4$}, are both \textcolor{LavenderMagenta}{bad}.}
    \label{fig:switching_paths}
    \label{fig:bad_paths}
    \label{fig:stars_and_trees}
\end{figure}

Let $P$ be an  $H'$-switching path. 
Observe that there is a unique path $Q$ in $H'$ such that the graph $H''$ obtained from $H' \cup P$ by deleting all internal vertices and all edges of $Q$ is an $H$-expansion;
We say $H''$ is obtained from $H'$ by \emph{switching onto $P$}.
We formalise this observation below, the proof provides more precisely how $Q$ can be obtained.

\begin{observation}
    \label{lem:expansion_stays_expansion_after_switching}
    Let $D,H$ be strongly $2$-connected digraphs and $H'$ be an $H$-expansion in $D$.
    Let $P$ be an $H'$-\earpath that is switching, then the graph $H''$ obtained from $H'$ by switching onto $P$ is an $H$-expansion in $D$.
\end{observation}

\begin{proof}
    Let $(u,v) \in \E{H}$ be such that $P$ is a switching path with respect to $(u,v)$.
    We consider \cref{item:switching_stars_2} and begin with the case that
    $\Start{P} \in V(\outstar{u, H'})$ \underline{and} $\End{P} \in  V(\instar{v, H'})$.

    Let $Q$ be the $u$-$v$-bridge of $H'$.
    Further, let $q_1$ be $\Start{P}$ if $\Start{P}$ lies on $Q$ and otherwise $q_1\coloneqq \Start{Q}$.
    Similarly, let $q_2$ be $\End{P}$ if $\End{P}$ lies on $Q$ and otherwise $q_2\coloneqq \End{Q}$.
    Note that $q_1 <_{Q} q_2$ since $\outstar{u, H'} \cup \instar{v, H'} \cup \Set{e~|~\Head{e}\in \V{\instar{v, H'}},\Tail{e}\in \V{\outstar{u, H'}}} \cup P$ is acyclic, that is \cref{item:switching_acyclic_3}.
    We consider the graph $H''$ obtained from $H' \cup P$ by deleting all internal vertices and all edges of $q_1 Q q_2$.
    
    Notice that in the construction of $H''$ we did neither add an outgoing edge to any vertex of $\bigcup_{w \in V(H)} \V{\intree{w, H'} \graphminus \rpath{w, H'}}$ nor an ingoing edge to any vertex of $\bigcup_{w \in V(H)} \V{\outtree{w, H'} \graphminus \rpath{w, H'}}$.
    Further, we only removed vertices and edges outside branchsets.
    Thus the branchset $\outtree{w,H'}\cup\rpath{w,H'}\cup\intree{w,H'}$ is contained in $H''$ and can be contracted into a single vertex in $H''$ for every $w \in V(H)$.
    By contracting $\outtree{w,H'}\cup\rpath{w,H'}\cup\intree{w,H'}$ for every $w \in V(H)$ in $H''$ and further contracting the unique $V(\outtree{u, H'})$-$\Start{P}$-path in $\outstar{u, H'}$ and the unique $\End{P}$-$V(\intree{v, H'})$-path in $\instar{v, H'}$, we obtain a subdivision of $H$.
    This implies that $H''$ is indeed an $H$-expansion.

    What is left is to consider the case for \cref{item:switching_stars_2} where either $\Start{P} \notin \V{\outstar{u, H'}}$ or $\End{P} \notin \V{\instar{v, H'}}$.
    We only treat the case where $\End{P} \notin \V{\instar{v, H'}}$ since the other case can be translated into this one by using~\cref{lem:inverted_augmentation}.
    Thus $\End{P} \in \V{\rpath{u, H'} \cup \outtree{u, H'}}$.
    The property \cref{item:switching_stars_2} implies that $\Start{P} \in \V{\instar{v, H'}}$.

    There exists a unique $\Start{\rpath{u, H'}}$-$\End{P}$-path $O$ in $\rpath{u, H'} \cup \outtree{u, H'}$.
    Note that $\Start{\rpath{u, H'}} \neq \End{P}$ since otherwise $\outstar{u, H'} \cup P$ contains a cycle.
    Furthermore, $\Start{P}$ is contained in $\V{O - \End{O}}$ if $\End{P} \in \V{\rpath{u, H'}}$, and otherwise the root of $\outtree{u, H'}$ is contained in $\V{O - \End{O}}$.
    Let $x$ be the last vertex of $O - \End{O}$ that is either $\Start{P}$ or has out-degree at least two in $H'$.
    We consider the graph $H''$ obtained from $H' \cup P$ by deleting all internal vertices and edges of $Q \coloneqq xO$.

    We show that $H''$ is an $H$-expansion.
    Let $\rpath{u, H''} \cup \outtree{u, H''}$ be the union of $(\rpath{u, H'} \cup \outtree{u, H'}) \cap H''$, $P$ and the unique $\V{\rpath{u, H'} \cup \outtree{u, H'}}$-$\Start{P}$-path in $\outstar{u, H'}$.
    Notice that $\rpath{u, H''} \cup \outtree{u, H'''}$ is indeed an out-arborescence, all whose vertices but its root have in-degree one in $H''$.
    By contracting $\intree{u, H'} \cup \rpath{u, H''} \cup \outtree{u, H'''}$ and $\outtree{w,H'}\cup\rpath{w,H'}\cup\intree{w,H'}$ for any $w \in V(H) \setminus \{ u \}$ in $H''$ we obtain a subdivision of $H$.
    This implies that $H''$ is indeed an $H$-expansion.
\end{proof}

\paragraph{\Bad paths.}

An $H'$-\earpath $P$ is called \emph{$H'$-\bad}, or \emph{\bad}, if $P$ starts and ends in $\V{\instar{v}} \cup \V{\outstar{v}}$ for some $v \in \V{H}$ and $\instar{v} \cup \outstar{v} \cup P$ contains a directed cycle.
See \cref{fig:bad_paths} for an illustration.

Note that $\Set{ V(\instar{v}), V(\outtree{v} \graphminus \rpath{v}): v \in \V{H}}$ is a partition of $V(H')$. Similarly, the set $\Set{ V(\intree{v} \graphminus \rpath{v}), V(\outstar{v}): v \in \V{H}}$ is a partition of $V(H')$.
\smallskip

\paragraph{Augmenting paths.}
Let $P$ be a $\mu$-\earpath and $u,v \in \V{H}$ such that $\Start{P} \in \V{\intree{u} \cup \outstar{u}}$ and $\End{P} \in \V{\instar{v} \cup \outtree{v}}$.
We say that $P$ is \emph{$H'$-augmenting} or simply \emph{augmenting}, see \cref{fig:augmenting_paths} for an illustration, if one of the following conditions is satisfied:
\renewcommand{\labelenumi}{\textbf{\theenumi}}
\renewcommand{\theenumi}{(A\arabic{enumi})}
\begin{enumerate}[labelindent=0pt,labelwidth=\widthof{\ref{last-item-augmenting}},itemindent=1em]
    \item\label{item:augmenting_3} $\Start{P} \in V(\outstar{u})$, $\End{P} \in V(\instar{v})$, $u \neq v$, and $(u, v) \notin \E{H}$,
    \item\label{item:augmenting_1} $\Start{P} \in V(\intree{u} \graphminus \rpath{u})$ and $\End{P} \in V(\outtree{v} \graphminus \rpath{v})$,
    \item\label{item:augmenting_2} $\Start{P} \in V(\intree{u} \graphminus \rpath{u})$, $\End{P} \in V(\instar{v})$, and $u \neq v$, or
    \item\label{item:augmenting_4} $\Start{P} \in V(\outstar{u})$, $\End{P} \in V(\outtree{v} \graphminus \rpath{v})$, and $u \neq v$.
    \label{last-item-augmenting}
\end{enumerate}

\begin{figure}[!ht]
    \centering
    \includegraphics[scale=1.5]{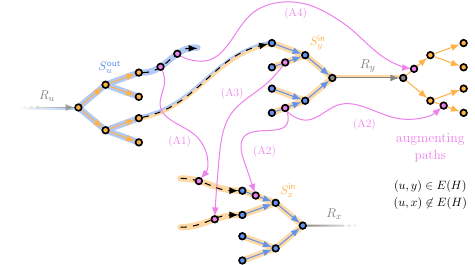}
    \caption{An illustration of the four different types of augmenting paths.}
    \label{fig:augmenting_paths}
\end{figure}

\begin{lemma}
    \label{lem:different-earpaths}
    Let $D$ and $H$ be strongly $2$-connected digraphs and $H'$ be an $H$-expansion in $D$.
    Then every $H'$-\earpath is exactly one of $H'$-switching, $H'$-\bad or $H'$-augmenting.
\end{lemma}
\begin{proof}
    We begin by showing that any $H'$-\earpath is $H'$-switching, $H'$-\bad or $H'$-augmenting.
    Let $P$ be a non-augmenting $H'$-\earpath.
    Then, none of the conditions \cref{item:augmenting_1,item:augmenting_2,item:augmenting_3,item:augmenting_4} is satisfied.
    
    Note that there is some $u\in \V{H}$ such that either $\Start{P}\in V(\intree{u} \graphminus \rpath{u})$ or $\Start{P} \in V(\outstar{u})$.
    If $\Start{P}\in V(\intree{u} \graphminus \rpath{u})$, since conditions \cref{item:augmenting_1,item:augmenting_2} are not satisfied it holds that $\End{P}\notin V(\outtree{v}\graphminus \rpath{v})$ and, if $u\neq v$, $\End{P}\notin V(\instar{v})$.
    Since $\rpath{v}\subseteq \instar{v}$ it follows that in this case $\End{P}\in V(\instar{u})$.
    
    If $\Start{P} \in V(\outstar{u})$, then, since conditions \cref{item:augmenting_3,item:augmenting_4} are not satisfied, either $\End{P} \in V(\instar{u})$, or $\End{P} \in V(\instar{v})$ for some $(u,v)\in \E{H}$, or $\End{P} \in V(\outtree{u} \graphminus \rpath{u})$.
    Notice that, with $\rpath{v}\subseteq \instar{v}$, the case where $\End{P}\in V(\rpath{v})$ is equal to the case where $\End{P}\in V(\instar{v})$ for some $(u,v)\in E(H)$.
    
    All four cases above can be resolved as follows:
    If $\Start{P}$ and $\End{P}$ are contained in $\iostar{u}$ for some $* \in \Set{\variablestyle{in},\variablestyle{out}}$, $P$ is either switching or \bad by definition.
    Thus, we may suppose that either $\Start{P} \in V(\outstar{u} \graphminus \rpath{u})$, $\End{P} \in V(\instar{u} \graphminus \rpath{u})$ for some $u \in \V{H}$ or $\Start{P} \in V(\outstar{u} \graphminus \instar{v})$, $\End{P} \in V(\instar{v} \graphminus \outstar{u})$ for $(u,v) \in \E{H}$.
    In the former case, $\instar{u} \cup \outstar{u} \cup P$ contains a cycle, and thus $P$ is \bad.
    In the latter case, $\outstar{u} \cup \instar{v} \cup P$ does not contain a cycle, and thus \cref{item:switching_models_1,item:switching_stars_2,item:switching_acyclic_3} hold and $P$ is switching.

    Now we argue that every $H'$-\earpath is exactly one of $H'$-augmenting, $H'$-switching or $H'$-\bad.
    Notice that every $H'$-switching path $P$ with respect to an edge $(u,v) \in \E{H}$ has the property that $\Start{P},\End{P}\in \V{\outstar{u}}\cup\V{\instar{v}}$ by \cref{item:switching_models_1} and that $\outstar{u} \cup \instar{v}  \cup \Set{e~|~\Head{e}\in \V{\instar{v}},\Tail{e}\in \V{\outstar{u}}} \cup P$ is acyclic by \cref{item:switching_acyclic_3}.
    Thus $\instar{u} \cup \outstar{u} \cup P$ and $\instar{v} \cup \outstar{v} \cup P$ are acyclic, which implies that no $H'$-switching path can be \bad.
    Moreover, $P$ cannot be augmenting since \cref{item:augmenting_1,item:augmenting_2} cannot be satisfied by \cref{item:switching_models_1}, with $(u,v)\in \E{H}$, \cref{item:augmenting_3} cannot be satisfied, and \cref{item:switching_stars_2} also rules out \cref{item:augmenting_4}.
    
    Let $P$ be an $H'$-\bad path for some vertex $w\in \V{H}$, i.e.~$\Start{P}, \End{P} \in \V{\instar{w} \cup \outstar{w}}$.
    We prove that $P$ is not $H'$-augmenting.
    If $\Start{P} \in \V{\outstar{w}}$, $P$ is not $H'$-augmenting since $\End{P}$ is either contained in $\V{\instar{w} \cup (\outtree{w} \graphminus \rpath{w})}$ or in $\V{\instar{v}}$ for some $v \in \OutNeighbours{H}{w}$.

    If $\Start{P} \in \V{\intree{w} \graphminus \rpath{w}}$, then $\End{P} \in \V{\instar{w}}$ since there exists an $\End{P}$-$\Start{P}$-path in $\instar{w} \cup \outstar{w}$ by definition of \bad-path.
    This implies that $P$ is not $H'$-augmenting.
    Otherwise, $\Start{P} \in \V{\outstar{u}}$ for some $u \in \InNeighbours{H}{w}$.
    Then $\End{P} \in \V{\instar{w}}$ since there exists an $\End{P}$-$\Start{P}$-path in $\instar{w} \cup \outstar{w}$.
    Thus $P$ is not $H'$-augmenting.
\end{proof}

\subsection{Laced paths}\label{subsec:lacedpaths}
This subsection deals with two directed paths that intersect.
While the paths do not necessarily need to be \earpaths, normalising the intersection between two directed paths is an important tool that finds application throughout this paper.
These so-called ``laced'' paths are what gives rise to chain-augmentations in particular.

Two paths $P$ and $Q$ are \emph{laced} if every weakly connected component of $P \cap Q$ is a directed path and, additionally, if there are at least two such components, then for the vertices $z_1,w_1,\dots,z_n,w_n \in \V{P}$ with $P \cap Q = \Set{z_1Pw_1,\dots,z_nPw_n}$ for all $1 \leq i < n$ the vertex $w_i$ reaches $z_{i+1}$ on $P$ and for all $1 < j \leq n$ the vertex $w_j$ reaches $z_{j-1}$ on $Q$.
Two paths $P$ and $Q$ are \emph{properly laced} if they are laced but not disjoint.
See~\cref{fig:laced} for an illustration.

\begin{figure}[!ht]
    \centering
    \begin{tikzpicture}[scale=1]
			
	\pgfdeclarelayer{background}
	\pgfdeclarelayer{foreground}
			
	\pgfsetlayers{background,main,foreground}
			
	\begin{pgfonlayer}{main}

        \node (a) [v:main] {};
        \node (z1) [v:main,position=0:10mm from a] {};
        \node (w1) [v:main,position=0:10mm from z1] {};
        \node (z2) [v:main,position=0:10mm from w1] {};
        \node (w2) [v:main,position=0:10mm from z2] {};
        \node (z3) [v:main,position=0:10mm from w2] {};
        \node (w3) [v:main,position=0:10mm from z3] {};
        \node (z4) [v:main,position=0:10mm from w3] {};
        \node (w4) [v:main,position=0:10mm from z4] {};
        \node (b) [v:main,position=0:10mm from w4] {};
        \node (c) [v:main,position=90:12mm from b] {};
        \node (d) [v:main,position=90:12mm from a] {};

        \node (z1g) [v:ghost,position=270:0.5mm from z1] {};
        \node (w1g) [v:ghost,position=270:0.5mm from w1] {};
        \node (z2g) [v:ghost,position=270:0.5mm from z2] {};
        \node (w2g) [v:ghost,position=270:0.5mm from w2] {};
        \node (z3g) [v:ghost,position=270:0.5mm from z3] {};
        \node (w3g) [v:ghost,position=270:0.5mm from w3] {};
        \node (z4g) [v:ghost,position=270:0.5mm from z4] {};
        \node (w4g) [v:ghost,position=270:0.5mm from w4] {};

        \node (Plabel) [v:ghost,position=270:3.5mm from b] {$P$};
        \node (Qlabel) [v:ghost,position=270:3.5mm from d] {\textcolor{LavenderMagenta}{$Q$}};
        \node (z1label) [v:ghost,position=90:2.5mm from z1] {$z_1$};
        
        \node (w1label) [v:ghost,position=270:2.8mm from w1] {$w_1$};
        \node (z2label) [v:ghost,position=270:2.8mm from z2] {$z_2$};

        \node (w2label) [v:ghost,position=90:2.5mm from w2] {$w_2$};
        \node (z3label) [v:ghost,position=90:2.5mm from z3] {$z_3$};

        \node (w3label) [v:ghost,position=270:2.8mm from w3] {$w_3$};
        \node (z4label) [v:ghost,position=270:2.8mm from z4] {$z_4$};

        \node (w4label) [v:ghost,position=90:2.5mm from w4] {$w_4$};

	\end{pgfonlayer}
			
	\begin{pgfonlayer}{background}

        \draw[e:main,->] (a) to (z1);
        \draw[e:main,->] (z1) to (w1);
        \draw[e:main,->] (w1) to (z2);
        \draw[e:main,->] (z2) to (w2);
        \draw[e:main,->] (w2) to (z3);
        \draw[e:main,->] (z3) to (w3);
        \draw[e:main,->] (w3) to (z4);
        \draw[e:main,->] (z4) to (w4);
        \draw[e:main,->] (w4) to (b);

        \draw[e:main,->,color=LavenderMagenta] (z1g) to (w1g);
        \draw[e:main,->,color=LavenderMagenta] (z2g) to (w2g);
        \draw[e:main,->,color=LavenderMagenta] (z3g) to (w3g);
        \draw[e:main,->,color=LavenderMagenta] (z4g) to (w4g);

        \draw[e:main,->,color=LavenderMagenta,bend right=35] (c) to (z4);
        \draw[e:main,->,color=LavenderMagenta,bend left=35] (w4) to (z3);
        \draw[e:main,->,color=LavenderMagenta,bend right=35] (w3) to (z2);
        \draw[e:main,->,color=LavenderMagenta,bend left=35] (w2) to (z1);
        \draw[e:main,->,color=LavenderMagenta,bend right=35] (w1) to (d);
 
	\end{pgfonlayer}	
			
	\begin{pgfonlayer}{foreground}
	\end{pgfonlayer}
   
\end{tikzpicture}
    \caption{Two properly laced paths $P$ and $Q$.}
    \label{fig:laced}
\end{figure}
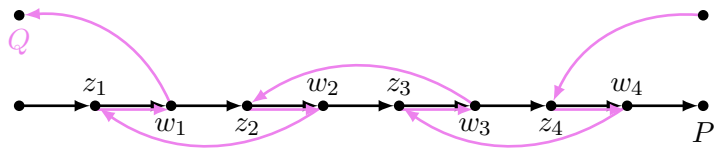

\begin{lemma}
    \label{lem:laced}
  Let $P$ and $Q$ be directed paths in a digraph $D$.
  Then, there exists a path $Q'\subseteq D$ with the same endpoints as $Q$ such that $Q' \subseteq P \cup Q$ and $P$ and $Q'$ are laced.
\end{lemma}
\begin{proof}
    Let $Q' \subseteq Q \cup P$ be a $\Start{Q}$-$\End{Q}$-path minimising the number of weakly connected components of $Q'\cap P$. We claim that $P$ and $Q'$ are laced. 

    For, if $Q'$ and $P$ are disjoint, then they are laced by definition. Otherwise, let $P_1, \dots, P_{\ell}$ be the weakly connected components of $P \cap Q'$ in the order in which they appear on $P$. Note that any such component is a directed path, possibly of length $0$.
    For $i\in[\ell]$ let $p_{2i-1} \coloneqq  \Start{P_i} \in \V{P}$ and let $p_{2i} \coloneqq  \End{P_i}$.
    Then $p_1, \dots, p_{2\ell}$ appear on $P$ in this order.
    Now, if $Q'$ and $P$ are not laced, then there must be an index $i\in[\ell-1]$ such that $p_{2i}$ appears on $Q'$ before $p_{2(i+1)-1}$.
    But $P' \coloneqq  p_{2i}Pp_{2i+1}$ is a subpath of $P$ disjoint from $Q'$ and thus $Q'' \coloneqq  Q'p_{2i} \cdot P' \cdot p_{2i+1}Q'$ is a $\Start{Q}$-$\End{Q}$-path such that $Q'' \cap P$ has fewer weakly connected components than $Q'\cap P$, a contradiction to the choice of $Q'$.
\end{proof}

\section{From \earpaths to augmentations}
\label{sec:obtaining_the_augmentations}

Now, by categorising different types of \earpaths, we understand how an expansion can interact with the remainder of the graph.
Next, we want to explain how the different \earpaths can be used to obtain one of the described augmentations (see~\cref{subsec:augmentations}).

\subsection{Augmenting paths give basic augmentations}

The most straightforward way how to extend an expansion to the expansion of a larger, strongly 2-connected graph is by using an augmenting path.
We show that this always yields a basic augmentation.

\begin{theorem}
    \label{thm:add_augmenting_path}
    Let $D,H$ be strongly $2$-connected digraphs, and $H'$ be an $H$-expansion in $D$.
    If $D$ contains an $H'$-augmenting path, then $D$ admits a basic $H$-augmentation.
\end{theorem}
\begin{proof}

Let $Q$ be an $H'$-augmenting path in $D$.
Our goal is to show that $Q$ can be used to find a basic $H$-augmentation $K$ by constructing a $K$-expansion $K'$ as a subgraph of $D$.

Since $Q$ is $H'$-augmenting, there are vertices $u, v \in \V{H}$ such that $\Start{Q} \in \V{\intree{u}} \cup \V{\outstar{u}}$ and $\End{Q} \in \V{\instar{v}} \cup\V{\outtree{v}}$ and one of the conditions \cref{item:augmenting_1,item:augmenting_2,item:augmenting_3,item:augmenting_4} is satisfied.

We proceed to define a digraph $H^{*}$ from $H$ as follows.
If conditions \cref{item:augmenting_3} or \cref{item:augmenting_4} hold, let $u^{*} = u$.
Otherwise, if conditions \cref{item:augmenting_1} or \cref{item:augmenting_2} hold, perform an in-split at $u$ in $H$ as follows.
Consider the maximal in-arborescence $T^i$ in $\intree{u}$ that is rooted at $\Start{Q}$.
This provides a partition of $\InNeighbours{H}{u}$ as follows.
Let $N^u_e$ be the set of vertices $x \in \InNeighbours{H}{u}$ such that the $x$-$u$-bridge of $H'$ ends in a leaf of $T^i$.
Also, let $N^u_b\coloneqq \InNeighbours{}{u}\setminus N^u_e$.
By definition of $\intree{u}$ we obtain $\Abs{N^u_e} \geq 2$, and since $\Start{Q} \notin \V{\rpath{u}}$ we have $\Abs{N^u_b} \geq 1$, which yields an in-split of $u$.
Let $H_{2}$ be the digraph obtained by the previous in-split and let $u^{*} = \exposed{u}$ and
\begin{equation*}
    v' = \begin{cases}
        \base{u}, & \text{ if $u = v$ in $H$, or}\\
        v, & \text{ otherwise.}
    \end{cases}
\end{equation*}

Now, if conditions \cref{item:augmenting_3} or \cref{item:augmenting_2}, let $v^{*} = v'$.
Then, if conditions \cref{item:augmenting_1} or \cref{item:augmenting_4} hold, perform an out-split at $v'$ in $H_{2}$ as follows.
Consider the maximal out-arborescence $T^o$ in $\outtree{v'}$ that is rooted at $\End{Q}$.
This yields, similar as in the previous case, a partition of $\OutNeighbours{H}{v'}$ into two sets $N^{v'}_b$ and $N^{v'}_e$ with $\Abs{N^{v'}_e} \geq 2$ and $\Abs{N^{v'}_b} \geq 1$.
Thus, we can perform an out-split of $v'$.
Let $H^{*}$ be the digraph obtained by the previous out-split and set $v^{*} = \exposed{v'}$.

Observe that $(u^*, v^*) \notin E(H^*)$. 
Indeed, in the case that condition \cref{item:augmenting_3} holds, this is implied since $(u, v) \notin \E{H}$.
In any other case, by the definition of splits, $(u^*, v^*) \notin E(H^*)$.
Then, the digraph $K \coloneqq H^{*} + (u^{*}, v^{*})$ is a basic $H$-augmentation.

Next, we show that $K' \coloneqq H' \cup Q$ is an expansion of $K$ by constructing a model $\mu$ in $D$ such that $\mu(K) = K'$.
For all vertices $w \in \V{H-u-v}$ we choose $\intree{w,K'} \coloneqq \intree{w,H'}$, $\rpath{w,K'} \coloneqq \rpath{w,H'}$ and $\outtree{w,K'} \coloneqq \outtree{w,H'}$ and then define $\mu(w) \coloneqq \intree{w,K'}\cup\rpath{w,K'}\cup\outtree{w,K'}$.
Moreover, for all $(w,x)\in \E{H-u-v}$ we define $\mu((w,x))$ to be the $w$-$x$-bridge of $K'$.
Also, for the newly added edge $\Brace{u^{*}, v^{*}}$, we define $\mu(\Brace{u^{*},v^{*}}) = Q$.

To continue, in case we do not perform an in-split of $u$, we define $\mu(u^{*})$ as follows.
In the case where $\Start{Q} \in \V{\rpath{u,H'}}$, we define $\rpath{u^{*},K'}$ to be $\rpath{u,H'}\Start{Q}$ and $\outtree{u^{*},K'}$ to be $\outtree{u,H'}\cup \Start{Q}\rpath{v,H'}$.
In the case where $\Start{Q} \in \V{\outstar{u,H'}\graphminus(\outtree{u,H'}\cup\rpath{u,H'})}$, we additionally define $\mu$ for $u^{*}$ by adding to $\outtree{u,H'}$ the unique $\V{\outtree{u, H'}}$-$\Start{Q}$-path of $\outstar{u,H'}$, obtaining $\outtree{u^{*},K'}$, and then setting $\mu(u^{*}) \coloneqq \intree{u^{*},K'} \cup \rpath{u^{*}} \cup \outtree{u^{*},K'}$.
Secondly, in the case we perform an in-split at $u$, we define $\mu(u^{*})$ to be the tree $T^i$.
Here we set $\intree{u^{*},K'} \coloneqq T^i$, while both $\rpath{u^{*}, K'}$ and $\outtree{u^{*},K'}$ are the root of $T^i$ and thus trivial.
Moreover, let $R$ be the unique path in $\intree{u,H'}$ starting in the root of $T^i$ and ending in a vertex of in-degree at least two such that all internal vertices of $R$ have in- and out-degree exactly one.
Now let $\intree{\base{u},K'} \coloneqq \intree{u,H'}\graphminus(T^i\cup (R-\End{R}))$.
Also $\rpath{\base{u},K'} = \rpath{u,H'}$, and $\outtree{\base{u},K'}=\outtree{u,H'}$ and finally, $\mu(\base{u}) \coloneqq \intree{\base{u},K'} \cup \rpath{\base{u},K'} \cup \outtree{\base{u},K'}$.

Analogously we now define $\intree{v^{*},K'}$, $\rpath{v^{*},K'}$, $\outtree{v^{*},K'}$ and then $\mu(v^{*})$, depending on whether there was an out-split at $v$ and accordingly $\intree{\base{v},K'}$, $\rpath{\base{v},K'}$, $\outtree{\base{v},K'}$ and then $\mu(\base{v})$. Note that in the case that $u = v$ in $H$ in order to define $\mu(v^{*})$ and $\mu(\base{v})$ we further refine the previously defined $\mu(\base{u})$, since in that case $v'$ becomes $\base{u}$ in $H_{2}$.

As we now have the in- and out-trees for all vertices $w\in \V{K}$, we can obtain $\outstar{w,K'}$ and $\instar{w,K'}$.
Now for every edge $(x,y)\in\E{K}$ such that $\mu((x,y))$ is not yet defined, let $\mu((x,y))$ be the $x$-$y$-bridge of $K'$.

Now observe that $\mu(K) = K'$ and indeed a butterfly-model of $K$ in $D$.
\end{proof}

\subsection{Constructing chain or collarette augmentations}

Next, we show that when finding an \earpath that is properly laced with the model of an edge of the graph we already have an expansion of, we can obtain a chain or a collarette augmentation.

\begin{theorem}
\label{thm:add_chain}
    Let $D,H$ be strongly $2$-connected digraphs, and $H'$ be an $H$-expansion in $D$.
    Let $\Brace{u,v} \in \E{H}$ and $P$ be a path in $D$ with $\Start{P} \in  V(\intree{v}) \cup V(\outstar{v})$, $\End{P} \in V(\instar{u}) \cup V(\outtree{u})$ that is properly laced with the path $Q\coloneqq \outstar{u} \cap \instar{v}$ such that $\V{P}\cap V(H')=V(P\cap Q)\cup \{\Start{P},\End{P}\}$. 
    Then $D$ admits a chain or collarette $H$-augmentation.
\end{theorem}

\begin{figure}[!ht]
    \centering
    \begin{subfigure}[b]{\textwidth}
    	\resizebox{\textwidth}{!}{\includegraphics{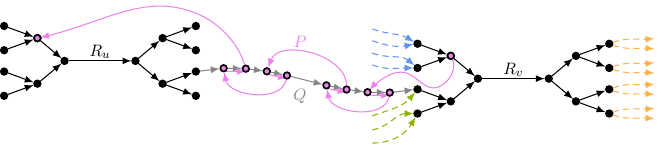}}
        \caption{Adding $P$ to the $H$-expansion $H'$.}
        \label{fig:chain-case-expansion}
    \end{subfigure}\hfill
    \begin{subfigure}[b]{\textwidth}
    \centering
		\includegraphics{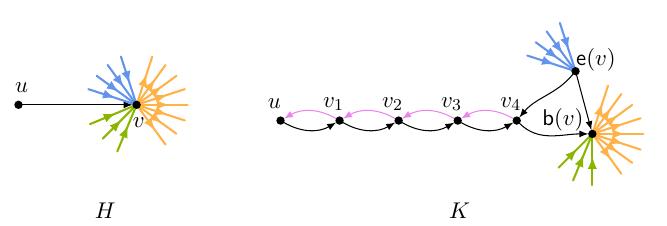}
        \caption{The resulting chain-augmentation $K$ of $H$.}
        \label{fig:chain-case-original}
    \end{subfigure}
    \caption{An example of a chain-augmentation and how it looks in the \subref{fig:chain-case-expansion} expansion and the \subref{fig:chain-case-original} actual digraph.}
    \label{fig:chain-case}
\end{figure}

\begin{proof}
    Our goal is to show that the path $P$ can be used to find a chain or collarette $H$-augmentation $K$ by constructing a $K$-expansion $K'$ as a subgraph of $D$.
    
    Let $Q_1,Q_2,\dots,Q_k$ be the weakly connected components of $P\cap Q$ ordered by occurrence along $Q$ (see~\cref{fig:chain-case-expansion}).

    \paragraph{First assume that $\Abs{\Set{\Start{P},\End{P}} \cap \V{\outstar{v} \cap \intree{u}}} \leq 1$, or both $\Start{P}$ and $\End{P}$ lie in $\outstar{v} \cap \instar{u}$ and $\Start{P} <_{\outstar{v}\cap\instar{u}} \End{P}$.}
    To obtain a graph $H_1$ from $H$, we subdivide the edge $\Brace{u,v}$ $k$ times obtaining the vertices $v_1,v_2,\dots,v_k$ in order of the resulting directed path and then add the edges $\Brace{v_x,v_{x-1}}$ for $x \in [2,k]$ (see~\cref{fig:chain-case-original}).
    
    If $\Start{P} \in V(\outstar{v})$, we add the edge $\Brace{v,v_k}$ to $H_1$ obtaining a graph $H_2$.
    Otherwise, i.e.~if $\Start{P} \in \V{\intree{v} \graphminus \rpath{v}}$, we first perform an in-split at $v$ as follows.
    Consider the maximal in-arborescence $T^i$ in $\intree{v}$ that is rooted at $\Start{P}$.
    This provides a partition of $\InNeighbours{H}{v}$ as follows.
    Let $N^v_e$ be the set of vertices $x \in \InNeighbours{H}{v}$ for which the $x$-$v$-bridge ends in a leaf of $T^i$.
    Also, let $N^v_b\coloneqq \InNeighbours{}{v}\setminus N^v_e$.
    By definition of $\intree{v}$ we obtain $\Abs{N^v_e} \geq 2$ and since $\Start{Q} \notin \V{\rpath{v}}$ we have $\Abs{N^v_b} \geq 1$, which yields the in-split of $v$ into the vertices $v_e$ and $v_b$.
    Finally, we add the edge $\Brace{v_e,v_k}$ and denote the resulting graph by $H_2$.
    
    Similarly, if $\End{P} \in \V{\instar{u}}$, we now add the edge $\Brace{v_1,u}$ to $H_2$ obtaining a graph $H_3$.
    Otherwise, i.e.~if $\End{P} \in \V{\outtree{u}\graphminus \rpath{u}}$, we first perform an out-split in $H_2$ at $u$ as follows.
    Consider the maximal out-arborescence $T^o$ in $\outtree{u}$ that is rooted at $\End{P}$.
    This yields, similar as in the previous case, a partition of $\OutNeighbours{H}{u}$ into two sets $N^u_b$ and $N^u_e$ with $\Abs{N^u_e} \geq 2$ and $\Abs{N^u_b} \geq 1$.
    Thus, we can perform an out-split of $u$ into the vertices $u_e$ and $u_b$ and add the edge $\Brace{v_1,u_e}$.
    We denote the resulting graph by $H_3$.
    
    By construction, this yields a strongly 2-connected graph $K \coloneqq H_3$, which is a chain augmentation of $H$.
    Next, we show that $K' \coloneqq H' \cup P$ is an expansion of $K$ by constructing a model $\mu$ in $D$ such that $\mu(K) = K'$.
    For all vertices $w \in \V{H-u-v}$ we choose $\intree{w,K'} \coloneqq \intree{w,H'}$, $\rpath{w,K'} \coloneqq \rpath{w,H'}$ and $\outtree{w,K'} \coloneqq \outtree{w,H'}$ and then define $\mu(w) \coloneqq \intree{w,K'}\cup\rpath{w,K'}\cup\outtree{w,K'}$.
    Moreover, for all $(w,x)\in \E{H-u-v}$ we define $\mu((w,x))$ to be the $w$-$x$-bridge of $K'$.
    
    First, for every $i\in[k]$, we define $\mu(v_i)=Q_i$.
    In particular, $\rpath{v_i,K'}\coloneqq Q_i$ and $\intree{v_i,K'}$ and $\outtree{v_i,K'}$ are trivial one-vertex trees.
    Also, for every $i\in[2,k]$, we define $\mu((v_{i-1},v_i))$ to be the subpath $\End{Q_{i-1}}P\Start{Q_{i}}$ of $P$, and $\mu((v_{i},v_{i-1}))$ to be the subpath $\End{Q_{i}}Q\Start{Q_{i-1}}$ of $Q$.
    
    We now consider two cases, depending on whether we performed an in-split at $v$ or not.
    Firstly, if $\Start{P} \in \V{\outstar{v,H'}}$, then for the newly added edge $\Brace{v,v_k}$, we define $\mu(\Brace{v,v_k}) = P\Start{Q_1}$.
    In the case where $\Start{P} \in \V{\rpath{v,H'}}$, we define $\rpath{v,K'}$ to be $\rpath{v,H'}\Start{P}$ and $\outtree{v,K'}$ to be $\outtree{v,H'}\cup \Start{P}\rpath{v,H'}$.
    In the case where $\Start{P}\in \V{\outstar{v,H'}\graphminus(\outtree{v,H
    }\cup\rpath{v,H'})}$, we additionally define $\mu$ for $v$ by adding to $\outtree{v,H'}$ the path of $\outstar{v,H'}$ starting from $\outtree{v,H'}$ and ending in $\Start{P}$, obtaining $\outtree{v,K'}$, and then setting $\mu(v)\coloneqq \intree{v,K'}\cup\rpath{v}\cup\outtree{v,K'}$.

    Secondly, in the case we perform an in-split at $v$, we define $\mu(v_e)$ to be the tree $T^i$.
    Here we set $\intree{v_e,K'}\coloneqq T^i$, while both $\rpath{v_e,K'}$ and $\outtree{v_e,K'}$ are the root of $T^i$ and thus trivial.

    Let $R$ be the unique path in $\intree{v,H'}$ starting in the root of $T^i$ and ending in a vertex of in-degree at least two such that all internal vertices of $R$ have in- and out-degree exactly one.
    Now let $\intree{v_b,K'}\coloneqq \intree{v,H'}\graphminus(T^i\cup (R-\End{R}))$.
    Also $\rpath{v_b,K'} = \rpath{v,H'}$, and $\outtree{v_b,K'}=\outtree{v,H'}$ and finally, $\mu(v_b) \coloneqq \intree{v_b,K'} \cup \rpath{v_b,K'} \cup \outtree{v_b,K'}$.

    Analogously we now define $\intree{u,K'}$, $\rpath{u,K'}$, $\outtree{u,K'}$ and then $\mu(u)$, in case there was no out-split at $u$ and otherwise $\intree{u_e,K'}$, $\rpath{u_e,K'}$, $\outtree{u_e,K'}$ and then $\mu(u_e)$ as well as $\intree{u_b,K'}$, $\rpath{u_b,K'}$, $\outtree{u_b,K'}$ and then $\mu(u_b)$.

    As we now have the in- and out-trees for all vertices $w\in \V{K}$, we can obtain $\outstar{w,K'}$ and $\instar{w,K'}$.
    Now for every edge $(x,y)\in\E{K}$ such $\mu((x,y))$ is not yet defined we choose $\mu((x,y))$ to be the $x$-$y$-bridge of $K'$.

    Now observe that $\mu(K) = K'$ is indeed a butterfly-model of $K$ in $D$.

    \paragraph{Now consider that both $\Start{P}$ and $\End{P}$ lie in $\outstar{u} \cap \instar{v}$ and $\End{P} \leq_{\outstar{v}\cap\instar{u}} \Start{P}$.}
        Let $\inP$ be the shortest $\Start{P}$-$Q$-subpath of $P$ and $\outP$ the shortest $Q$-$\End{P}$-subpath of $P$.
        Moreover, let $F\coloneqq \instar{u,H'}\cap\outstar{v,H'}$.
        We distinguish two cases depending on the order of $\End{\inP}$ and $\Start{\outP}$ on $Q$.
        See \cref{fig:collarette_case_and_last_chain_case} for an illustration.

        \begin{figure}[!ht]
            \centering
            \begin{subfigure}[b]{0.49\textwidth}
                \centering
                \begin{tikzpicture}[scale=1]
			
	\pgfdeclarelayer{background}
	\pgfdeclarelayer{foreground}
			
	\pgfsetlayers{background,main,foreground}
			
	\begin{pgfonlayer}{main}

        \node (u) [v:main] {};
        \node (v) [v:main,position=0:26mm from u] {};

        \node (x_1) [v:main,position=315:10mm from u] {};
        \node (x_2) [v:main,position=225:10mm from v] {};

        \node (y_1) [v:main,position=135:10mm from v] {};
        \node (y_2) [v:main,position=45:10mm from u] {};

        \node (y_1label) [v:ghost,position=70:4mm from y_1] {$\End{P}$};
        \node (y_2label) [v:ghost,position=110:4mm from y_2] {$\Start{P}$};
        \node (ulabel) [v:ghost,position=180:3mm from u] {$u$};
        \node (vlabel) [v:ghost,position=0:3mm from v] {$v$};

        \node (Ilabel) [v:ghost,position=110:7.5mm from x_1] {\textcolor{CornflowerBlue}{$I$}};
        \node (Olabel) [v:ghost,position=290:7.5mm from y_1] {\textcolor{LavenderMagenta}{$O$}};

        \node (Flabel) [v:ghost,position=105:5mm from u] {$F$};
        \node (Qlabel) [v:ghost,position=275:5mm from v] {$Q$};

        \node (blocktobottom) [v:ghost,position=270:3mm from x_1] {};
 
	\end{pgfonlayer}
			
	\begin{pgfonlayer}{background}

        \draw[e:main,->,bend right=25] (u) to (x_1);
        \draw[e:main,->,bend right=10] (x_1) to (x_2);
        \draw[e:main,->,bend right=25] (x_2) to (v);

        \draw[e:main,->,bend right=25] (v) to (y_1);
        \draw[e:main,->,bend right=10] (y_1) to (y_2);
        \draw[e:main,->,bend right=25] (y_2) to (u);

        \draw[e:main,->,color=CornflowerBlue] (y_2) to (x_1);
        \draw[e:main,color=LavenderMagenta,->] (x_2) to (y_1);
        
	\end{pgfonlayer}	
			
	\begin{pgfonlayer}{foreground}
	\end{pgfonlayer}
   
\end{tikzpicture}
                \caption{The marked subpaths of $F$ and $Q$ are contractible; they yield the new models of the subdividing vertices of $(v,u)$ and $(u,v)$.}
                \label{fig:collarette_case}
            \end{subfigure}\hfill
            \begin{subfigure}[b]{0.49\textwidth}
                \centering
                \begin{tikzpicture}[scale=1]
			
	\pgfdeclarelayer{background}
	\pgfdeclarelayer{foreground}
			
	\pgfsetlayers{background,main,foreground}
			
	\begin{pgfonlayer}{main}

        \node (u) [v:main] {};
        \node (v) [v:main,position=0:26mm from u] {};

        \node (x_1) [v:main,position=315:10mm from u] {};
        \node (x_2) [v:main,position=225:10mm from v] {};

        \node (y_1) [v:main,position=135:10mm from v] {};
        \node (y_2) [v:main,position=45:10mm from u] {};

        \node (x_1label) [v:ghost,position=250:4mm from x_1] {$\Start{\textcolor{PastelOrange}{P'}}$};
        \node (x_2label) [v:ghost,position=290:4mm from x_2] {$\End{\textcolor{PastelOrange}{P'}}$};
        \node (y_1label) [v:ghost,position=70:4mm from y_1] {$\End{P}$};
        \node (y_2label) [v:ghost,position=110:4mm from y_2] {$\Start{P}$};
        \node (ulabel) [v:ghost,position=180:3mm from u] {$u$};
        \node (vlabel) [v:ghost,position=0:3mm from v] {$v$};

        \node (Ilabel) [v:ghost,position=260:5mm from y_2] {\textcolor{CornflowerBlue}{$I$}};
        \node (Olabel) [v:ghost,position=280:5mm from y_1] {\textcolor{LavenderMagenta}{$O$}};

        \node (Flabel) [v:ghost,position=105:5mm from u] {$F$};
        \node (Qlabel) [v:ghost,position=275:5mm from v] {$Q$};

        \node (PPlabel) [v:ghost,position=215:7mm from y_1] {\textcolor{PastelOrange}{$P'$}};

        \node (blocktobottom) [v:ghost,position=270:3mm from x_1] {};
 
	\end{pgfonlayer}
			
	\begin{pgfonlayer}{background}

        \draw[line width=6pt,opacity=0.6,color=PastelOrange, bend right=15,line cap=round] (y_2) to (x_2);
        \draw[line width=6pt,opacity=0.6,color=PastelOrange, bend right=15,line cap=round] (x_1) to (y_1); 
        \draw[line width=6pt,opacity=0.6,color=PastelOrange, bend right=10,line cap=round] (y_1) to (y_2);

        \draw[e:main,->,bend right=25] (u) to (x_1);
        \draw[e:main,->,bend right=10] (x_1) to (x_2);
        \draw[e:main,->,bend right=25] (x_2) to (v);

        \draw[e:main,->,bend right=25] (v) to (y_1);
        \draw[e:main,->,bend right=10] (y_1) to (y_2);
        \draw[e:main,->,bend right=25] (y_2) to (u);

        \draw[e:main,->,color=CornflowerBlue,bend right=15] (y_2) to (x_2);
        \draw[e:main,color=LavenderMagenta,->,bend right=15] (x_1) to (y_1);
        
	\end{pgfonlayer}	
			
	\begin{pgfonlayer}{foreground}
	\end{pgfonlayer}
   
\end{tikzpicture}
                \caption{We can instead consider the path $P'$ that is properly laced with $F$ and for which $\Start{P'} \leq_{Q} \End{P'}$.}
                \label{fig:reduction_case}
            \end{subfigure}
            \caption{An illustration of the case that both $\Start{P}$ and $\End{P}$ lie in $\outstar{u} \cap \instar{v}$ and $\End{P} \leq_{\outstar{v}\cap\instar{u}} \Start{P}$, in the proof of \cref{thm:add_chain}, and the first segment $\inP$ as well as the last segment $\outP$ of $P$.
            There are two cases depending on the order of the end-vertices of $\inP$ and $\outP$ on $Q$ yielding \subref{fig:collarette_case} a collarette augmentation or \subref{fig:reduction_case} a different path $P'$ that allows us to reduce the situation to a previous case of the proof.}
            \label{fig:collarette_case_and_last_chain_case}
        \end{figure}
        
        \begin{description}
            \item[$\End{\inP}\leq_{Q}\Start{\outP}$:]
                In this case, we obtain $K$ by subdividing the edge $(u,v)$ with the new vertex $x$ and the edge $(v,u)$ with the new vertex $y$ and then adding the edges $(x,y)$ and $(y,x)$.
                Then $K$ is a collarette $H$-augmentation.
                To see that $H'\cup\inP\cup\outP$ is indeed a $K'$ expansion, it suffices to observe that both paths $\End{O} F \Start{I}$ and $\End{\inP}Q\Start{\outP}$ are either singletons or fully butterfly-contractible into a singleton within $K'$.
                Moreover, for every $w\in \V{H}$ it is still possible to contract $\intree{w,K'}\cup\rpath{w,K'}\cup\outtree{w,K'}$ into a single vertex.
                Hence, performing all of these butterfly-contraction leaves us with a subdivision of $K$.

            \item[$\Start{\outP}<_{Q}\End{\inP}$:]
                In this case we consider the path $P' \coloneqq \outP \cdot \End{\outP}F\Start{\inP} \cdot \inP$.
                Note that $P'$ is an $\V{\outstar{u}}$-$\V{\instar{v}}$-path that is properly laced with the path $F$ and otherwise disjoint from $K'$.
                Additionally the vertices $\Start{P'}$ and $\End{P'}$ occur on $Q$ in this order.
                Thus $P'$ and $F$ yield the situation solved in the preceding case.\qedhere
        \end{description}
\end{proof}

\subsection{Constructing bracelet augmentations}

Finally, there are two very specific case in which we end up with a bracelet augmentation.
This is described by the following two \namecref{thm:construct_bracelet}s.
In some cases of these proofs we again obtain basic, chain or collarette augmentations.

\begin{theorem}
    \label{thm:construct_bracelet}
    Let $D,H$ be strongly $2$-connected digraphs and let $H'$ be an $H$-expansion in $D$.
    Let $\Brace{x,w}, \Brace{w,y} \in \E{H}$ and let $P$ be a path in $D$ with $\Start{P} \in V(\outstar{x})$, $\End{P} \in V(\instar{y})$.
    Further, let $Q$ be either the path in $\outstar{a} \cap \instar{w}$ for some $(x, w) \neq (a, w) \in \E{H}$ or the path in $\outstar{w} \cap \instar{b}$ for some $(w,y) \neq (w, b) \in \E{H}$ such that $P$ is properly laced with $Q$, $P\cap Q$ is exactly one path and $P$ is otherwise internally disjoint from $H'$.
    Then $D$ admits either a basic~$H$-augmentation or a bracelet~$H$-augmentation.
\end{theorem}
\begin{proof}
    \setcounter{claimcounter}{0}
    By \cref{lem:inverted_augmentation}, we can assume without loss of generality that $P$ is (properly) laced with the path $Q$ in $\outstar{a} \cap \instar{w}$ for some $(x,w) \neq (a, w) \in \E{H}$.
    Note that all vertices of $Q$ have in- and out-degree one in $H'$.

    If $D$ contains an $H^*$-augmenting path for some $H$-expansion $H^*$, then \cref{thm:add_augmenting_path} implies that $D$ admits a basic $H$-augmentation.
    Therefore, in the rest of the proof, we assume that $D$ does not contain an $H^*$-augmenting path for any $H$-expansion $H^*$.
    We show that the graph $H' \cup P$ is an expansion of the graph obtained from $H$ by a bracelet augmentation with respect to $(x,w), (w,y), (a, w)$, which implies that $D$ admits a bracelet~$H$-augmentation.

    Let $c$ be the first vertex of $P$ on $Q$ and $d$ is the last vertex of $P$ on $Q$, i.e.~$P \cap Q = cQd$.
    Since $\Start{P} \in \V{\outstar{x, H'}}$, $c \in \V{\instar{w, H'}}$ and $\outstar{x, H'} \cup \instar{w, H'} \cup Pc$ is acyclic, the properties~\cref{item:switching_models_1,item:switching_stars_2,item:switching_acyclic_3} hold and the path $Pc$ is $H'$-switching.
    Let $H''$ be the $H$-expansion obtained from $H'$ by switching onto $Pc$ and let $K \subseteq H'$ be the $H''$-switching path obtained by switching onto $Pc$.
    Note that $d \neq \End{K}$ since $\InDegree{H'}{d} = 1$.
    Further, note that ${cQ} = \Brace{\intree{w, H''} \cup \rpath{w, H''}} \graphminus (\intree{w, H'} \cup \rpath{w, H'})$ by construction of $H''$.
    
    \begin{claim}
        $\InDegree{H}{w}= 2$.
    \end{claim}
    \begin{claimproof}
        Note that $\InDegree{H''}{v} = 1$ for any $v \in \V{cQ} \setminus \{c\}$ and $\InDegree{H''}{c}=2$.
        We suppose for a contradiction that $\InDegree{H}{w}>2$.
        Then there exists a vertex in $\V{\intree{w, H''} \graphminus cQ}$ and this implies ${cQ} \subseteq {\intree{w, H''} \graphminus \rpath{w, H''}}$.
        We deduce that $d \in \V{\intree{w, H''} \graphminus \rpath{w, H''}}$.
        Since $\End{P} \in \V{\instar{y, H'}} = \V{\instar{y, H''}}$, $dP$ is $H''$-augmenting of type~\cref{item:augmenting_2}, a contradiction.
    \end{claimproof}
    Then $d, \End{K} \in \V{\rpath{w, H''}}$ since $\InDegree{H''}{c}=2$.
    In particular, $d \in \V{\outstar{w, H''}}$.
    Since $\End{P} \in \V{\instar{y, H'}} = \V{\instar{y, H''}}$ and $\outstar{w, H''} \cup \instar{y, H''} \cup dP$ is acyclic, the properties~\cref{item:switching_models_1,item:switching_stars_2,item:switching_acyclic_3} hold and $dP$ is $H''$-switching.
    Let $H'''$ be the $H$-expansion obtained from $H''$ by switching onto $dP$ and let $L \subseteq H''$ be the $H'''$-switching path obtained by switching onto $dP$.
    Note that $\rpath{w, H'''}= cQd$ by construction of $H''$ and $H'''$.
    \begin{claim}
        $\OutDegree{H}{w}=2$.
    \end{claim}
    \begin{claimproof}
        Note that $\OutDegree{H'''}{v} = 1$ for any $v \in \V{\rpath{w, H''}} \setminus \{d, \End{\rpath{w, H''}}\}$ and $\OutDegree{H'''}{d}=2$ by construction of $H'''$.
        
        We suppose for a contradiction that $\OutDegree{H}{w}>2$.
        Then $d \rpath{w, H''} \subseteq \outtree{w, H'''}$, which implies $\End{K} \in \V{\outtree{w, H'''} \graphminus \rpath{w, H'''}}$.
        Since $\Start{K} \in \V{\outstar{x, H''}}=\V{\outstar{x, H'''}}$, $K$ is $H'''$-augmenting of type~\cref{item:augmenting_4}, a contradiction.
    \end{claimproof}
        Let $y \neq b \in \V{H}$ such that $\OutNeighbours{H}{w}=\Set{y, b}$.
        See \cref{fig:construct_bracelet}.
    
        Note that $\End{K}, \Start{L} \in \V{\instar{b, H'''}}$.
        Further, note that $\End{K} = \Start{\rpath{w, H'}}$ and $\Start{L} = \End{\rpath{w, H'}}$ by construction of $K$ and $L$.
    \begin{claim}
        Either $a = y$ or $(x,a), (a,y) \in \E{H}$.
    \end{claim}
    \begin{claimproof}
        We assume $a \neq y$.
        Note that $d \in \V{\outstar{a, H'}}$ and $\End{P} \in \V{\instar{y, H'}}$.
        If $(a,y) \notin \E{H}$, then $dP$ is $H'$-augmenting of type~\cref{item:augmenting_3}, a contradiction.
        Thus $(a,y) \in \E{H}$, which implies that $dP$ is $H'$-switching since $\outstar{a, H'} \cup \instar{y, H'} \cup dP$ is acyclic.
        Let $\Tilde{H}$ be the $H$-expansion obtained from $H'$ by switching onto $dP$.
        Then $c \in V(\rpath{a, \Tilde{H}} \cup \outtree{a, \Tilde{H}})$ since $\OutDegree{\Tilde{H}}{d}=2$.
        If $(x,a) \notin \E{H}$, then $Pc$ is $\Tilde{H}$-augmenting either of type~\cref{item:augmenting_3} or of type~\cref{item:augmenting_4} since $\Start{Pc} \in \V{\outstar{x, H'}} = \V{\outstar{x, \Tilde{H}}}$, a contradiction.
    \end{claimproof}

    \begin{claim}
        Either $x = b$ or $(x,b), (b,y) \in \E{H}$.
    \end{claim}
    \begin{claimproof}
        We assume $x \neq b$.
        Note that $\Start{K} \in \V{\outstar{x, H'''}}$ and $\End{K} \in \V{\instar{b, H'''}}$.
        If $(x,b) \notin \E{H}$, then $K$ is $H'''$-augmenting of type~\cref{item:augmenting_3}, a contradiction.
        Thus $(x,b) \in \E{H}$, which implies that $K$ is $H'''$-switching since $\outstar{x, H'''} \cup \instar{b, H'''} \cup K$ is acyclic.
        Let $\hat{H}$ be the $H$-expansion obtained from $H'''$ by switching onto $K$.
        Then $\Start{L} \in V(\intree{b, \hat{H}} \cup \rpath{b, \hat{H}})$ since $\InDegree{\hat{H}}{\End{K}} = 2$.
        If $(b,y) \notin \E{H}$, then $L$ is $\hat{H}$-augmenting either of type~\cref{item:augmenting_3} or of type~\cref{item:augmenting_2} since $\End{L} \in \V{\instar{y, H'''}} = \V{\instar{y, \Hat{H}}}$.
    \end{claimproof}

    Thus $H$ satisfies the conditions on $w$ and its neighbours for a bracelet augmentation.
    Let $\Bar{H}$ be the digraph obtained from $H$ by bracelet augmentation with respect to $(x,w), (w, y), (a, w)$.
    By choosing $cQd$ as the \redpath of the subdivison vertex, $H' \cup P$ becomes an $\Bar{H}$-expansion.
\end{proof}

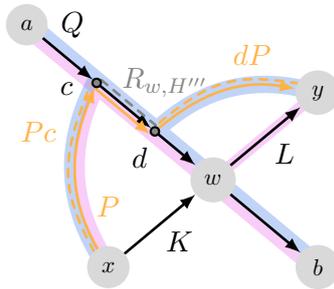
\begin{figure}[!ht]
    \centering
        \begin{tikzpicture}[scale=1]
			
	\pgfdeclarelayer{background}
	\pgfdeclarelayer{foreground}
			
	\pgfsetlayers{background,main,foreground}
			
	\begin{pgfonlayer}{main}

        \node (w) [v:model] {$w$};
        \node (x) [v:model,position=220:18mm from w] {$x$};
        \node (y) [v:model,position=40:18mm from w] {$y$};
        \node (b) [v:model,position=320:18mm from w] {$b$};
        \node (z_1) [v:main,fill=DarkGray,position=140:10mm from w] {};
        \node (z_1g1) [v:ghost,position=240:0.7mm from z_1] {};
        \node (z_1g2) [v:ghost,position=60:0.7mm from z_1] {};
        \node (z_2) [v:main,fill=DarkGray,position=140:10mm from z_1] {};
        \node (z_2g1) [v:ghost,position=240:0.7mm from z_2] {};
        \node (z_2g2) [v:ghost,position=60:0.7mm from z_2] {};
        \node (a) [v:model,position=140:12mm from z_2] {$a$};
        \node (wg1) [v:ghost,position=220:0.7mm from w] {};
        \node (ag1) [v:ghost,position=220:0.7mm from a] {};
        \node (bg1) [v:ghost,position=220:0.7mm from b] {};
        \node (xg1) [v:ghost,position=180:0.7mm from x] {};
        \node (ug1) [v:ghost,position=180:0.7mm from z_2] {};

        \node (wg2) [v:ghost,position=40:0.7mm from w] {};
        \node (ag2) [v:ghost,position=40:0.7mm from a] {};
        \node (bg2) [v:ghost,position=40:0.7mm from b] {};
        \node (xg2) [v:ghost,position=0:0.7mm from x] {};
        \node (ug2) [v:ghost,position=290:0.7mm from z_2] {};
        \node (uug) [v:ghost,position=70:1mm from z_1] {};
        \node (yyg) [v:ghost,position=120:1mm from y] {};
	\end{pgfonlayer}
			
	\begin{pgfonlayer}{background}

        \draw [line width=2.3mm,color=CornflowerBlue!40!white,line cap=round] (z_1.center) to [bend left=30] (y.center);

        \draw [line width=1.5mm,color=LavenderMagenta!40!white,line cap=round] (ag1.center) to (wg1.center)
        (xg2.center) to [bend left=30] (ug2.center)
        (wg1.center) to (bg1.center)
        (w.center) to (y.center);
        
        \draw [line width=1.5mm,color=CornflowerBlue!40!white,line cap=round] (ag2.center) to (wg2.center)
        (wg2.center) to (bg2.center);

        \draw [line width=1.5mm,opacity=0.4,color=CornflowerBlue,line cap=round] (xg1.center) to [bend left=30] (ug1.center);

        \draw[e:main,->] (z_2) to (z_1);
        \draw[e:main,->] (a) to (z_2);
        \draw[e:main,->] (z_1) to (w);
        \draw[e:main,->] (w) to (b);
        \draw[e:main,->] (w) to (y);
        \draw[e:main,->] (x) to (w);

        \draw[e:main,->,out=120,in=248,color=PastelOrange] (x) to (z_2);
        \draw[e:main,->,out=40,in=170,color=PastelOrange] (z_1) to (y);
        
        \draw[e:main,->,color=PastelOrange] (z_2g1) to (z_1g1);
        \draw[e:main,dashed,color=DarkGray] (z_2g2) to (z_1g2);

        \draw[e:main,dashed,out=120,in=248,looseness=1.2,color=PastelOrange] (xg1) to (ug1);

       \draw[e:main,dashed,out=40,in=170,looseness=1.2,color=PastelOrange] (uug) to (yyg);
	\end{pgfonlayer}	
			
	\begin{pgfonlayer}{foreground}

        \node (clabel) [v:ghost,position=190:4mm from z_2] {$c$};
        \node (dlabel) [v:ghost,position=240:4mm from z_1] {$d$};

        \node (Plabel) [v:ghost,position=90:8mm from x] {\textcolor{PastelOrange}{$P$}};

        \node (Rlabel) [v:ghost,position=0:9mm from z_2] {\textcolor{DarkGray}{$R_{w,H'''}$}};
        
        \node (Qlabel) [v:ghost,position=0:6mm from a] {$Q$};

        \node (dPlabel) [v:ghost,position=150:10mm from y] {\textcolor{PastelOrange}{$dP$}};

        \node (Pclabel) [v:ghost,position=220:10mm from z_2] {\textcolor{PastelOrange}{$Pc$}};

        \node (Klabel) [v:ghost,position=20:10mm from x] {$K$};
        \node (Llabel) [v:ghost,position=20:10mm from w] {$L$};
	\end{pgfonlayer}
   
\end{tikzpicture}
    \caption{The $H$-expansions $H', \textcolor{LavenderMagenta}{H''}$ and \textcolor{CornflowerBlue}{$H'''$} in the proof of \cref{thm:construct_bracelet}.}
    \label{fig:construct_bracelet}
\end{figure}

We turn our attention to the second case, which differs from the first one in that $P$ intersects $Q$ in exactly two paths.

\begin{theorem}
    \label{thm:construct_bracelet_two}
    Let $D,H$ be strongly $2$-connected digraphs and let $H'$ be an $H$-expansion in $D$.
    Let $\Brace{x,w}, \Brace{w,y} \in \E{H}$ and let $P$ be a path in $D$ with $\Start{P} \in V(\outstar{x})$, $\End{P} \in V(\instar{y})$.
    Further, let $Q$ be either the path in $\outstar{a} \cap \instar{w}$ for some $(x, w) \neq (a, w) \in \E{H}$ or the path in $\outstar{w} \cap \instar{b}$ for some $(w,y) \neq (w, b) \in \E{H}$ such that $P$ is properly laced with $Q$, $P\cap Q$ is exactly two paths and $P$ is otherwise internally disjoint from $H'$.
    Then $D$ admits either a basic~$H$-augmentation, a chain~$H$-augmentation, a collarette~$H$-augmentation or a bracelet~$H$-augmentation.
\end{theorem}
\begin{proof}
    \setcounter{claimcounter}{0}
    By \cref{lem:inverted_augmentation}, we can assume without loss of generality that $P$ is (properly) laced with the path $Q$ in $\outstar{a} \cap \instar{w}$ for some $(x,w) \neq (a, w) \in \E{H}$.
    Note that all vertices of $Q$ have in- and out-degree one in $H'$.

    If $D$ contains an $H^*$-augmenting path for some $H$-expansion $H^*$, then \cref{thm:add_augmenting_path} implies that $D$ admits a basic $H$-augmentation.
    Therefore, in the rest of the proof, we assume that $D$ does not contain an $H^*$-augmenting path for any $H$-expansion $H^*$.

    Let $c,d,e,f \in V(P) \cap V(Q)$ such that $P \cap Q = cQd \cup eQf$, and $c$ is the first vertex of $P$ on $Q$ and $f$ is the last vertex of $P$ on $Q$.
    Note that $Pc$ is $H'$-switching.
    Let $H''$ be the $H$-expansion obtained by switching onto $Pc$ and let $K \subseteq H'$ be the $H''$-switching path obtained by switching onto $Pc$.

    \begin{claim}\label{clm:a_neq_y}
        Either $a \neq y$ or there exists a chain or collarette~$H$-augmentation.
    \end{claim}
    \begin{claimproof}
        We assume that $a = y$.
        Then $\Start{dP} \in \V{\intree{w,H''}} \cup \V{\outstar{w,H''}}$ and $\End{dP} \in \V{\instar{a, H''}}$.
        Furthermore, $dP$ is properly laced with $Qc$, which implies that $dP$ is properly laced with the path in $\outstar{a, H''} \cap \instar{w, H''}$, and $dP$ is otherwise internally disjoint to $H''$.
        By~\cref{thm:add_chain}, $D$ admits a chain or collarette $H$-augmentation.
    \end{claimproof}
    By~\cref{clm:a_neq_y}, we can assume that $a \neq y$.
    We show that there exists an $H$-expansion $\Tilde{H}$ and a path $\Tilde{P}$ satisfying the conditions of~\cref{thm:construct_bracelet}.
    Then $D$ admits either a basic~$H$-augmentation or a bracelet~$H$-augmentation.
    
    If $(a,y) \notin E(H)$, then $fP$ is $H''$-augmenting of type~\cref{item:augmenting_3}, contradicting our assumption.
    Thus $(a,y) \in E(H)$, which implies that $fP$ is $H''$-switching.
    Let $H'''$ be the $H$-expansion obtained from $H''$ by switching onto $fP$.
    Note that $e \in \V{\outstar{a, H'''}}$ by construction of $H'''$.
    
    \begin{claim}\label{clm:in_degree}
        $\InDegree{H}{w}=2$.
    \end{claim}
    \begin{claimproof}
        Note that $\InDegree{H'''}{v} = 1$ for any $v \in \V{cQ} \setminus \{c\}$ and $\InDegree{H'''}{c}=2$.
        We suppose for a contradiction that $\InDegree{H}{w}>2$.
        Then there exists a vertex in $\V{\intree{w, H'''} \graphminus cQ}$ and this implies ${cQ} \subseteq {\intree{w, H'''} \graphminus \rpath{w, H'''}}$.
        We deduce that $d \in \V{\intree{w, H'''} \graphminus \rpath{w, H'''}}$.
        Thus $dPe$ is $H'''$-augmenting of type~\cref{item:augmenting_1} or type~\cref{item:augmenting_2}, a contradiction.
    \end{claimproof}
    \cref{clm:in_degree} implies that $d \in \V{\rpath{w,H'''}}$.
    Since the path $d P e$ is not $H'''$-augmenting of type~\cref{item:augmenting_4}, $e = \End{d P e} \in \V{\instar{a, H'''}}$.
    In particular, $e \in \V{\rpath{a, H'''}}$ as $e \in \V{\outstar{a, H'''}}$.
    Further, since $d P e$ is not $H'''$-augmenting of type~\cref{item:augmenting_3}, $(a,w) \in E(H)$, which implies that $d P e$ is $H'''$-switching.
    Let $\Tilde{H}$ be the $H$-expansion obtained from $H'''$ by switching onto $d P e$.
    See~\cref{fig:bracelet_two}.
    Further, let $L \subseteq H'''$ be the $\Tilde{H}$-switching path obtained by switching onto $d P e$.
    Note that $\End{L} \in \V{\instar{a,\Tilde{H}}}$.

    \begin{claim}\label{claim:out_degree}
        $\OutDegree{H}{w}=2$.
    \end{claim} 
    \begin{claimproof}
        We remark that $\Start{K} \in \V{\outstar{x,\Tilde{H}}}$.
        We suppose for a contradiction that $\OutDegree{H}{w}>2$.
        Then $\End{K} \in \V{\outtree{w,\Tilde{H}} \graphminus \rpath{w,\Tilde{H}}}$ by construction of $\Tilde{H}$.
        This implies that $K$ is $\Tilde{H}$-augmenting of type~\cref{item:augmenting_4}, a contradiction.
    \end{claimproof}
    By~\cref{claim:out_degree}, $\rpath{w, H'} \subseteq \outstar{w, \Tilde{H}} \cap \instar{y, \Tilde{H}}$.
    Note that $\End{K} = \Start{\rpath{w, H'}}$ and $\End{\rpath{w, H'}} = \Start{L}$.
    Thus $\Tilde{P}\coloneqq K \End{K} \rpath{w, H'} \Start{L} L$ is a path.

    Finally, we show that $\Tilde{P}$ satisfies the conditions of~\cref{thm:construct_bracelet} with respect to the edges $(x,w), (w,a)$ and the path $\Tilde{Q}\coloneqq \outstar{w, \Tilde{H}} \cap \instar{y, \Tilde{H}}$.
    Note that $(w,y) \neq (w,a)$.
    The path $\Tilde{P}$ has the properties $\Start{\Tilde{P}} \in \V{\outstar{x,\Tilde{H}}}$, $\End{\Tilde{P}} \in \V{\instar{a,\Tilde{H}}}$, $\Tilde{P}$ is properly laced with $\Tilde{Q}$, $\Tilde{P} \cap \Tilde{Q}$ is exactly one path, and $\Tilde{P}$ is otherwise internally disjoint to $\Tilde{H}$.
    Thus $\Tilde{P}$ satisfies indeed the conditions of \cref{thm:construct_bracelet}, which implies that $D$ admits either a basic $H$-augmentation or a bracelet $H$-augmentation.    
    \end{proof}

    \begin{figure}[!ht]
    \centering
    \resizebox{!}{5.3cm}{%
     \begin{tikzpicture}[scale=1]
			
	\pgfdeclarelayer{background}
	\pgfdeclarelayer{foreground}
			
	\pgfsetlayers{background,main,foreground}
			
	\begin{pgfonlayer}{main}

        \node (z) [v:main, fill=DarkGray] {};
        \node (g) [v:main, fill=DarkGray, position=0:10mm from z] {};
        \node (d) [v:main, fill=DarkGray, position=135:10mm from z] {};
        \node (c) [v:main, fill=DarkGray, position=135:10mm from d] {};
        \node (f) [v:main, fill=DarkGray, position=135:10mm from c] {};
        \node (e) [v:main, fill=DarkGray, position=135:10mm from f] {};
        
        \node (a) [v:model, position=135:10mm from e] {$a$};
        \node (x) [v:model, position=220:18mm from z] {$x$};
        \node (y) [v:model, position=40:18mm from g] {$y$};

        \node (ee) [v:ghost,position=220:0.7mm from e] {};
        \node (ff) [v:ghost,position=220:0.7mm from f] {};
        \node (cc) [v:ghost,position=220:0.7mm from c] {};
        \node (dd) [v:ghost,position=220:0.7mm from d] {};

        \node (x1) [v:ghost,position=190:0.7mm from x] {};
        \node (x2) [v:ghost,position=80:0.7mm from x] {};
        \node (z1) [v:ghost,position=240:0.7mm from z] {};
        \node (z2) [v:ghost,position=60:0.7mm from z] {};
        \node (c1) [v:ghost,position=210:0.7mm from c] {};
        \node (c2) [v:ghost,position=0:0.7mm from c] {};
        \node (g1) [v:ghost,position=290:0.7mm from g] {};
        \node (g2) [v:ghost,position=90:0.6mm from g] {};
        \node (y1) [v:ghost,position=260:0.7mm from y] {};
        \node (y2) [v:ghost,position=180:1mm from y] {};

\node (ap) [v:ghost,position=90:1mm from a] {};
\node (ep) [v:ghost,position=90:1mm from e] {};
\node (fp) [v:ghost,position=90:1mm from f] {};
\node (yp) [v:ghost,position=120:3mm from y] {};
         
         \node (yright) [v:ghost,position=0:10mm from y] {};
    \end{pgfonlayer}
			
 	\begin{pgfonlayer}{background}    

        \draw [line width=2.3mm,color=CornflowerBlue!40!white,line cap=round] (e.center) to [bend left=45] (d.center);

        \draw [line width=1.5mm,color=LavenderMagenta!40!white,line cap=round] (x1.center) to [bend left=35] (c1.center)
         (ff.center) to  (z1.center)
         (z1.center) to (g1.center)
         (g1.center) to (y1.center);

        \draw [line width=2.3mm,color=LavenderMagenta!40!white,line cap=round]  (g.center) to [out=330,in=270] (yright) to [out=90,in=35](a.center);

        \draw [line width=1.5mm,color=CornflowerBlue!40!white,line cap=round] (x2.center) to [bend left=32] (c2.center)
         (fp.center) to  (z2.center)
         (z2.center) to (g2.center)
         (g2.center) to (y2.center);

        \draw[e:main,->] (a) to (e);
        \draw[e:main,->] (e) to (f);
        \draw[e:main,->] (f) to (c);
        \draw[e:main,->] (c) to (d);
        \draw[e:main,->] (d) to (z);
        \draw[e:main,->] (z) to (g);
        \draw[e:main,->] (x) to (z);
        \draw[e:main,->] (g) to (y);

        \draw[e:main,->,out=120, in=240, color=PastelOrange] (x) to (c);
        \draw[e:main,->,color=PastelOrange] (cc) to (dd);\draw[e:main,<-,out=-5,in=95,color=PastelOrange] (e) to (d);
        \draw[e:main,->,color=PastelOrange] (ee) to (ff);
        \draw[e:main,->,out=20, in=140,color=PastelOrange] (f) to (y);

        \draw[e:main,dashed,->,color=DarkGray] (ap) to (ep);
        \draw[e:main,dashed,->,color=DarkGray] (ep) to (fp);
        \draw[e:main,dashed,->,color=DarkGray,out=20, in=140] (fp) to (yp);

        \draw[e:main,dashed,color=DarkGray] (z1) to (g1);

        \draw[e:main,->] (g) to [out=330,in=270] (yright.center) to [out=90,in=30] (a);
        
    \end{pgfonlayer}

    \begin{pgfonlayer}{foreground}    

        \node (elabel) [v:ghost,position=190:4mm from e] {$e$};
        \node (flabel) [v:ghost,position=190:4mm from f] {$f$};
        \node (clabel) [v:ghost,position=190:4mm from c] {$c$};
        \node (dlabel) [v:ghost,position=190:4mm from d] {$d$};

        \node (Rlabel) [v:ghost,position=320:7mm from z] {\textcolor{DarkGray}{$R_{w,H'}$}};

        \node (Pclabel) [v:ghost,position=150:9mm from x] {\textcolor{PastelOrange}{$Pc$}};

        \node (dPelabel) [v:ghost,position=55:8mm from d] {\textcolor{PastelOrange}{$dPe$}};

        \node (fPlabel) [v:ghost,position=160:15mm from y] {\textcolor{PastelOrange}{$fP$}};

        \node (Qlabel) [v:ghost,position=-10:6mm from a] {$Q$};

        \node (Klabel) [v:ghost,position=20:9mm from x] {$K$};

        \node (Llabel) [v:ghost,position=-20:15mm from g] {$L$};

    \end{pgfonlayer}

\end{tikzpicture}}
        \caption{The $H$-expansions $H', \textcolor{LavenderMagenta}{H'''}$ and \textcolor{CornflowerBlue}{$\Tilde{H}$} in the proof of \cref{thm:construct_bracelet_two}.
        The dashed gray path belongs to $\outstar{a,H'''}$ and $\outstar{a, \Tilde{H}}$.}
        \label{fig:bracelet_two}
\end{figure}
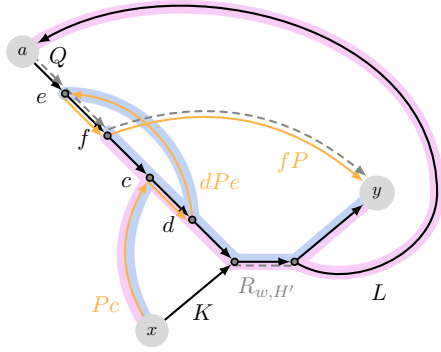

\section{Escapes}
\label{sec:escapes}

In this section, we introduce the two concepts of \emph{\blocking vertices}~(see \cref{fig:blocking_vertex}) and of \emph{escapes}.
For a given strongly $2$-connected digraph $D$ and a strongly $2$-connected butterfly-minor $H$ of $D$, a blocking vertex in a non-trivial branchset of $H$ in $D$ yields a sufficient condition for $D$ admitting an $H$-augmentation.

\begin{definition}[\Blocking vertex]
Let $D$ be a digraph and $H$ be a butterfly-minor of $D$. Let $H'$ be an expansion of $H$ in $D$ and $v \in \V{H}$. We call a vertex $r \in \V{\rpath{v}}$ \emph{\blocking for $v$} if
    \begin{itemize}
      \item there is no switching path starting in $\V{\rpath{v}r} \setminus \Set{r}$ and ending in $\V{H'}\setminus \V{\rpath{v} r}$, and
      \item there is no switching path starting in $\V{H'}\setminus \V{r\rpath{v}}$ and ending in $\V{r \rpath{v}} \setminus \Set{r}$.
    \end{itemize}
\end{definition}

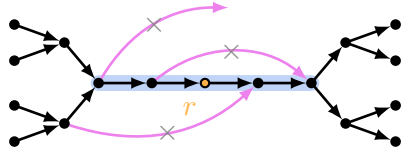
\begin{figure}[!ht]
    \centering
        \begin{tikzpicture}[scale=1]
			
    \pgfdeclarelayer{background}
	\pgfdeclarelayer{foreground}
			
	\pgfsetlayers{background,main,foreground}
			
	\begin{pgfonlayer}{main}

        \node (v) [v:main] {};
        \node (v_1) [v:main, position=0:7mm from v] {};
        \node (v_2) [v:main,fill=PastelOrange, position=0:7mm from v_1] {};
        \node (v_3) [v:main, position=0:7mm from v_2] {};
        \node (v_4) [v:main, position=0:7mm from v_3] {};

        \node (z_1) [v:main, position=50:7mm from v_4] {};
        \node (z_2) [v:main, position=310:7mm from v_4] {};
        \node (w_1) [v:main, position=20:7mm from z_1] {};
        \node (w_2) [v:main, position=340:7mm from z_1] {};
        \node (w_3) [v:main, position=20:7mm from z_2] {};
        \node (w_4) [v:main, position=340:7mm from z_2] {};

        \node (x_1) [v:main, position=130:7mm from v] {};
        \node (x_2) [v:main, position=230:7mm from v] {};
        \node (y_1) [v:main, position=160:7mm from x_1] {};
        \node (y_2) [v:main, position=200:7mm from x_1] {};
        \node (y_3) [v:main, position=160:7mm from x_2] {};
        \node (y_4) [v:main, position=200:7mm from x_2] {};

        \node (a) [v:ghost, position=30:20mm from v] {};
        
        \node (rlabel) [v:ghost, position=240:4mm from v_2] {\textcolor{PastelOrange}{$r$}};

    \end{pgfonlayer}

    \begin{pgfonlayer}{background}

        \draw [line width=6pt,opacity=0.4,color=CornflowerBlue,line cap=round] (v.center) to (v_4.center);
        
        \draw[e:main,->] (y_1) to (x_1);
        \draw[e:main,->] (y_2) to (x_1);
        \draw[e:main,->] (y_3) to (x_2);
        \draw[e:main,->] (y_4) to (x_2);
        \draw[e:main,->] (x_1) to (v);
        \draw[e:main,->] (x_2) to (v);

        \draw[e:main,->] (v) to (v_1);
        \draw[e:main,->] (v_1) to (v_2);
        \draw[e:main,->] (v_2) to (v_3);
        \draw[e:main,->] (v_3) to (v_4);

        \draw[e:main,->] (v_4) to (z_1);
        \draw[e:main,->] (v_4) to (z_2);
        \draw[e:main,->] (z_1) to (w_1);
        \draw[e:main,->] (z_1) to (w_2);
        \draw[e:main,->] (z_2) to (w_3);
        \draw[e:main,->] (z_2) to (w_4);

        \draw[e:main,->,bend left=40,color=LavenderMagenta] (v_1) to node [pos=0.5] (x1) {\textcolor{DarkGray}{$\times$}} (v_4);
        \draw[e:main,->,bend right=30,color=LavenderMagenta] (x_2) to node [pos=0.5] (x2) {\textcolor{DarkGray}{$\times$}} (v_3);
        \draw[e:main,->,bend left=30,color=LavenderMagenta] (v) to node [pos=0.5] (x3) {\textcolor{DarkGray}{$\times$}} (a);
    \end{pgfonlayer}

    \begin{pgfonlayer}{foreground}
    \end{pgfonlayer}

\end{tikzpicture}
    \caption{A \textcolor{CornflowerBlue}{root path} with a \textcolor{PastelOrange}{blocking vertex $r$} and some forbidden \textcolor{LavenderMagenta}{switching paths}.}
    \label{fig:blocking_vertex}
\end{figure}

More precisely, in this section we prove the following \namecref{thm:blocking_vertex_implies_augmentation}.

\begin{theorem} \label{thm:blocking_vertex_implies_augmentation}
    Let $D$ be a strongly $2$-connected digraph and $H$ be a strongly $2$-connected butterfly-minor of $D$.
    Further, let $H'$ be an expansion of $H$ in $D$ and let $v\in \V{H}$ be a vertex whose branchset in $H'$ is non-trivial.
    If there exists a \blocking vertex for $v$, then $D$ admits an~$H$-augmentation.
\end{theorem}

The proof of \cref{thm:blocking_vertex_implies_augmentation} builds on the concept of \emph{escapes}~(see \cref{fig:escape}).

\begin{definition}[Escape]
Let $D$ be a digraph and $H$ be a butterfly-minor of $D$. Let $H'$ be an expansion of $H$ in $D$ and $(u,v) \in \E{H}$.
The path $P$ obtained by concatenation of paths $P_0, P_1, \dots, P_{2k}$ for $k \geq 1$ is called an \emph{$H'$-escape with respect to $(u,v)$} if:
\begin{enumerate}[label=(\alph*)]
    \item $P_0$ is a path with $\End{P_0} \in V(\instar{v} \cap \outstar{u})$ that is either non-parallel $H'$-switching such that $\Start{P_0} \in \V{\Brace{\outstar{x} \cup \intree{v}} \graphminus \rpath{v}}$ for some $(u,v) \neq (x,v) \in \E{H}$ or $H'$-\bad such that $\Start{P_0} \in V(\intree{v} \cup \outstar{v})$,
    \item $P_i$ for $i \in [1, 2k -1]$ odd is a (possibly trivial) path in $\instar{v} \cap \outstar{u}$,
    \item $P_i$ for even $i \in [2, 2k-2]$ is a $H'$-\bad path with $\Start{P_i}, \End{P_i} \in V(\instar{v} \cap \outstar{u})$,
    \item $P_{2k}$ is a path with $\Start{P_{2k}} \in V(\instar{v} \cap \outstar{u})$ that is either non-parallel $H'$-switching such that $\End{P_{2k}} \in V(\Brace{\instar{w} \cup \outtree{u}} \graphminus R_{u})$ for some $(u,v) \neq (u,w) \in \E{H}$ or $H'$-\bad such that $\End{P_{2k}} \in V(\instar{u} \cup \outtree{u})$,
    \item the path $P$ and the path $\instar{v} \cap \outstar{u}$ are properly laced, and
    \item either $\Start{P_0} \in \V{\instar{v} \graphminus \outstar{u}}$ or $\End{P_{2k}} \in \V{\outstar{u} \graphminus \instar{v}}$.
\end{enumerate}
We call $k$ the \emph{duration} of $P$ and refer to it as $\duration{P}$.
\end{definition}
\noindent We remark that if $P$ is an $H'$-escape with respect to $(u,v)$, then $\Inverted{P}$ is an $\Inverted{H'}$-escape with respect to $(v,u)$.

In \cref{subsec:exist_escapes}, we prove the following sufficient condition for the existence of escapes.
\begin{restatable}{lemma}{ExistenceOfEscape} \label{lem:existence_of_escape}
Let $D$ be a strongly $2$-connected digraph, and let $H$ be a strongly $2$-connected butterfly-minor of $D$.
Also, let $H'$ be an expansion of $H$ in $D$ and let $v\in \V{H}$.
If there is a \blocking vertex $r$ for $v$ such that $V(\intree{v}) \setminus \Set{r} \neq \emptyset$, then either $D$ admits a basic $H$-augmentation or there is an $H'$-escape with respect to $(u,v)$ for some $u \in \InNeighbours{H}{v}$. 
\end{restatable}
The proof of \cref{lem:existence_of_escape} follows the following idea:
Since $D$ is strongly $2$-connected, there exists a path in $D -r$ starting in $V(\intree{v}) \setminus \{r\}$ and ending outside $\V{\instar{v}}$.
We show that this path contains either a basic~$H$-augmentation or an $H'$-escape.

In \cref{subsec:escape-augmentations}, we prove that escapes provide augmentations.
\begin{restatable}{lemma}{NonProperEscape} \label{lem:non-proper_escape}
    Let $D$ be a strongly $2$-connected digraph and $H$ be a strongly $2$-connected butterfly-minor of $D$.
    Let $H'$ be an $H$-expansion in $D$ with an $H'$-escape of duration $1$.
    Then $D$ admits  an $H$-augmentation.
\end{restatable}

\begin{restatable}{lemma}{ProperEscape} \label{lem:proper_escape}
Let $D$ be a digraph and $H$ be a butterfly-minor of $D$. Let $H'$ be an expansion of $H$ in $D$ with an $H'$-escape of duration at least $2$.
Then $D$ admits an $H$-augmentation.
\end{restatable}

Assuming \cref{lem:existence_of_escape,lem:non-proper_escape,lem:proper_escape} to be true we can prove \cref{thm:blocking_vertex_implies_augmentation}.
\begin{proof}[Proof of \cref{thm:blocking_vertex_implies_augmentation}]
Since the branchset of $v$ is not a singleton, either $V(\intree{v, H'}) \setminus \{r\} \neq \emptyset$ or $V(\outtree{v, H'}) \setminus \{r\} \neq \emptyset$.
In the former case, we apply \cref{lem:existence_of_escape} to obtain that either $D$ admits a basic~$H$-augmentation or there exists an $H'$-escape.
If there is an $H'$-escape, then \cref{lem:non-proper_escape,lem:proper_escape} imply that $D$ admits an $H$-augmentation.

In the latter case, $r$ is a \blocking vertex for $v$ in $\Inverted{H}$ and $\V{\intree{v, \Inverted{H}'}} \setminus \{r\} \neq \emptyset$.
Thus we can apply \cref{lem:existence_of_escape,lem:non-proper_escape,lem:proper_escape} to $\Inverted{H}$ to obtain that $\Inverted{D}$ admits a $\Inverted{H}$-augmentation.
Then $D$ admits an $H$-augmentation, by \cref{lem:inverted_augmentation}.
\end{proof}

\subsection{Sufficient condition for the existence of escapes}\label{subsec:exist_escapes}
We prove the following sufficient condition for the existence of escapes.

\ExistenceOfEscape*

\begin{proof}
\setcounter{claimcounter}{0}
We start by recalling that if $D$ contains an $H'$-augmenting path, then \cref{thm:add_augmenting_path} implies that $D$ admits a basic $H$-augmentation.
Therefore, in the rest of the proof, we assume that $D$ does not contain any $H'$-augmenting path, and under this assumption, we show the existence of an escape with respect to $(u,v)$ for some $u\in \InNeighbours{H}{v}$.

We start by considering the disjoint, non-empty sets $A\coloneqq  V(\Brace{\intree{v} \cup \rpath{v} r} - r)$ and $B\coloneqq  V(\Brace{H' \graphminus \instar{v}} \cup (r \rpath{v}- r))$. See~\cref{fig:escape}.
Keep in mind that $A\subseteq  V(\instar{v} \graphminus r \rpath{v})$. As $D$ is strongly $2$-connected there exists an $A$-$B$-path in $D - r$.
We consider such an $A$-$B$-path $Q$ with $E(Q) \setminus E(H')$ being minimal and show that some terminal segment of $Q$ forms the desired escape.

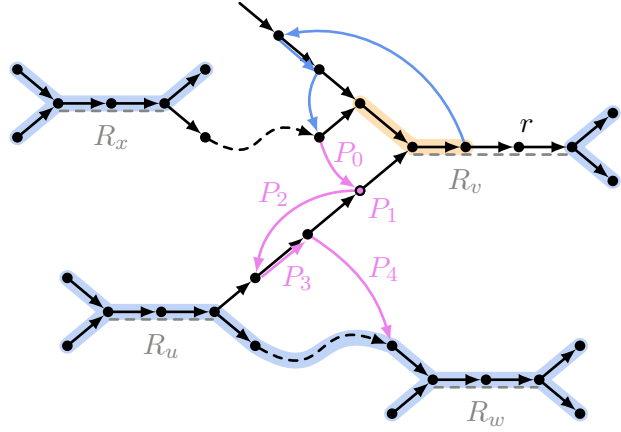
\begin{figure}[!ht]
    \centering
    \begin{tikzpicture}[scale=1]
			
	\pgfdeclarelayer{background}
	\pgfdeclarelayer{foreground}
			
	\pgfsetlayers{background,main,foreground}
			
	\begin{pgfonlayer}{main}
        \node (v) [v:main] {};
        \node (vg) [v:ghost,position=270:1mm from v] {};
        \node (v_1) [v:main,position=0:7mm from v] {};
        \node (v_2) [v:main,position=0:7mm from v_1] {};
        \node (v_3) [v:main,position=0:7mm from v_2] {};
        \node (v_3g) [v:ghost,position=270:1mm from v_3] {};
        \node (vb_1) [v:main,position=40:7mm from v_3] {};
        \node (vb_2) [v:main,position=320:7mm from v_3] {};
        \node (va_1) [v:main,position=140:9mm from v] {};
        \node (va_2) [v:main,position=220:9mm from v] {};
        \node (z_1) [v:main,position=140:7mm from va_1] {};
        \node (z_2) [v:main,position=140:7mm from z_1] {};
        \node (z_3) [v:ghost,position=140:7mm from z_2] {};

        \node (z_1g) [v:ghost,position=270:0.7mm from z_1] {};
        \node (z_2g) [v:ghost,position=180:0.6mm from z_2] {};

         \node (va_2g) [v:main,fill=LavenderMagenta,position=220:9mm from v] {};
         
        \node (q) [v:main,position=220:7mm from va_1] {};
        \node (xb_2) [v:main,position=180:15mm from q] {};
        \node (x_2) [v:main,position=140:7mm from xb_2] {};
        \node (xb_1) [v:main,position=40:7mm from x_2] {};
        \node (x_1) [v:main,position=180:7mm from x_2] {};
        \node (x_0) [v:main,position=180:7mm from x_1] {};
        \node (xa_1) [v:main,position=140:7mm from x_0] {};
        \node (xa_2) [v:main,position=220:7mm from x_0] {};

        \node (x_0g) [v:ghost,position=270:1mm from x_0] {};
        \node (x_2g) [v:ghost,position=270:1mm from x_2] {};
        
        \node (p) [v:main,position=220:9mm from va_2] {};
        \node (ub_1) [v:main,position=220:9mm from p] {};
        \node (u_2) [v:main,position=220:7mm from ub_1] {};
        \node (ub_2) [v:main,position=320:7mm from u_2] {};
        \node (u_1) [v:main,position=180:7mm from u_2] {};
        \node (u_0) [v:main,position=180:7mm from u_1] {};
        \node (ua_1) [v:main,position=140:7mm from u_0] {};
        \node (ua_2) [v:main,position=220:7mm from u_0] {};

        \node (u_0g) [v:ghost,position=270:1mm from u_0] {};
        \node (u_2g) [v:ghost,position=270:1mm from u_2] {};

        \node (pg) [v:ghost,position=270:0.7mm from p] {};
        \node (ub_1g) [v:ghost,position=0:0.7mm from ub_1] {};
        
        \node (wa_1) [v:main,position=0:18mm from ub_2] {};
        \node (w_0) [v:main,position=320:7mm from wa_1] {};
        \node (wa_2) [v:main,position=220:7mm from w_0] {};
        \node (w_1) [v:main,position=0:7mm from w_0] {};
        \node (w_2) [v:main,position=0:7mm from w_1] {};
        \node (wb_1) [v:main,position=40:7mm from w_2] {};
        \node (wb_2) [v:main,position=320:7mm from w_2] {};

        \node (w_0g) [v:ghost,position=270:1mm from w_0] {};
        \node (w_2g) [v:ghost,position=270:1mm from w_2] {};

        \node (Rxlabel) [v:ghost, position=270:4.5mm from x_1] {\textcolor{DarkGray}{$R_x$}};

        \node (Rulabel) [v:ghost, position=270:4.5mm from u_1] {\textcolor{DarkGray}{$R_u$}};

        \node (Rwlabel) [v:ghost, position=270:4.5mm from w_1] {\textcolor{DarkGray}{$R_w$}};

        \node (Rvlabel) [v:ghost, position=270:4.5mm from v_1] {\textcolor{DarkGray}{$R_v$}};

        \node (rlabel) [v:ghost, position=70:3mm from v_2] {$r$};

        \node (P0label) [v:ghost, position=330:4.5mm from q] {\textcolor{LavenderMagenta}{$P_0$}};
        
        \node (P1label) [v:ghost, position=320:4mm from va_2] {\textcolor{LavenderMagenta}{$P_1$}};

        \node (P2label) [v:ghost, position=130:7mm from p] {\textcolor{LavenderMagenta}{$P_2$}};
        
        \node (P3label) [v:ghost, position=0:5.5mm from ub_1] {\textcolor{LavenderMagenta}{$P_3$}};

        \node (P4label) [v:ghost, position=335:11mm from p] {\textcolor{LavenderMagenta}{$P_4$}};
	\end{pgfonlayer}
			
	\begin{pgfonlayer}{background}
        \draw [line width=6pt,opacity=0.4,color=PastelOrange,line cap=round] (va_1.center) to (v.center) to (v_1.center);
        
        \draw [line width=6pt,opacity=0.4,color=CornflowerBlue,line cap=round] (xa_1.center) to (x_0.center)
        (xa_2.center) to (x_0.center)
        (x_0.center) to (x_2.center)
        (x_2.center) to (xb_1.center);

        \draw [line width=6pt,opacity=0.4,color=CornflowerBlue,line cap=round] (ua_1.center) to (u_0.center)
        (ua_2.center) to (u_0.center)
        (u_0.center) to (u_2.center)
        (u_2.center) to (ub_2.center)
        (wa_1.center) to (w_0.center)
        (wa_2.center) to (w_0.center)
        (w_0.center) to (w_2.center)
        (w_2.center) to (wb_1.center)
        (w_2.center) to (wb_2.center)
        (ub_2) to [out=330,in=160,looseness=1.6] (wa_1);

        \draw [line width=6pt,opacity=0.4,color=CornflowerBlue,line cap=round] (vb_1.center) to (v_3.center)
        (vb_2.center) to (v_3.center);
  
        \draw[e:main,->] (v) to (v_1);
        \draw[e:main,->] (v_1) to (v_2);
        \draw[e:main,->] (v_2) to (v_3);
        \draw[e:main,->] (v_3) to (vb_1);
        \draw[e:main,->] (v_3) to (vb_2);
        \draw[e:main,->] (va_1) to (v);
        \draw[e:main,->] (va_2) to (v);
        \draw[e:main,->] (z_1) to (va_1);
        \draw[e:main,->] (z_2) to (z_1);
        \draw[e:main,->] (z_3) to (z_2);

        \draw[e:main,->] (q) to (va_1);

        \draw[e:main,dashed,->] (xb_2) to [out=330,in=160,looseness=1.6] (q);
        \draw[e:main,->] (x_2) to (xb_2);
        \draw[e:main,->] (x_2) to (xb_1);
        \draw[e:main,->] (x_1) to (x_2);
        \draw[e:main,->] (x_0) to (x_1);
        \draw[e:main,->] (xa_1) to (x_0);
        \draw[e:main,->] (xa_2) to (x_0);

        \draw[e:main,->] (p) to (va_2);
        \draw[e:main,->] (ub_1) to (p);
        \draw[e:main,->] (u_2) to (ub_1);
        \draw[e:main,->] (u_2) to (ub_2);
        \draw[e:main,->] (u_1) to (u_2);
        \draw[e:main,->] (u_0) to (u_1);
        \draw[e:main,->] (ua_1) to (u_0);
        \draw[e:main,->] (ua_2) to (u_0);
        \draw[e:main,dashed,->,out=330,in=160,looseness=1.6] (ub_2) to (wa_1);
        \draw[e:main,->] (wa_1) to (w_0);
        \draw[e:main,->] (wa_2) to (w_0);
        \draw[e:main,->] (w_0) to (w_1);
        \draw[e:main,->] (w_1) to (w_2);
        \draw[e:main,->] (w_2) to (wb_1);
        \draw[e:main,->] (w_2) to (wb_2);

        \draw[e:main,dashed,color=DarkGray] (vg) to (v_3g);
        \draw[e:main,dashed,color=DarkGray] (u_0g) to (u_2g);
        \draw[e:main,dashed,color=DarkGray] (x_0g) to (x_2g);
        \draw[e:main,dashed,color=DarkGray] (w_0g) to (w_2g);

        \draw[e:main,->,bend right=40,color=CornflowerBlue] (v_1) to (z_2);
        \draw[e:main,->,color=CornflowerBlue] (z_2g) to (z_1g);
        \draw[e:main,->,out=240,in=110,color=CornflowerBlue] (z_1) to (q);

        \draw[e:main,->,bend right=20,color=LavenderMagenta] (q) to (va_2);
        \draw[e:main,->,bend right=40,color=LavenderMagenta] (va_2) to (ub_1);
        \draw[e:main,->,color=LavenderMagenta] (ub_1g) to (pg);
        \draw[e:main,->,bend left=20,color=LavenderMagenta] (p) to (wa_1);
        
	\end{pgfonlayer}	
			
	\begin{pgfonlayer}{foreground}
	\end{pgfonlayer}
   
\end{tikzpicture}
    \caption{An \textcolor{LavenderMagenta}{escape} with respect to $(u,v)$ of duration $2$.
        The vertex $r$ is \blocking and the \textcolor{CornflowerBlue}{lightblue path} is the initial segment $Q \Start{P}$ of the \textcolor{PastelOrange}{$A$}-\textcolor{CornflowerBlue}{$B$}-path~$Q$ in the proof of \cref{lem:existence_of_escape}. }
    \label{fig:escape}
\end{figure}

We now prove a series of properties of $Q$.
First, we observe that $Q$ starts and ends with $H'$-\earpaths, i.e.~the first and last edge of $Q$ are not contained in $E(H')$.
This follows from the fact that 
there is neither an $A$-$\Brace{V(H') \setminus A}$-path nor an $\Brace{V(H') \setminus B}$-$B$-path in $H' - r$.

\begin{claim}
    \label{claim:escapes-bad-non-par}
    Any $H'$-\earpath $\Tilde{Q}$ contained in $Q$ is either $H'$-\bad or non-parallel $H'$-switching.
\end{claim}
\begin{claimproof}
Let $\Tilde{Q}$ be an $H'$-\earpath contained in $Q$.
By assumption, $D$ does not contain any $H'$-augmenting path.
Thus, due to~\cref{lem:different-earpaths},
we have to prove that $\Tilde{Q}$ is not parallel $H'$-switching.
Suppose the contrary, towards a contradiction.

By definition of parallel $H'$-switching path, there exists $(y,z) \in E(H)$ such that there is an $\Start{P_0}$-$\End{P_0}$-path $\Upsilon$ in $\outstar{y} \cup \instar{z} \cup \Set{e~|~\Head{e}\in \V{\instar{z}},\Tail{e}\in \V{\outstar{y}}}$.
If $\Upsilon$ is disjoint from $r$, then the digraph obtained from $Q$ by replacing $\Tilde{Q}$ by $\Upsilon$
is disjoint from $r$ and
contains an $A$-$B$-path $Q'$ where $E(Q')\setminus E(H')$ is smaller than $E(Q)\setminus E(H')$, contradicting the minimality of $Q$.

Thus we can assume that $r \in \V{\Upsilon}$.
Then either $y = v$, $z \in \OutNeighbours{H}{v}$,
$\Start{\Tilde{Q}} \in \V{\rpath{v}r - r }$ and $\End{\Tilde{Q}} \in \V{(r \rpath{v} \cup \outtree{v} - r) \cup \instar{z} }$
or $y \in \InNeighbours{H}{v}$, $z=v$, 
$\Start{\Tilde{Q}} \in \V{\outstar{y} \cup (\intree{v} \cup \rpath{v} r - r)}$ and $\End{\Tilde{Q}} \in \V{r \rpath{v} - r}$.
In both cases, we obtain a contradiction to $r$ being \blocking.
\end{claimproof}

\begin{claim}
    $Q$ is not an $H'$-\earpath.
\end{claim}
\begin{claimproof}
    By~\cref{claim:escapes-bad-non-par}, it suffices to show that there is neither an $H'$-switching nor an $H'$-\bad path starting in $A$ and ending in $B$.
    
    Note that $A = \V{\intree{v} \cup (\rpath{v} r - r)}$.
    Any $H'$-\bad-path starting in $A$ ends in $\V{\instar{v}}$, which is disjoint from $B$.

    Any $H'$-switching path starting in $\V{\intree{v} \graphminus \rpath{v}}$ ends in $\V{\instar{v}}$.
    Since $r$ is \blocking, no $H'$-switching path starting in $\V{\intree{v} \graphminus \rpath{v}}$ ends in $\V{r \rpath{v} -r} = B \cap \V{\instar{v}}$.
    Thus it remains to consider $H'$-switching $\V{\rpath{v} r - r}$-$B$-paths.
    Since $r$ is \blocking and $B \subseteq V(H') \setminus \V{\rpath{v} r}$, no such path exists.
    Thus $Q$ is not an $H'$-\earpath.
\end{claimproof}

We can conclude that $Q$ is the concatenation of paths $Q_0, \dots, Q_{2 \ell}$ for some $\ell \in \N$ such that:
\begin{itemize}
    \item for $i \in [0,2k]$ even $Q_i$ is either a bad $H'$-\earpath or a non-parallel switching $H'$-\earpath,
    \item for $i \in [0,2k]$ odd $Q_i$ is a (possibly trivial) path in $H'$.
\end{itemize}

We are now ready to define the terminal segment $P$ of $Q$, which forms the desired escape.
If for every even $i \in [0,2 \ell -2]$ the path $Q_i$ is $H'$-bad, set $P\coloneqq Q$.
Otherwise let $i \in [0,2 \ell -2]$ be even and maximal such that $Q_i$ is non-parallel $H'$-switching and set $P \coloneqq \Start{Q_i} Q $.

By definition, $P$ is a concatenation of paths $P_0, \dots, P_{2 k}$ for some $k \in \N$ such that:
\begin{itemize}
    \item $P_i$ is either a bad $H'$-\earpath or a non-parallel switching $H'$-\earpath for $i \in \Set{0, 2k}$,
    \item $P_i$ is a bad $H'$-\earpath for $i \in [2, 2k-2]$ even,
    \item $P_i$ is a (possibly trivial) path in $H'$ for $i \in [0,2k]$ odd.
\end{itemize}

Note that $H' - (\{r\} \cup A \cup B)$ is the disjoint union of the paths in $\outstar{x} \cap \instar{v}$ for $x \in \InNeighbours{H}{v}$.
Let $u \in \InNeighbours{H}{v}$ such that $\End{P_0} \in \V{\outstar{u} \cap \instar{v}}$.
\begin{claim} \label{clm:internal_paths}
    $\Start{P_i}, \End{P_i} \in \V{\outstar{u} \cap \instar{v}}$ for any $i \in [1,2 k -1]$.
\end{claim}
\begin{claimproof}
    By induction over $i \in [1, 2 k -1]$.
    
    We assume that $\End{P_{i - 1}} = \Start{P_i} \in \V{\outstar{u} \cap \instar{v}}$ for some $i \in [1,2 k -1]$.
    If $i$ is odd, $P_i$ is a path in $H' - (\{r\} \cup A \cup B)$.
    Then $\End{P_i} \in \V{\outstar{u} \cap \instar{v}}$ since $H' - (\{r\} \cup A \cup B)$ is the disjoint union of the paths in $\outstar{x} \cap \instar{v}$ for $x \in \InNeighbours{H}{v}$.

    If $i$ is even, $P_i$ is a bad path with $\End{P_i} \in \V{H' - (\{r\} \cup A \cup B)}$.
    Then there exists an $\End{P_i}$-$\Start{P_i}$-path in $\instar{y} \cup \outstar{y}$ for some $y \in \V{H}$.
    Since $\Start{P_i} \in \V{\outstar{u} \cap \instar{v}}$, $y \in \{u,v\}$.
    If $y = u$, $\End{P_i} \in \V{\instar{v} \cup \outstar{u}}$ since $\V{\instar{u} \cup \outstar{u}} \setminus B = \V{\instar{v} \cap \outstar{u}}$ and $\End{P_i} \notin B$.
    If $y = v$, $\End{P_i} \in \V{\instar{v} \cap \outstar{u}}$ since only the vertices in $\V{\instar{v} \cap \outstar{v}}$ have a path to $\Start{P_i}$ in $\instar{v} \cup \outstar{v}$.
\end{claimproof}

We show that $P$ is indeed an escape with respect to $(u,v)$:
\begin{enumerate}[label=(\alph*)]
    \item\label{item:check_escape_1} $P_0$ is a path with $\Start{P_0} \in \V{H'-r} \setminus B$ and $\End{P_0} \in V(\instar{v} \cap \outstar{u})$.
    If $P_0$ is non-parallel $H'$-switching, then $\Start{P_0} \in \V{\Brace{\outstar{x} \cup \intree{v}} \graphminus \rpath{v}}$ for some $(u,v) \neq (x,v) \in \E{H}$ since $\Start{P_0} \in \V{H' - r} \setminus B$.
    If $P_0$ is bad, then $P_0 = Q_0$ and thus $\Start{P_0} \in A$.
    This implies that $\Start{P_0} \in V(\intree{v} \cup \rpath{v}) \subseteq \V{\intree{v} \cup \outstar{v}}$,
    \item For $i \in [1, 2k -1]$ odd, $P_i$ is a path in $H' - A - B - r$.
    Since $\Start{P_i}, \End{P_i} \in \V{\outstar{u} \cap \instar{v}}$ by \cref{clm:internal_paths}, $P_i$ is a path in $\outstar{u} \cap \instar{v}$.
    \item For $i \in [2, 2k -2]$ even, $P_i$ is a $H'$-bad path with $\Start{P_i}, \End{P_i} \in \V{\outstar{u} \cap \instar{v}}$ by \cref{clm:internal_paths}.
    \item $P_{2k}$ is a path with $\Start{P_{2k}} = \End{P_{2k-1}} \in \V{\outstar{u} \cap \instar{v}}$ and $\End{P_{2k}} \in B$ that is either $H'$-bad or non-parallel $H'$-switching.
    If $P_{2k}$ is $H'$-bad, then $\End{P_{2k}} \in \V{\instar{u} \cup \outtree{u}}$ since $\End{P_{2k}} \in B$.
    If $P_{2k}$ is non-parallel $H'$-switching, then $\End{P_{2k}} \in \V{\instar{w} \cup (\outtree{u} \graphminus \rpath{u})}$ for some $(u,v) \neq (u,w) \in E(H)$ or $\End{P_{2k}} \in \V{\instar{v} \graphminus \rpath{v}}$.
    As $B \cap \V{\instar{v} \graphminus \rpath{v}} = \emptyset$ and $\End{P_{2k}} \in B$, $\End{P_{2k}} \in \V{(\outtree{u} \graphminus \rpath{u}) \cup \instar{w}}$ for some $(u,v) \neq (u,w) \in E(H)$.
    \item By \cref{lem:laced} and by minimality of $Q$, $P$ is laced with the path in $\outstar{u} \cap \instar{v}$.
    Furthermore, it is properly laced since $P \cap \outstar{u} \cap \instar{v} \neq \emptyset$,
    \item We proved in \labelcref{item:check_escape_1} that $\Start{P_0} \in \V{\instar{v} \graphminus \outstar{u}}$. \qedhere
\end{enumerate}

\end{proof}

\subsection{Escapes give augmentations}
\label{subsec:escape-augmentations}
The following \namecref{lem:non-proper_escape} shows how escapes of duration $1$ can be used to construct $H$-augmentations.

\NonProperEscape*

\begin{proof}
    Let $P$ be an $H'$-escape of duration $1$ with respect to $(u,v) \in E(H)$.
    If $\Start{P_0} \in V(\intree{v, H'} \cup \outstar{v, H'})$ and $\End{P_2} \in V(\instar{u, H'} \cup \outtree{u, H'})$, we obtain either a collarette~$H$-augmentation or a chain~$H$-augmentation by \cref{thm:add_chain}.
    Since $\Inverted{P}$ is an escape of duration $1$ in $\Inverted{H'}$ and by~\cref{lem:inverted_augmentation}, we can assume without loss of generality that $\End{P_2} \notin V(\instar{u, H'} \cup \outtree{u, H'})$, i.e.~$\End{P_2} \in V(\instar{w, H'})$ for some $(u,v) \neq (u, w) \in E(H)$, which implies that $P_2$ is non-parallel $H'$-switching.

    In the following, we consider $H$-expansions obtained from $H'$ by either switching onto $P_2$ or, if $P_0$ is non-parallel $H'$-switching, onto $P_0$.
    For a detailed explanation of how $H'$ changes under switching, we refer to the proof of~\cref{lem:expansion_stays_expansion_after_switching}.

    Let $H''$ be the $H$-expansion obtained from $H'$ by switching onto $P_2$.
    By construction of $H''$, $\End{P_0} \in V(R_{u, H''} \cup \outtree{u, H''})$ holds.
    Note that $\Start{P_0} \in \V{\intree{v, H'} \cup \outstar{v, H'} \cup \outstar{x, H'}}$ for some $(u,v) \neq (x, v) \in E(H)$ by definition of escape.
    \begin{description}
        \item[Case 1 -- $\Start{P_0} \in V(\intree{v, H'} \graphminus \outstar{v, H'})$:]
            Note that $\intree{v, H'} \graphminus \outstar{v, H'} = \intree{v, H''} \graphminus \outstar{v, H''}$ by construction of $H''$.
            Thus $\Start{P_0} \in V(\intree{v, H''} \graphminus \outstar{v, H''})$, which implies that $P_0$ is $H''$-augmenting of type~\cref{item:augmenting_1} or \cref{item:augmenting_2}.
        \item[Case 2 -- $\Start{P_0} \in V(\outstar{v, H'})$:]
            We begin by showing $\Start{P_0} \in V(\outstar{v, H''})$.
            If $\Start{P_0} \in V(\rpath{v, H'})$, then $\Start{P_0} \in V(\rpath{v, H''})$ since $\rpath{v, H'} = \rpath{v, H''}$ by construction of $H''$.
            If $\Start{P_0} \in V(\outstar{v, H'} \graphminus \rpath{v, H'})$, then $\End{P_2} \in \V{\outstar{u, H'} \graphminus \instar{v, H'}}$ by definition of escape.
            Under the condition $\End{P_2} \in \V{\outstar{u, H'}\graphminus \instar{v, H'}}$, $\outstar{v, H'} = \outstar{v, H''}$ holds by construction of $H''$.
            Thus $\Start{P_0} \in V(\outstar{v, H''})$.

            If $(v,u) \notin \E{H}$, then $P_0$ is $H''$-augmenting of type~\cref{item:augmenting_2} for $\End{P_0} \in V(R_{u, H''})$, and $H''$-augmenting of type~\cref{item:augmenting_4} for $\End{P_2} \in V(\outtree{u, H''} \graphminus \rpath{u, H''})$.
            
            Thus, by \cref{thm:add_augmenting_path}, we assume that $(v,u) \in \E{H}$ and show that we can apply \cref{thm:construct_bracelet} to $P$ in $H'$.
            Then $D$ admits either a basic~$H$-augmentation or a bracelet~$H$-augmentation.
            See~\cref{fig:escape_non_proper}.
            
            The path $P$ starts in $\Start{P_0} \in V(\outstar{v, H'})$ and ends in $\End{P_2} \in V(\instar{w, H'})$.
            Furthermore there are edges $(v,u), (u,w) \in \E{H}$ and an edge $(u,v) \in \E{H}$ such that $P$ is properly laced with the path $Q$ in $\outstar{u,H'} \cap \instar{v, H'}$, $P\cap Q$ is exactly one path and $P$ is otherwise internally disjoint to $H'$.
            Thus we can apply \cref{thm:construct_bracelet}.
        \item[Case 3 -- $\Start{P_0} \in V(\outstar{x, H'})$:] Then $P_0$ is non-parallel $H'$-switching.
            Let $H'''$ be the $H$-expansion obtained from $H'$ by switching onto $P_0$.
            Note that $\Start{P_2} \in V(\intree{v, H'''} \cup \rpath{v, H'''})$.

            We remark that $\End{P_2} \in \V{\instar{w, H'}}$ and show that $\End{P_2} \in \V{\instar{w, H'''}}$.
            If $\Start{P_0} \in \V{\instar{v, H'}}$, then $\instar{w,H'} = \instar{w, H'''}$, which implies $\End{P_2} \in \V{\instar{w, H'''}}$.
            Otherwise, $\End{P_2} \in \V{\outstar{u, H'}}$ by definition of escape.
            Then $\End{P_2} \in \V{\outstar{u, H'} \cap \instar{w, H'}}$.
            Since $\Start{P_0} \notin \V{\outstar{u, H'}}$ and $\End{P_0} \notin \V{\instar{w, H'}}$, $\outstar{u, H'} \cap \instar{w, H'} = \outstar{u, H'''} \cap \instar{w, H'''}$, which implies $\End{P_2} \in \V{\instar{w, H'''}}$.
            
            If $(v,w) \notin \E{H}$, then $P_2$ is $H'''$-augmenting of type~\cref{item:augmenting_1} for $\Start{P_2} \in \V{\intree{v, H'''} \graphminus R_{v, H'''}}$ and of type~\cref{item:augmenting_2} for $\Start{P_2} \in \V{\outstar{v, H'''}}$.

            Thus, by \cref{thm:add_augmenting_path}, we assume that $(v, w) \in \E{H}$ and show that we can apply \cref{thm:construct_bracelet} to $P$ in $H'$.
            Then $D$ admits either a basic~$H$-augmentation or a bracelet~$H$-augmentation.
            See~\cref{fig:escape_non_proper}.
            
            The path $P$ starts in $\Start{P_0} \in V(\outstar{x, H'})$ and ends in $\End{P_2} \in V(\instar{w, H'})$.
            Furthermore there are edges $(x,v), (v,w) \in \E{H}$ and an edge $(u,v) \in \E{H}$ such that $P$ is properly laced with the path $Q$ in $\outstar{u,H'} \cap \instar{v, H'}$, $P\cap Q$ is exactly one path and $P$ is otherwise internally disjoint to $H'$.
            Thus we can apply \cref{thm:construct_bracelet}.\qedhere
    \end{description}
\end{proof}

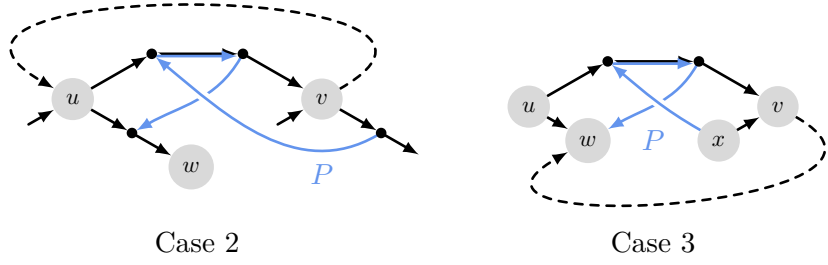
\begin{figure}[!ht]
    \centering
        \begin{tikzpicture}[scale=1]
			
	\pgfdeclarelayer{background}
	\pgfdeclarelayer{foreground}
			
	\pgfsetlayers{background,main,foreground}
			
	\begin{pgfonlayer}{main}

        \node (C) [v:ghost] {};
        \node (L) [v:ghost,position=180:30mm from C] {};
        \node (R) [v:ghost,position=0:30mm from C] {};

        \node (Llabel) [v:ghost,position=270:17mm from L] {Case 2};
        \node (Rlabel) [v:ghost,position=270:17mm from R] {Case 3};

         \node (Lfig) [v:ghost,position=0:0mm from L] {
         \begin{tikzpicture}[scale=1]
			
    \pgfdeclarelayer{background}
	\pgfdeclarelayer{foreground}
			
	\pgfsetlayers{background,main,foreground}
			
	\begin{pgfonlayer}{main}
	    \node (u) [v:model] {$u$};
        \node (z_1) [v:main,position=30:12mm from u] {};
        \node (z_1g) [v:ghost,position=330:0.6mm from z_1] {};
        \node (z_2) [v:main,position=0:12mm from z_1] {};
        \node (z_2g) [v:ghost,position=210:0.6mm from z_2] {};
        \node (v) [v:model,position=330:12mm from z_2] {$v$};

        \node (b) [v:ghost,position=210:7mm from v] {};
        \node (v_1) [v:main,position=330:9mm from v] {};
        \node (v_2) [v:ghost,position=330:6mm from v_1] {};

        \node (u_1) [v:main,position=330:9mm from u] {};
        \node (w) [v:model,position=330:9mm from u_1] {$w$};
        \node (y) [v:ghost,position=210:7mm from u] {};

        \node (Plabel) [v:ghost,position=270:10mm from v] {\textcolor{CornflowerBlue}{$P$}}; 
        
	\end{pgfonlayer}
			
	\begin{pgfonlayer}{background}
        \draw[e:main,->,white,out=330,in=210,looseness=2] (v) to (w);
        
        \draw[e:main,->] (u) to (z_1);
        \draw[e:main,->] (z_1) to (z_2);
        \draw[e:main,->] (z_2) to (v);
        \draw[e:main,->] (v) to (v_1);
        \draw[e:main,->] (v_1) to (v_2);
        \draw[e:main,->] (u) to (u_1);
        \draw[e:main,->] (u_1) to (w);
        \draw[e:main,->] (y) to (u);
        \draw[e:main,->] (b) to (v);
        \draw[e:main,dashed,->,out=30,in=150,looseness=2] (v) to (u);
        \draw[e:main,->,out=240,in=30,color=CornflowerBlue] (z_2) to (u_1);
        \draw[line width=4pt,color=white,line cap=round,out=210,in=320] (v_1) to (z_1);
        \draw[e:main,->,out=210,in=320,color=CornflowerBlue] (v_1) to (z_1);
        \draw[e:main,->,color=CornflowerBlue] (z_1g) to (z_2g);
        
 \end{pgfonlayer}	
			
	\begin{pgfonlayer}{foreground}
	\end{pgfonlayer}
   
\end{tikzpicture}};

\node (Rfig) [v:ghost,position=0:0mm from R] {
         \begin{tikzpicture}[scale=1]
			
    \pgfdeclarelayer{background}
	\pgfdeclarelayer{foreground}
			
	\pgfsetlayers{background,main,foreground}
			
	\begin{pgfonlayer}{main}
	    \node (u) [v:model] {$u$};
        \node (z_1) [v:main,position=30:12mm from u] {};
        \node (z_1g) [v:ghost,position=330:0.6mm from z_1] {};
        \node (z_2) [v:main,position=0:12mm from z_1] {};
        \node (z_2g) [v:ghost,position=210:0.6mm from z_2] {};
        \node (v) [v:model,position=330:12mm from z_2] {$v$};

        \node (x) [v:model,position=210:9mm from v] {$x$};

        \node (w) [v:model,position=330:9mm from u] {$w$};

        \node (Plabel) [v:ghost,position=300:12mm from z_1] {\textcolor{CornflowerBlue}{$P$}}; 
        
	\end{pgfonlayer}
			
	\begin{pgfonlayer}{background}
        \draw[e:main,->] (u) to (z_1);
        \draw[e:main,->] (z_1) to (z_2);
        \draw[e:main,->] (z_2) to (v);
        \draw[e:main,->] (u) to (w);
        \draw[e:main,->] (b) to (v);
        \draw[e:main,white,->,out=30,in=150,looseness=2] (v) to (u);
        \draw[e:main,->,out=240,in=30,color=CornflowerBlue] (z_2) to (w);
        \draw[line width=4pt,color=white,line cap=round,out=150,in=320] (x) to (z_1);
        \draw[e:main,->,out=150,in=320,color=CornflowerBlue] (x) to (z_1);
        \draw[e:main,->,color=CornflowerBlue] (z_1g) to (z_2g);
        \draw[e:main,->,dashed,out=330,in=210,looseness=2] (v) to (w);
 \end{pgfonlayer}	
			
	\begin{pgfonlayer}{foreground}
	\end{pgfonlayer}
   
\end{tikzpicture}};
\end{pgfonlayer}	

   	\begin{pgfonlayer}{background}
	\end{pgfonlayer}
 
	\begin{pgfonlayer}{foreground}
	\end{pgfonlayer}
 \end{tikzpicture}
    \caption{Cases 2 and 3 of the proof of \cref{lem:non-proper_escape}. The paths of $H'$ are shown in black, with the dotted path being the one whose existence we prove last.}
    \label{fig:escape_non_proper}
\end{figure}

Before showing how we obtain $H$-augmentations from escapes of duration at least $2$, we construct escapes starting with a \bad-path.

\begin{lemma}\label{lem:construct_escape_with_bad_path}
    Let $D$ be a digraph and $H$ be a strongly $2$-connected butterfly-minor of $D$.
    Further, let $H'$ be an expansion of $H$ in $D$ with an $H'$-escape $Q$ of duration at least $2$.
    Then there exists an $H^*$-escape $P$ for some $H$-expansion $H^*$ such that $P_0$ is $H^*$-bad.
    Additionally, if $Q_{2 \duration{Q}}$ is $H'$-bad, then $P_{2 \duration{P}}$ is $H^*$-bad.
\end{lemma}
\begin{proof}
        Let $Q$ be an $H'$-escape of duration at least $2$.
        Further, let $(u,v) \in E(H)$ such that $Q$ is an escape with respect to $(u,v)$.
        If $Q_0$ is $H'$-bad, then $Q$ and $H'$ are as desired.
        Thus we assume that $Q_0$ is non-parallel $H'$-switching.

        Let $H^*$ be the $H$-expansion obtained from $H'$ by switching onto $Q_0$.
        See \cref{fig:escape_construct_bad_path}.
        For a detailed explanation of how $H'$ and $H^*$ differ, we refer to the proof of~\cref{lem:expansion_stays_expansion_after_switching}.
        
        Note that $Q_2$ is $H^*$-bad with $\Start{Q_2} \in \V{\intree{v, H^*} \cup \rpath{v, H^*}}$ by construction of $H^*$.
        Further, for any even $j \in [4,2 \duration{Q} -2]$, $Q_j$ is $H^*$-bad with $\Start{Q_j}, \End{Q_j} \in \V{\outstar{u, H^*} \cap \instar{v, H^*}}$.
        For any odd $\ell \in [3,2 \duration{Q} -1]$, $Q_j \subseteq \outstar{u, H^*} \cap \instar{v, H^*}$ holds.
        In particular, $\Start{Q_{2 \duration{Q}}} \in V(\instar{v, H^*} \cap \outstar{u, H^*})$.

        We show that $\End{Q_{2 \duration{Q}}} \in \V{\instar{u, H^*} \cup \outtree{u, H^*} \cup \instar{w, H^*}}$ for some $(u,v) \neq (u,w) \in E(H)$.
        Then $Q_{2 \duration{Q}}$ is either non-parallel $H^*$-switching such that $\End{Q_{2 \duration{Q}}} \in V(\Brace{\instar{w, H^*} \cup \outtree{u, H^*}} \graphminus R_{u, H^*})$ for some $(u,v) \neq (u,w) \in \E{H}$ or $H^*$-\bad such that $\End{Q_{2 \duration{Q}}} \in V(\instar{u, H^*} \cup \outtree{u, H^*})$ since $\Start{Q_{2 \duration{Q}}} \in \V{\outstar{u, H^*} \cap \instar{v, H^*}}$.
        Furthermore, we show that if $Q_{2 \duration{Q}}$ is $H'$-bad, then $Q_{2 \duration{Q}}$ is $H^*$-bad.
        This implies that $P\coloneqq \Start{Q_2} Q$ is the desired $H^*$-escape.
        
        If $\Start{Q_0} \in \V{\instar{v, H'} \graphminus \outstar{u, H'}}$, then $H'$ and $H^*$ differ only in $\intree{v}$, the $u$-$v$-bridge and possibly in the $x$-$v$-bridge for some $(u,v) \neq (x, v) \in E(H)$.
        Thus $\outtree{u, H'} = \outtree{u, H^*}$, $\instar{u, H'} = \instar{u, H^*}$ and $\instar{w, H'} = \instar{w, H^*}$.
        This implies $\End{Q_{2 \duration{Q}}} \in \V{\instar{u, H'} \cup \outtree{u, H'} \cup \instar{w, H'}} = \V{\instar{u, H^*} \cup \outtree{u, H^*} \cup \instar{w, H^*}}$ for some $(u,v) \neq (u,w) \in E(H)$.
        In particular, if $Q_{2 \duration{Q}}$ is $H'$-bad, then $Q_{2 \duration{Q}}$ is $H^*$-bad.

        If $\Start{Q_0} \notin \V{\instar{v, H'} \graphminus \outstar{u, H'}}$, then $\Start{Q_0} \in \V{\outstar{x}}$ for some $(u,v) \neq (x,v) \in E(H)$.
        Furthermore, $\End{Q_{2 \duration{Q}}} \in \V{\outstar{u, H'} \graphminus \instar{v, H'}}$ by definition of escape.
        Thus $H'$ and $H^*$ differ only in $\intree{v}$, the $u$-$v$-bridge, the $x$-$v$-bridge, and possible $\rpath{x} \cup \outtree{x}$.
        Since $\Start{Q_0} \notin \V{\outstar{u, H'}}$ and $\End{Q_0} \notin \V{\outstar{u, H'} \graphminus \instar{v, H'}}$, $\outstar{u, H'} \graphminus \instar{v, H'} \subseteq \outstar{u, H^*} \graphminus \instar{v, H^*}$.
        Thus $\End{Q_{2 \duration{Q}}} \in \V{\outstar{u, H'} \graphminus \instar{v, H'}} \subseteq \V{\outstar{u, H^*} \graphminus \instar{v, H^*}}$, which implies that $\End{Q_{2 \duration{Q}}} \in \V{\instar{u, H^*} \cup \outtree{u, H^*} \cup \instar{w, H^*}}$ for some $(u,v) \neq (u,w) \in E(H)$.
        In particular, if $Q_{2 \duration{Q}}$ is $H'$-bad, then $Q_{2 \duration{Q}}$ is $H^*$-bad.
        This completes the proof.
\end{proof}

\begin{figure}[!ht]
    \centering
        \begin{tikzpicture}[scale=1]
			
    \pgfdeclarelayer{background}
	\pgfdeclarelayer{foreground}
			
	\pgfsetlayers{background,main,foreground}
			
	\begin{pgfonlayer}{main}
	    \node (v) [v:model] {$v$};
        \node (v_1) [v:ghost,position=150:9mm from v] {};
        \node (v_2) [v:ghost,position=210:9mm from v] {};
        \node (w) [v:model,position=330:12mm from v] {$w$};
        \node (z_1) [v:main,position=30:9mm from v] {};
        \node (z_1g) [v:ghost,position=330:0.6mm from z_1] {};
        \node (z_2) [v:main,position=30:6mm from z_1] {};
        \node (z_2g) [v:ghost,position=330:0.6mm from z_2] {};
        \node (z_3) [v:main,position=30:6mm from z_2] {};
        \node (z_3g) [v:ghost,position=330:0.6mm from z_3] {};
        \node (z_4) [v:main,position=30:6mm from z_3] {};
        \node (z_4g) [v:ghost,position=330:0.6mm from z_4] {};
        \node (u) [v:model,position=30:9mm from z_4] {$u$};
        
        \node (y) [v:main,position=150:9mm from u] {};
        \node (x) [v:model,position=150:9mm from y] {$x$};
        \node (u_1) [v:ghost,position=30:9mm from u] {};
        \node (u_2) [v:ghost,position=330:9mm from u] {};
        
        \node (Qlabel) [v:ghost,position=80:15mm from v] {\textcolor{CornflowerBlue}{$Q$}};
        \node (Q0label) [v:ghost,position=270:8mm from x] {\textcolor{CornflowerBlue}{$Q_0$}};
        \node (Q2label) [v:ghost,position=270:8mm from z_4] {\textcolor{CornflowerBlue}{$Q_2$}};
	\end{pgfonlayer}
			
	\begin{pgfonlayer}{background}

        \draw [line width=6pt,opacity=0.4,color=PastelOrange,line cap=round] (v_1.center) to (w.center);

        \draw [line width=6pt,opacity=0.4,color=PastelOrange,line cap=round]
        (v_2.center) to (u_1.center)
        (x.center) to (y.center)
        (y.center) [bend right=30] to (z_3.center);

        \draw [line width=6pt,opacity=0.4,color=PastelOrange,line cap=round] (u.center) to (u_2.center);
        
        \draw[e:main,->] (v_1) to (v);
        \draw[e:main,->] (v_2) to (v);
        \draw[e:main,->] (v) to (w);
        \draw[e:main,->] (v) to (z_1);
        \draw[e:main,->] (z_1) to (z_2);
        \draw[e:main,->] (z_2) to (z_3);
        \draw[e:main,->] (z_3) to (z_4);
        \draw[e:main,->] (z_4) to (u);
        \draw[e:main,->] (y) to (u);
        \draw[e:main,->] (x) to (y);
        \draw[e:main,->] (u) to (u_1);
        \draw[e:main,->] (u) to (u_2);

        \draw[e:main,->,color=CornflowerBlue] (z_1g) to (z_2g);
        \draw[e:main,->,color=CornflowerBlue] (z_3g) to (z_4g);

        \draw[e:main,->,out=220,in=90,color=CornflowerBlue] (y) to (z_3);
        \draw[e:main,->,out=270, in=340,color=CornflowerBlue] (z_4) to (z_1);
        \draw[e:main,->,out=150, in=70,color=CornflowerBlue] (z_2) to (v);
        
	\end{pgfonlayer}	
			
	\begin{pgfonlayer}{foreground}
	\end{pgfonlayer}
   
\end{tikzpicture}
   
    \caption{The $H$-expansions $H'$, \textcolor{PastelOrange}{$H^*$} and the path \textcolor{CornflowerBlue}{$Q$} in the proof of \cref{lem:construct_escape_with_bad_path}.}
    \label{fig:escape_construct_bad_path}
\end{figure}
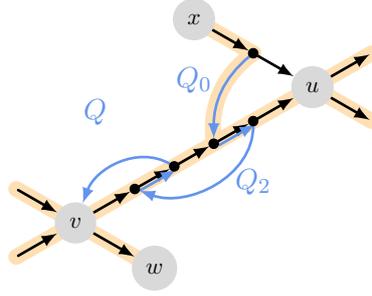

\ProperEscape*
\begin{proof}
    The graph $\Inverted{H}'$ is an expansion of $\Inverted{H}$ in $\Inverted{D}$ with an $\Inverted{H}'$-escape of duration at least $2$.
    By \cref{lem:construct_escape_with_bad_path}, there exists an $\Inverted{H}$-expansion $\Inverted{H}''$ with an $\Inverted{H}''$-escape $O$ such that $O_0$ is $H''$-bad.
    Then $P\coloneqq \Inverted{O}$ is an $H''$-escape such that $P_{2 \duration{P}}$ is $H''$-bad.

    If $P$ has duration $1$, $D$ admits an $H$-augmentation by \cref{lem:non-proper_escape}.
    Thus we can assume that $P$ has duration at least $2$ and apply \cref{lem:construct_escape_with_bad_path} to $P$ in order to obtain an $H$-expansion $H'''$ with an $H'''$-escape $Q$ such that $Q_0$ and $Q_{2 \duration{Q}}$ are $H'''$-bad.

    Let $(u,v) \in E(H)$ such that $Q$ is an $H'''$-escape with respect to $(u,v)$.
    Then $\Start{Q} \in V(\outstar{v, H'''} \cup \intree{v, H'''})$, $\End{Q} \in V(\outtree{u, H'''} \cup \instar{u, H'''})$ since $Q_0$ and $Q_{2 \duration{O}}$ are $H'''$-bad.
    Furthermore, $Q$ is properly laced with the path in $\outstar{u, H'''} \cap \instar{v, H'''}$ and otherwise internally disjoint from $H'''$ by definition of escape.
    Thus, we can apply~\cref{thm:add_chain}, which implies that $D$ admits an $H$-augmentation.
\end{proof}

\section{Directed splitter theorem}
\label{sec:mainproof}

In this section we make use of the introduced concepts and observations about them to prove the desired construction sequence for all strongly 2-connected digraphs.

\directedSplitterTheorem*

\begin{proof}
    \setcounter{claimcounter}{0}
    We show the statement by induction on $\Abs{\E{D}}-\Abs{\E{H}}$.
    The inductions basis holds true since $\Abs{\E{D}}=\Abs{\E{H}}$ implies $H \cong D$.
    
    For the induction step let $\Abs{\E{D}}>\Abs{\E{H}}$.
    We show that there exists an $H$-augmentation $K$ such that $D$ contains a $K$-expansion.
    Note that $\Abs{\E{K}} > \Abs{\E{H}}$ implies $\Abs{\E{D}}-\Abs{\E{K}} < \Abs{\E{D}}-\Abs{\E{H}}$.
    Thus we can apply the induction hypothesis to $D$ and $K$, and obtain a sequence $D_0, \dots, D_n$ as in the statement such that $D_0 = K$ and $D_n = D$.
    Then $H, D_0, \dots, D_n$ forms the desired sequence.

    For showing that there exists an $H$-augmentation $K$ such that $D$ contains a $K$-expansion we proceed with a case distinction on the structure of the expansions of $H$.
    Notice that $H^{\star}\subsetneq D$ for all $H$-expansions $H^{\star}$ in $D$.
    \begin{description}
        \item[Case 1 -- augmenting path:] Let $H'$ be an $H$-expansion in $D$ and assume that there exists an $H'$-augmenting path $Q$.
        Then, by \cref{thm:add_augmenting_path}, there exists a basic $H$-augmentation $K$ such that $K' \coloneqq H' \cup Q$ is a $K$-expansion in $D$.
    \end{description}

    \renewcommand{\labelenumi}{\textbf{\theenumi}}
    \renewcommand{\theenumi}{(\arabic{enumi})}
    \begin{enumerate}[labelindent=0pt,labelwidth=\widthof{\ref{item:assumption_1}},leftmargin=4ex]
        \item\label{item:assumption_1} Hence, from here on we may assume that \emph{no} $H$-expansion $H^{\star}$ in $D$ has an $H^{\star}$-augmenting path.
    \newcounter{enumTemp}
    \setcounter{enumTemp}{\theenumi}
    \end{enumerate}
       
    \begin{description}
        \item[Case 2 -- non-trivial branchset with in- and out-degree exactly two:]
        Let $H'$ be an $H$-expan-sion in $D$ and $v \in \V{H}$ such that $\InDegree{H}{v} = \OutDegree{H}{v} = 2$ and $\intree{v} \cup \rpath{v} \cup \outtree{v}$ is non-trivial.
        It follows from the assumption on the degrees of $v$ that $\rpath{v}=\intree{v} \cup \rpath{v} \cup \outtree{v}$ and thus, $\rpath{v}$ is non-trivial as well.
        Now choose $H''$ to be an $H$-expansion such that $\rpath{v}$ is of minimal length among all $H$-expansions for which $\rpath{v}$ is non-trivial.
        
        First, we observe that for each $H''$-switching path $Q$ with
        $|\Set{\Start{Q}, \End{Q}} \cap V(\rpath{v})| = 1$,
        we have $\Set{\Start{Q}, \End{Q}} \cap V(\rpath{v}) \subsetneq \Set{\Start{\rpath{v}}, \End{\rpath{v}}}$.
        Otherwise, one could switch onto $Q$ to obtain an $H$-expansion in which $\rpath{v}$ is shorter but still non-trivial.
        
        We distinguish two subcases:
        Either there is an $H''$-switching path $P_1$ starting in $\Start{\rpath{v}}$ and an $H''$-switching path $P_2$ ending in $\End{\rpath{v}}$ which are otherwise disjoint to $\rpath{v}$, or no such two paths exist.
        
        We begin with the former case that such paths $P_1$ and $P_2$ exist.
        By~\cref{lem:laced}, we can assume that $P_1$ and $P_2$ are laced.
        Further, by~\cref{lem:expansion_stays_expansion_after_switching}, we can obtain an $H$-expansion $H'''$ from $H''$ by switching\footnote{From now on, we apply \cref{lem:expansion_stays_expansion_after_switching} when switching onto switching paths without explicitly referring to it.} onto $P_1$.
        Let $Q_1 \subsetneq H''$ be the $H'''$-switching path obtained by switching onto $P_1$.

        We name a few vertices in $H$.
        Let $a$ be the in-neighbour of $v$ such that $\Start{P_2}\in \V{\outstar{a}}$ and let $x$ be the other in-neighbour of $v$.
        Similarly, let $b$ be the out-neighbour of $v$ such that $\End{Q_1}\in \V{\instar{b}}$ and let $y$ be the other out-neighbour of $v$.

        Notice that $\End{P_2}=\Start{Q_1}$.
        If $P_1$ and $P_2$ are disjoint, then the concatenation of $P_2$ and $Q_1$ 
        allows us to apply~\cref{thm:construct_bracelet} which yields that $D$ admits a basic or bracelet $H$-augmentation.
        
        Therefore, we may assume that $P_1$ and $P_2$ are not disjoint and therefore properly laced.
        In case $P_1\cap P_2$ has a unique component, we can again apply \cref{thm:construct_bracelet} using $P_2$ in $H'''$ to conclude that $D$ admits a basic or bracelet $H$-augmentation.
        Similarly, if $P_1\cap P_2$ has exactly two components, we can apply \cref{thm:construct_bracelet_two} using $P_2$ in $H'''$ to conclude that $D$ admits a basic, chain, collarette or bracelet $H$-augmentation.
        
        Otherwise, let $L$ be the shortest $\Start{P_2}$-$V(P_1)$-subpath of $P_2$ and let $R$ be the shortest $V(P_1)$-$\End{P_2}$-subpath of $P_2$.

        Notice first that $(a,b)\in \E{H}$ as otherwise $L$ would be an augmenting path for $H'''$ contradicting \cref{item:assumption_1}.
        Hence, $L$ and $R$ are both switching for $H'''$.

        As $P_1\cap P_2$ has at least three components, we obtain a chain augmentation witnessed by a subpath of $P_2$ as follows, see \cref{fig:degree2bracelet2} for an illustration.
        Let $H''''$ be the $H$-expansion obtained from $H'''$ by first switching onto $L$ and then onto $R$.
        This is possible since $R$ remains switching by switching onto $L$.

        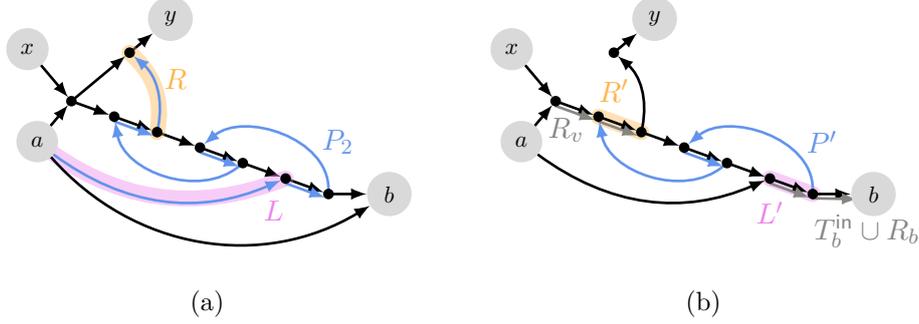
\begin{figure}[!ht]
        \centering
        \begin{subfigure}[c]{0.4\textwidth}
            \centering
            \begin{tikzpicture}[scale=1]
			
    \pgfdeclarelayer{background}
	\pgfdeclarelayer{foreground}
			
	\pgfsetlayers{background,main,foreground}
			
	\begin{pgfonlayer}{main}
	    \node (v) [v:main] {};
        \node (x) [v:model,position=130:9mm from v] {$x$};
        \node (a) [v:model,position=230:7mm from v] {$a$};

        \node (z) [v:main,position=40:10mm from v] {};
        \node (y) [v:model,position=40:7mm from z] {$y$};

        \node (v_1) [v:main,position=340:6mm from v] {};
        \node (v_1g) [v:ghost,position=250:0.6mm from v_1] {};        
        \node (v_2) [v:main,position=340:6mm from v_1] {};
        \node (v_2g) [v:ghost,position=250:0.6mm from v_2] {};        

        \node (v_3) [v:main,position=340:6mm from v_2] {};
        \node (v_3g) [v:ghost,position=250:0.6mm from v_3] {};        

        \node (v_4) [v:main,position=340:6mm from v_3] {};
        \node (v_4g) [v:ghost,position=250:0.6mm from v_4] {};        

        \node (v_5) [v:main,position=340:6mm from v_4] {};
        \node (v_5g) [v:ghost,position=250:0.6mm from v_5] {};        

        \node (v_6) [v:main,position=340:6mm from v_5] {};
        \node (v_6g) [v:ghost,position=250:0.6mm from v_6] {};

        \node (b) [v:model,position=360:8mm from v_6] {$b$};
        
        \node (Llabel) [v:ghost,position=250:4.5mm from v_5] {\textcolor{LavenderMagenta}{$L$}};
        \node (Rlabel) [v:ghost,position=330:7mm from z] {\textcolor{PastelOrange}{$R$}};
        \node (P2label) [v:ghost,position=80:7mm from v_6] {\textcolor{CornflowerBlue}{$P_2$}};
	\end{pgfonlayer}
			
	\begin{pgfonlayer}{background}

        \draw[e:main,->] (x) to (v);
        \draw[e:main,->] (a) to (v);
        \draw[e:main,->] (v) to (z);
        \draw[e:main,->] (z) to (y);
        \draw[e:main,->] (v) to (v_1);
        \draw[e:main,->] (v_1) to (v_2);
        \draw[e:main,->] (v_2) to (v_3);
        \draw[e:main,->] (v_3) to (v_4);
        \draw[e:main,->] (v_4) to (v_5);
        \draw[e:main,->] (v_5) to (v_6);
        \draw[e:main,->] (v_6) to (b);

        \draw [line width=6pt,opacity=0.4,color=LavenderMagenta,bend right=30,line cap=round] (a.center) to (v_5.center);
        \draw [line width=6pt,opacity=0.4,color=PastelOrange,bend right=30,line cap=round] (v_2.center) to (z.center);

        \draw[e:main,->,bend right=30,color=CornflowerBlue] (a) to (v_5);
        \draw[e:main,->,color=CornflowerBlue] (v_5g) to (v_6g);
        \draw[e:main,->,out=90,in=40,color=CornflowerBlue] (v_6) to (v_3);
        \draw[e:main,->,color=CornflowerBlue] (v_3g) to (v_4g);
        \draw[e:main,->,out=220,in=280,color=CornflowerBlue] (v_4) to (v_1);
        \draw[e:main,->,color=CornflowerBlue] (v_1g) to (v_2g);
        \draw[e:main,->,bend right=30,color=CornflowerBlue] (v_2) to (z);

        \draw[e:main,->,bend right=40] (a) to (b);

	\end{pgfonlayer}	
			
	\begin{pgfonlayer}{foreground}
	\end{pgfonlayer}
   
\end{tikzpicture}
   
            \caption{}
            \label{fig:deg2first}
        \end{subfigure}
        \begin{subfigure}[c]{0.4\textwidth}
            \centering
            \begin{tikzpicture}[scale=1]
			
    \pgfdeclarelayer{background}
	\pgfdeclarelayer{foreground}
			
	\pgfsetlayers{background,main,foreground}
			
	\begin{pgfonlayer}{main}
	    \node (v) [v:main] {};
        \node (vg) [v:ghost,position=250:0.6mm from v] {};
        \node (x) [v:model,position=130:9mm from v] {$x$};
        \node (a) [v:model,position=230:7mm from v] {$a$};

        \node (z) [v:main,position=40:10mm from v] {};
        \node (y) [v:model,position=40:7mm from z] {$y$};

        \node (v_1) [v:main,position=340:6mm from v] {};
        \node (v_1g) [v:ghost,position=250:0.6mm from v_1] {};        
        \node (v_2) [v:main,position=340:6mm from v_1] {};
        \node (v_2g) [v:ghost,position=250:0.6mm from v_2] {};        

        \node (v_3) [v:main,position=340:6mm from v_2] {};
        \node (v_3g) [v:ghost,position=250:0.6mm from v_3] {};        

        \node (v_4) [v:main,position=340:6mm from v_3] {};
        \node (v_4g) [v:ghost,position=250:0.6mm from v_4] {};        

        \node (v_5) [v:main,position=340:6mm from v_4] {};
        \node (v_5g) [v:ghost,position=250:0.6mm from v_5] {};        

        \node (v_6) [v:main,position=340:6mm from v_5] {};
        \node (v_6g) [v:ghost,position=250:0.6mm from v_6] {};

        \node (b) [v:model,position=360:8mm from v_6] {$b$};
        \node (bg) [v:ghost,position=200:2mm from b] {}; 
        \node (Llabel) [v:ghost,position=270:4.5mm from v_5] {\textcolor{LavenderMagenta}{$L'$}};
        \node (Rlabel) [v:ghost,position=270:5mm from z] {\textcolor{PastelOrange}{$R'$}};
        \node (P2label) [v:ghost,position=80:7mm from v_6] {\textcolor{CornflowerBlue}{$P'$}};

        \node (Rvlabel) [v:ghost,position=290:4mm from v] {\textcolor{DarkGray}{$R_v$}};
        \node (Rblabel) [v:ghost,position=260:5mm from b] {\textcolor{DarkGray}{$\intree{b}\cup \rpath{b}$}};
	\end{pgfonlayer}
			
	\begin{pgfonlayer}{background}
        \draw [line width=6pt,opacity=0.4,color=LavenderMagenta,line cap=round] (v_5.center) to (v_6.center);
        \draw [line width=6pt,opacity=0.4,color=PastelOrange,line cap=round] (v_1.center) to (v_2.center);
        
        \draw[e:main,->] (x) to (v);
        \draw[e:main,->] (a) to (v);
        \draw[e:main,->] (z) to (y);
        \draw[e:main,->] (v) to (v_1);
        \draw[e:main,->] (v_1) to (v_2);
        \draw[e:main,->] (v_2) to (v_3);
        \draw[e:main,->] (v_3) to (v_4);
        \draw[e:main,->] (v_4) to (v_5);
        \draw[e:main,->] (v_5) to (v_6);
        \draw[e:main,->] (v_6) to (b);

        \draw[e:main,->,bend right=30] (a) to (v_5);
        \draw[e:main,->,out=90,in=40,color=CornflowerBlue] (v_6) to (v_3);
        \draw[e:main,->,color=CornflowerBlue] (v_3g) to (v_4g);
        \draw[e:main,->,out=220,in=280,color=CornflowerBlue] (v_4) to (v_1);
        \draw[e:main,->,bend right=30] (v_2) to (z);

        \draw[e:main,->,color=DarkGray] (vg) to (v_1g);
        \draw[e:main,->,color=DarkGray] (v_1g) to (v_2g);

        \draw[e:main,->,color=DarkGray] (v_5g) to (v_6g);
        \draw[e:main,->,color=DarkGray] (v_6g) to (bg);

        \draw[e:main,white,->,bend right=40] (a) to (b);
	\end{pgfonlayer}	
			
	\begin{pgfonlayer}{foreground}
	\end{pgfonlayer}
   
\end{tikzpicture}
   
            \caption{}
            \label{fig:deg2second}
        \end{subfigure}
        \caption{The subcase of Case 2 in the proof of \cref{thm:mainthm2} that applies \cref{thm:add_chain}.
                If \cref{thm:construct_bracelet} is not applicable, we may assume that the intersection of $P_1$ and $P_2$ has at least two components.
                After switching onto $P_1$, we thus can obtain paths $L$ and $R$ as the first and last segment of $P_2$.
                As $L$ is not augmenting, the edge $(a,b)$ has to exist.
                We may now switch onto both subpaths $L$ and $R$ of $P_2$ \subref{fig:deg2first}.
                For the resulting $H$-expansion $H''''$ it follows that a subpath of the remaining part of $P_2$ witnesses a chain augmentation \subref{fig:deg2second} in $D$.}
        \label{fig:degree2bracelet2}
    \end{figure}

        Let $L'$ be the component of $P_1 \cap P_2$ intersecting $L$ and $R'$ be the component of $P_1 \cap P_2$ intersecting $R$.
        Now let $P'$ be the shortest subpath of $P_2$
        starting in $L'$ and ending in $R'$.
        We have that $\Start{P'} \in \V{\outstar{b,H''''}} \cup \V{\intree{b,H''''}}$ and $\End{P'} \in \V{\rpath{v,H''''}} \subseteq \V{\outtree{v,H''''}} \cup \V{\instar{v,H''''}}$ and also $P'\cap H'''' = \left(P'\cap P_1'\right)\cup\{ \Start{P'},\End{P'}\}$ where $P_1'\coloneqq \instar{b,H''''}\cap \outstar{v,H''''}\subseteq P_1$. 
        Note that as $P_1\cap P_2$ has at least three components the intersection $\instar{b,H''''}\cap \outstar{v,H''''}$ is not empty.
        By~\cref{thm:add_chain}, $P'$ witnesses that $D$ admits a chain or a collarette $H$-augmentation.
        
        We turn our attention to the case where we do not have both an $H''$-switching path $P_1$ starting in $\Start{\rpath{v}}$ and an $H''$-switching path $P_2$ ending in $\End{\rpath{v}}$ which have exactly one endpoint on $\rpath{v}$.

        Recall the above observation that, by the minimality of $\rpath{v}$, any switching path with exactly one end in $\rpath{v}$ must intersect $\rpath{v}$ in $\Start{\rpath{v}}$ or $\End{\rpath{v}}$.
        Hence, if there exists an $H''$-switching path starting in $\Start{\rpath{v}}$ and not ending on $\rpath{v}$, then there is no $H''$-switching path ending in $\End{\rpath{v}}$ and starting in $\V{D} \setminus\V{\rpath{v}}$, which makes $\Start{\rpath{v}}$ a \blocking vertex.
        Therefore, we can apply \cref{thm:blocking_vertex_implies_augmentation} to obtain that $D$ admits a basic~$H$-augmentation, a chain~$H$-augmentation, a collarette~$H$~augmentation or a bracelet~$H$-augmentation.
        Otherwise, assume that there are no $H''$-switching paths starting in $\Start{\rpath{v}}$ and ending in $\V{D}\setminus\V{\rpath{v}}$, then $\End{\rpath{v}}$ is a \blocking vertex and we can again apply \cref{thm:blocking_vertex_implies_augmentation}.
    \end{description}

    \begin{enumerate}[labelindent=0pt,labelwidth=\widthof{\ref{item:assumption_3}},leftmargin=4ex]
        \setcounter{enumi}{\theenumTemp}
        \item\label{item:assumption_3} Therefore, we may assume that in every $H$-expansion $H^{\star}$ for every $v\in \V{H}$ with in- and out-degree exactly two the path $\rpath{v,H^{\star}}$ is trivial.
    \setcounter{enumTemp}{\theenumi}
    \end{enumerate}
        
    \begin{description}
        \item[Case 3 -- non-trivial branchset with in- or out-degree at least three:]
        Let $H'$ be an $H$-expan-sion in $D$ and $v \in \V{H}$ such that $\InDegree{H}{v, H'}$ or $\OutDegree{H}{v, H'}$ is at least $3$ and $\intree{v, H'} \cup \rpath{v, H'} \cup \outtree{v, H'}$ is non-trivial.
        Without loss of generality we may assume that $\InDegree{H}{v} \geq 3$ since otherwise \cref{lem:inverted_augmentation} permits us to consider $\Inverted{H}$ and $\Inverted{D}$ instead.
        We denote by $s$ be the root of $\intree{v}$.

        If $s$ is a blocking vertex for $v$, then we can apply \cref{thm:blocking_vertex_implies_augmentation} to obtain that $D$ admits a basic~$H$-augmentation, a chain~$H$-augmentation, a collarette~$H$~augmentation or a bracelet~$H$-augmentation since the branchset of $v$ is non-trivial.

        Otherwise there exists an $H'$-switching path $Q$ starting in $V(H' \graphminus \rpath{v})$ and ending in $V(\rpath{v} - s)$.
        We may switch onto $Q$ and thereby obtain a new $H$-expansion $H''$ where $\intree{v,H''}$ is non-trivial.
        We construct an $H$-expansion $H'''$ from $H''$ by fixing $\intree{v, H''}$ and picking $\outstar{v}$ such that $\rpath{v}$ has minimal length along all possible out-stars at $v$ under fixed $\intree{v,H''}$.
        Let $r$ be the root of $\outtree{v}$ in $H'''$.
        Notice that there trivially does not exist a switching path starting in $\V{D}\setminus V(r\rpath{v})$ and ending in $V(r\rpath{v})\setminus\{ r\} = \emptyset$.
        Moreover, suppose there exists a switching path $Q$ which starts in $V(\rpath{v}r)\setminus\{ r\}$ and ends in $\V{D}\setminus V(\rpath{v}r)$.
        Then $Q$ must end in $(\outtree{v} - r) \cup \instar{w}$ for some $w \in \OutNeighbours{H}{v}$ and thus, by switching onto $Q$, we would get an $H$-expansion $H''''$ such that $\intree{v,H'''}=\intree{v,H''''}$ and $\rpath{v,H''''}\subsetneq \rpath{v,H'''}$ which contradicts our choice of $H'''$.
        Therefore, $r$ is \blocking for $v$ in $H'''$.
        This and the assumption that $V(\intree{v}) \setminus \Set{r} \neq \emptyset$ allows us to apply \cref{thm:blocking_vertex_implies_augmentation} to obtain that $D$ admits a basic~$H$-augmentation, a chain~$H$-augmentation, a collarette~$H$~augmentation or a bracelet~$H$-augmentation.
    \end{description}

    The outcome of Case 3 allows us to assume that the branchsets of all vertices $v\in \V{H}$ where $v$ has out- or in-degree at least three are trivial.
    Moreover, notice that the branchset of any vertex with $v\in \V{H}$ such that $v$ has in- and out-degree exactly two equals $\rpath{v}$ which is, by \cref{item:assumption_3}, trivial.
    Finally, as $H$ is strongly $2$-connected, all of its vertices fall into one of these two categories.
    
    \begin{enumerate}[labelindent=0pt,labelwidth=\widthof{\ref{item:assumption_4}},leftmargin=4ex]
        \setcounter{enumi}{\theenumTemp}
        \item\label{item:assumption_4} Therefore, we may assume that in every $H$-expansion $H^{\star}$ every $v\in \V{H}$ has a trivial branchset.
    \setcounter{enumTemp}{\theenumi}
    \end{enumerate}
        
    \begin{description}
        
        \item[Case 4 -- non-parallel switching path:]
        Let $H'$ be an $H$-expansion in $D$.
        Notice that if there were a non-parallel switching path $Q$, then by switching $H'$ onto $Q$, we would obtain an $H$-expansion with a non-trivial branchset, which would contradict \cref{item:assumption_4}.
    \end{description}
        
   \begin{enumerate}[labelindent=0pt,labelwidth=\widthof{\ref{item:assumption_5}},leftmargin=4ex]
        \setcounter{enumi}{\theenumTemp}
        \item\label{item:assumption_5}
        Therefore, by \cref{lem:different-earpaths}, we may assume that for every $H$-expansion $H^{\star}$ every $H^{\star}$-\earpath is either parallel switching or \bad.
    \setcounter{enumTemp}{\theenumi}
    \end{enumerate}

    With this we are ready to tackle the last but also most involved case.
    By \cref{item:assumption_4}, for every $H$-expansion $H^{\star}$ the branchset of every vertex $v \in \V{H^{\star}}$ is trivial.
    We refer to the unique vertex of $H^{\star}$ in the branchset of $v$ as $v_{H^{\star}}$.

    \begin{description}        
        \item[Case 5 -- all branchsets trivial and no non-parallel switching path:]
        Let $\widetilde{H}$ be an $H$-expansion in $D$.
        As a first step we select $H'$ to be an $H$-expansion in $D$ with the additional property that there exists $(u,v)\in \E{H}$ such that the path $\outstar{u,H'}\cap\instar{v,H'}$
        %\SW{path between stars}
        %\GS{You mean the bridge?}
        has at least one vertex.
        We achieve such an $H$-expansion $H'$ as follows.

        In case $\widetilde{H}$ satisfies the property above, we set $H' \coloneqq \widetilde{H}$.

        If no such path exists, then $\widetilde{H}\subseteq D$ is isomorphic to $H$.
        In this case, we choose $Q$ to be any $\widetilde{H}$-\earpath.
        Let $u \in \V{H}$ be the vertex such that $\Start{Q} = u_{\widetilde{H}}$ and $v \in \V{H}$ be the vertex such that $\End{Q} = v_{\widetilde{H}}$.
        As there are no augmenting paths $\Brace{u,v} \in \E{H}$ and $\Brace{u_{\widetilde{H}},v_{\widetilde{H}}} \in \E{D}$.
        Furthermore, as $D$ is simple, $Q$ has a length of at least two and $Q$ is a parallel switching path.
        We obtain the $H$-expansion $H'$ by switching $\widetilde{H}$ onto $Q$.
        
        So, in both cases, $H'$ is a subdivision of $H$, and there exists an edge $(u,v)\in \E{H}$ such that the unique $u_{H'}$-$v_{H'}$-path $Q$ in $H'$ corresponding to this edge has an internal vertex.
        We fix $u$, $v$, and $Q$ for what follows.

        \begin{claim}
            \label{claim:starts_and_ends_in_stars}
            There exist $H'$-\earpaths $\outP$ and $\inP$ as follows: the path $\outP$ starts in an internal vertex of $Q$ with $\End{\outP} \in \V{H'}\setminus\V{Q-u_{H'}}$ and the path $\inP$ ends in an internal vertex of $Q$ with $\Start{\inP} \in \V{H'}\setminus\V{Q-v_{H'}}$.
            Additionally, for any such paths we have $\End{\outP} \in \V{\instar{u}}$ and $\Start{\inP} \in \V{\outstar{v}}$.
        \end{claim}
        \begin{claimproof}
            As $D$ is strongly 2-connected, every internal vertex of $Q$ has out- and in-degree at least two, and $v_{H'}$ is not a cutvertex in $D$.
            Thus there exists a path $\outP$ starting at an internal vertex of $Q$ and ending in $\V{H'}\setminus\V{Q-u_{H'}}$.
            Similarly, $u_{H'}$ is not a cutvertex, and we can find a path $\inP$ ending in an internal vertex of $Q$ and starting in $\V{H'}\setminus\V{Q-v_{H'}}$.
        
            By \cref{item:assumption_1}, neither $\outP$ nor $\inP$ is augmenting.
            Therefore for every $w \notin \OutNeighbours{H}{u}$ with $w \neq u$ we have $\End{\outP} \notin \V{\instar{w}}$ and for every $t \notin \InNeighbours{H}{v}$ with $t \neq v$ we have $\Start{\inP} \notin \V{\outstar{t}}$.
            
            Also, by \cref{item:assumption_5}, neither $\outP$ nor $\inP$ is non-parallel switching.
            Thus, we have $\End{\outP} \notin \V{\instar{w}}$ for every $w \in \OutNeighbours{H}{u}$ with $w \neq u$ and $\Start{\inP} \notin \V{\outstar{t}}$ for every $t \in \InNeighbours{H}{v}$ with $t \neq v$.
            
            In particular, $\End{\outP} \notin \V{\instar{v}}$ and $\Start{\inP} \notin \V{\outstar{u}}$.   
            Therefore, the two observations above yield $\End{\outP} \in \V{\instar{u}}$ and $\Start{\inP} \in \V{\outstar{v}}$ as desired.
        \end{claimproof}
        
        Next, we want to choose specific paths that allow us to apply \cref{thm:add_chain}.
        
        \begin{claim}
            \label{claim:tidy}
            There are $H'$-\earpaths $\inPnew$ and $\outPnew$ such that
            \begin{itemize}
                \item $\inPnew$ is an $\outstar{v'}$-$\V{Q-u_{H'} - v_{H'}}$-path  and $\outPnew$ is a $\V{Q- u_{H'} - v_{H'}}$-$\instar{u'}$-path, and
                \item if $\inPnew$ and $\outPnew$ are internally disjoint, then
                there exists an $\End{\inPnew}$-$\Start{\outPnew}$-path $\Pstar$ that is properly laced with $Q$ and internally disjoint from $\inPnew$, $\outPnew$, and $H' \graphminus (Q - u_{H'} - v_{H'})$.
            \end{itemize}
        \end{claim}
        \begin{claimproof}
            Let $\inP$ and $\outP$ be paths as provided by \cref{claim:starts_and_ends_in_stars}.
            The path $\outP$ is not needed and can be discarded.

            Let $\Pstaraux$ be some $\End{I}$-$\V{H' \graphminus (Q - u_{H'})}$-path in $D - v_{H'}$.
            By \cref{lem:laced} we can choose $\Pstaraux$ laced with $Q - u_{H'} - v_{H'}$.
            Then $\Pstaraux$ is properly laced with $Q$ since $\Start{\Pstaraux}\in V(Q)$.

            Let $x_o$ be the last intersection of $\Pstaraux$ with the internal vertices of $Q$ with respect to $\leq_{\Pstaraux}$.

            We define $\outPaux \coloneqq x_o \Pstaraux$.
            By construction, we know that $\outPaux$ starts in an internal vertex of $Q$, ends in $\V{H \graphminus (Q - u_{H'})}$ and is internally disjoint from $Q$.
            In particular, $\outPaux$ is an $H'$-\earpath.
            This allows us to apply \cref{claim:starts_and_ends_in_stars}, which yields that $\outPaux$ ends in $\instar{u_{H'}}$.

            If $\inP$ and $\outPaux$ are not internally disjoint, then we define $\inPnew \coloneqq \inP$ and $\outPnew \coloneqq \outPaux$ which have the desired properties.
            
            So now we assume that $\inP$ and $\outPaux$ are internally disjoint.
            In this case, we define $\outPnew \coloneqq \outPaux$.
            In order to obtain $\inPnew$, we need a few more steps.
            First let ${\Pstaraux}'$ be the prefix of $\Pstaraux$ until $x_0$, that is ${\Pstaraux}'=\Pstaraux x_0$.
            Note that ${\Pstaraux}'$ intersects $H'$ only in internal vertices of $Q$ and intersects $O'$ only in $x_0$.

            Note that $\inP$ and ${\Pstaraux}'$ intersect in $\End{\inP}=\Start{{\Pstaraux}'}$.
            Let $y$ be the last vertex of ${\Pstaraux}'$ on $\inP$ with respect to $\leq_{{\Pstaraux}'}$.
            Let $z$ be the first vertex (with respect to $\leq_{{\Pstaraux}'}$) that $y{\Pstaraux}'$ has on $Q$ (possibly $y =z$).
            We choose $\inPnew \coloneqq \inP y {\Pstaraux}' z$.
            As $y$ is the last intersection of ${\Pstaraux}'$ with $\inP$, the subpath $z {\Pstaraux}' x_0$ is internally disjoint from $\inPnew$ and $\outPnew$.
            We set $\Pstar \coloneqq z {\Pstaraux}' x_0$ and remark that $\Pstar$ is properly laced with $Q$ since $\Pstaraux$ is properly laced with $Q$.
            Thus we again obtain all the required properties.
            See \cref{fig:pathcases} for an illustration.
        \end{claimproof}
                    
            \begin{figure}[!ht]
                \centering
                    \input{fig_preparing_for_7.2}
                \caption{Two situations from the proof of \cref{claim:tidy} in the proof of \cref{thm:mainthm2}.
                The left side, (i), shows the situation when the path $\Pstaraux$ and $\inP$ intersect only in $\End{\inP}=\Start{\Pstaraux}$.
                In this case, the path $\Pstar$ is defined to be $\Pstaraux x_o$.
                The right side, (ii), depicts the situation where $\Pstaraux$ and $\inP$ intersect in an internal vertex.
                In this case we let $\inPnew$ turn onto $\Pstaraux$ at the last intersection of $\Pstaraux$ and $\inP$ along $\Pstaraux$ and set $\Pstar$ to be $\End{\inPnew} \Pstaraux x_0$.}
                \label{fig:pathcases}
            \end{figure}
        
        Now let $\inPnew$ and $\outPnew$ be as provided by \cref{claim:tidy}.
        We distinguish three cases as follows.

        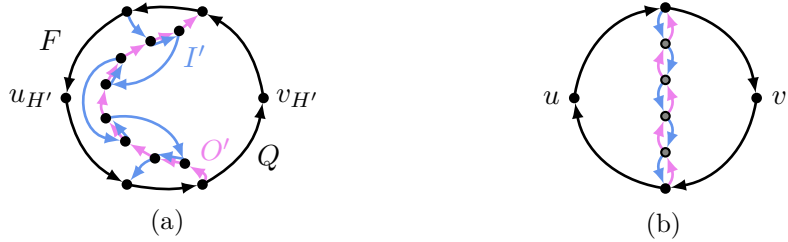
\begin{figure}[!ht]
        \centering
        \begin{subfigure}[c]{0.4\textwidth}
            \centering
            \begin{tikzpicture}[scale=1]
			
    \pgfdeclarelayer{background}
    \pgfdeclarelayer{foreground}
			
	\pgfsetlayers{background,main,foreground}
			
	\begin{pgfonlayer}{main}

        \node (u) [v:main] {};
        \node (v) [v:main,position=0:26mm from u] {};

        \node (x_1) [v:main,position=305:14mm from u] {};
        \node (x_2) [v:main,position=235:14mm from v] {};

        \node (y_1) [v:main,position=125:14mm from v] {};
        \node (y_2) [v:main,position=55:14mm from u] {};

        \node (z_1) [v:main,position=220:4mm from y_1] {};
        \node (z_1g) [v:ghost,position=270:0.6mm from z_1] {};
       
        \node (z_2) [v:main,position=200:4mm from z_1] {};
        \node (z_2g) [v:ghost,position=270:0.6mm from z_2] {};

        \node (z_3) [v:main,position=210:4.5mm from z_2] {};
        \node (z_3g) [v:ghost,position=310:0.6mm from z_3] {};
        
        \node (z_4) [v:main,position=240:4mm from z_3] {};
        \node (z_4g) [v:ghost,position=310:0.6mm from z_4] {};

        \node (z_5) [v:main,position=270:4.5mm from z_4] {};
        \node (z_5g) [v:ghost,position=0:0.7mm from z_5] {};
       
        \node (z_6) [v:main,position=310:4mm from z_5] {};
        \node (z_6g) [v:ghost,position=0:0.7mm from z_6] {};
        
        \node (z_7) [v:main,position=330:4.5mm from z_6] {};
        \node (z_7g) [v:ghost,position=90:0.6mm from z_7] {};

        \node (z_8) [v:main,position=350:4mm from z_7] {};
        \node (z_8g) [v:ghost,position=90:0.6mm from z_8] {};

        \node (ulabel) [v:ghost,position=180:4.5mm from u] {$u_{H'}$};
        \node (vlabel) [v:ghost,position=0:4.5mm from v] {$v_{H'}$};

        \node (Flabel) [v:ghost,position=105:8mm from u] {$F$};
        \node (Qlabel) [v:ghost,position=275:8mm from v] {$Q$};
        \node (Olabel) [v:ghost,position=70:5mm from x_2] {\textcolor{LavenderMagenta}{$O'$}};
        \node (Ilabel) [v:ghost,position=260:6mm from y_1] {\textcolor{CornflowerBlue}{$I'$}};
         
	    \end{pgfonlayer}
			
	    \begin{pgfonlayer}{background}

        \draw[e:main,->,bend right=25] (u) to (x_1);
        \draw[e:main,->,bend right=10] (x_1) to (x_2);
        \draw[e:main,->,bend right=25] (x_2) to (v);

        \draw[e:main,->,bend right=25] (v) to (y_1);
        \draw[e:main,->,bend right=10] (y_1) to (y_2);
        \draw[e:main,->,bend right=25] (y_2) to (u);
        
        \draw[e:main,->,out=60,in=330,color=LavenderMagenta] (x_2) to (z_8);
        \draw[e:main,->,color=LavenderMagenta] (z_8) to (z_7);
        \draw[e:main,->,color=LavenderMagenta] (z_7) to (z_6);
        \draw[e:main,->,color=LavenderMagenta] (z_6) to (z_5);
        \draw[e:main,->,out=120,in=260,color=LavenderMagenta] (z_5) to (z_4);
        \draw[e:main,->,color=LavenderMagenta] (z_4) to (z_3);
        \draw[e:main,->,color=LavenderMagenta] (z_3) to (z_2);
        \draw[e:main,->,color=LavenderMagenta] (z_2) to (z_1);
        \draw[e:main,->,out=40,in=220,color=LavenderMagenta] (z_1) to (y_1);

        \draw[e:main,->,out=290,in=140,color=CornflowerBlue] (y_2) to (z_2);
        \draw[e:main,->,color=CornflowerBlue] (z_2g) to (z_1g);
        \draw[e:main,->,out=260,in=0,color=CornflowerBlue] (z_1) to (z_4);
        \draw[e:main,->,color=CornflowerBlue] (z_4g) to (z_3g);
        \draw[e:main,->,out=190,in=170,looseness=1.4,color=CornflowerBlue] (z_3) to (z_6);
        \draw[e:main,->,color=CornflowerBlue] (z_6g) to (z_5g);
        \draw[e:main,->,out=10,in=110,color=CornflowerBlue] (z_5) to (z_8);
        \draw[e:main,->,color=CornflowerBlue] (z_8g) to (z_7g);
        \draw[e:main,->,out=200,in=50,color=CornflowerBlue] (z_7) to (x_1);

	\end{pgfonlayer}	
			
    \begin{pgfonlayer}{foreground}
    \end{pgfonlayer}
   
\end{tikzpicture}
            \caption{}
            \label{fig:Case5collarettea}
        \end{subfigure}
        \begin{subfigure}[c]{0.4\textwidth}
            \centering
            \begin{tikzpicture}[scale=1]
			
    \pgfdeclarelayer{background}
	\pgfdeclarelayer{foreground}
			
	\pgfsetlayers{background,main,foreground}
			
	\begin{pgfonlayer}{main}
	    \node (u) [v:main] {};
        \node (v) [v:main,position=0:24mm from u] {};

        \node (g) [v:ghost,position=0:12mm from u] {};

        \node (x_1) [v:main,position=90:12mm from g] {};
        \node (x_2) [v:main,position=90:7.2mm from g,fill=DarkGray] {};
        \node (x_3) [v:main,position=90:2.4mm from g,fill=DarkGray] {};
        \node (x_4) [v:main,position=270:2.4mm from g,fill=DarkGray] {};
        \node (x_5) [v:main,position=270:7.2mm from g,fill=DarkGray] {};
        \node (x_6) [v:main,position=270:12mm from g] {};
        
        \node (ulabel) [v:ghost,position=180:3mm from u] {$u$};
        \node (vlabel) [v:ghost,position=0:3mm from v] {$v$};
	\end{pgfonlayer}
			
	\begin{pgfonlayer}{background}

        \draw[e:main,->,bend left=35] (u) to (x_1);
        \draw[e:main,->,bend left=35] (x_1) to (v);
        
        \draw[e:main,->,bend left=35] (v) to (x_6);
        \draw[e:main,->,bend left=35] (x_6) to (u);

        \draw[e:main,->,out=230,in=110,color=CornflowerBlue] (x_1) to (x_2);
        \draw[e:main,->,out=50,in=290,color=LavenderMagenta] (x_2) to (x_1);
        \draw[e:main,->,out=310,in=70,color=CornflowerBlue] (x_2) to (x_3);
        \draw[e:main,->,out=130,in=250,color=LavenderMagenta] (x_3) to (x_2);
        \draw[e:main,->,out=230,in=110,color=CornflowerBlue] (x_3) to (x_4);
        \draw[e:main,->,out=50,in=290,color=LavenderMagenta] (x_4) to (x_3);
        \draw[e:main,->,out=310,in=70,color=CornflowerBlue] (x_4) to (x_5);
        \draw[e:main,->,out=130,in=250,color=LavenderMagenta] (x_5) to (x_4);
        \draw[e:main,->,out=230,in=110,color=CornflowerBlue] (x_5) to (x_6);
        \draw[e:main,->,out=50,in=290,color=LavenderMagenta] (x_6) to (x_5);
	\end{pgfonlayer}	
			
	\begin{pgfonlayer}{foreground}
	\end{pgfonlayer}
   
\end{tikzpicture}
   
            \caption{}
            \label{fig:Case5collaretteb}
        \end{subfigure}
        \caption{The collarette augmentation from Subcase 1 of Case 5 in the proof of \cref{thm:mainthm2}.
            \subref{fig:Case5collarettea} shows the situation with the two \earpaths $\inPnew$ and $\outPnew$ for the $H$-expansion $H'$ while \subref{fig:Case5collaretteb} depicts the corresponding collarette augmentation of $H$.}
        \label{fig:Case5collarette}
    \end{figure}
        
        \paragraph{The two vertices $\End{\outPnew}$ and $\Start{\inPnew}$ both lie on the path $F \coloneqq \instar{u}\cap\outstar{v}$ and $\End{\outPnew} \leq_F \Start{\inPnew}$.}
        We distinguish three cases based on the order of the vertices $\End{\inPnew}$ and $\Start{\outPnew}$ on the path $Q$ and whether $\inPnew$ and $\outPnew$ intersect.
            \begin{description}
                \item[Subcase 1] If $\End{\inPnew} \leq_{Q} \Start{\outPnew}$, then $\inPnew$ and $\outPnew$ witness that $D$ admits a collarette augmentation with respect to the two edges $(u,v)$ and $(v,u)$.
                For an illustration of the way this collarette augmentation is witnessed by the paths $\inPnew$ and $\outPnew$, see~\cref{fig:Case5collarette}.
                \item[Subcase 2] If $\inPnew$ and $\outPnew$ are disjoint and $\Start{\outPnew} <_{Q} \End{\inPnew}$, then let $F'\coloneqq \End{\outPnew} \Brace{\outstar{v}\cap\instar{u}} \Start{\inPnew}$ and apply \cref{thm:add_chain} to $F'$ and the path $J\coloneqq\outPnew \cdot \End{\outPnew}F'\Start{\inPnew} \inPnew$.
                This yields that $D$ admits a chain or collarette augmentation of $H$ witnessed by $J$.
                
                \item[Subcase 3] If $\inPnew$ and $\outPnew$ intersect and $\Start{\outPnew} <_{Q} \End{\inPnew}$, then there is a $\Start{\outPnew}$-$\End{\inPnew}$-path $J$ in $\inPnew \cup \outPnew$, which is parallel switching.
                Let $H''$ be obtained by switching onto $J$.
                Now, we choose $\inPnew'$ to be the shortest $\Start{\inPnew}$-$\V{J}$-subpath of $\inPnew$ and $\outPnew'$ the shortest $\V{J}$-$\End{\outPnew}$-subpath.
                Now we end up in Subcase 1.
                See \cref{fig:simplesubcases} for an illustration of the two cases above.
            \end{description}

        \begin{figure}[!ht]
                \centering
                    \begin{tikzpicture}[scale=1]
			
	\pgfdeclarelayer{background}
	\pgfdeclarelayer{foreground}
			
	\pgfsetlayers{background,main,foreground}
			
	\begin{pgfonlayer}{main}

        \node (C) [v:ghost] {};
        \node (R) [v:ghost,position=0:50mm from C] {};

        \node (Llabel) [v:ghost,position=270:15mm from C] {Subcase 2};
        \node (Rlabel) [v:ghost,position=270:15mm from R] {Subcase 3};

        \node (Cfig) [v:ghost,position=0:0mm from C] {
            \begin{tikzpicture}[scale=1]
			
            	\pgfdeclarelayer{background}
	            \pgfdeclarelayer{foreground}
			
	            \pgfsetlayers{background,main,foreground}
			
		    \begin{pgfonlayer}{main}

                \node (u) [v:main] {};
                \node (v) [v:main,position=0:26mm from u] {};

                \node (x_1) [v:main,position=315:13mm from u] {};
                \node (x_2) [v:main,position=225:13mm from v] {};

                \node (y_1) [v:main,position=135:13mm from v] {};
                \node (y_2) [v:main,position=45:13mm from u] {};

                \node (ulabel) [v:ghost,position=180:4.5mm from u] {$u_{H'}$};
                \node (vlabel) [v:ghost,position=0:4.5mm from v] {$v_{H'}$};

                \node (Ilabel) [v:ghost,position=260:8.2mm from y_2] {\textcolor{CornflowerBlue}{$I'$}};
                \node (Olabel) [v:ghost,position=280:8.2mm from y_1] {\textcolor{LavenderMagenta}{$O'$}};

                \node (Flabel) [v:ghost,position=105:5mm from u] {$F$};
                \node (Qlabel) [v:ghost,position=275:5mm from v] {$Q$};

                \node (PPlabel) [v:ghost,position=225:5.1mm from y_1] {\textcolor{PastelOrange}{$P'$}};

	        \end{pgfonlayer}
			
	          \begin{pgfonlayer}{background}   

                \draw[line width=6pt,opacity=0.6,color=PastelOrange, bend right=15,line cap=round] (y_2) to (x_2);
                \draw[line width=6pt,opacity=0.6,color=PastelOrange, bend right=15,line cap=round] (x_1) to (y_1); 
                \draw[line width=6pt,opacity=0.6,color=PastelOrange, bend right=10,line cap=round] (y_1) to (y_2);

                \draw[e:main,->,bend right=25] (u) to (x_1);
                \draw[e:main,->,bend right=10] (x_1) to (x_2);
                \draw[e:main,->,bend right=25] (x_2) to (v);

                \draw[e:main,->,bend right=25] (v) to (y_1);
                \draw[e:main,->,bend right=10] (y_1) to (y_2);
                \draw[e:main,->,bend right=25] (y_2) to (u);

                \draw[e:main,->,color=CornflowerBlue,bend right=15] (y_2) to (x_2);
                \draw[e:main,color=LavenderMagenta,->,bend right=15] (x_1) to (y_1);
        
	           \end{pgfonlayer}

	           \begin{pgfonlayer}{foreground}
	           \end{pgfonlayer}
   
            \end{tikzpicture}
        };

        \node (Rfig) [v:ghost,position=0:0mm from R] {
            \begin{tikzpicture}[scale=1]
			
            	\pgfdeclarelayer{background}
	            \pgfdeclarelayer{foreground}
			
	            \pgfsetlayers{background,main,foreground}
			
	        \begin{pgfonlayer}{main}

                \node (u) [v:main] {};
                \node (v) [v:main,position=0:26mm from u] {};

                \node (x_1) [v:main,position=315:13mm from u] {};
                \node (x_2) [v:main,position=225:13mm from v] {};

                \node (y_1) [v:main,position=135:13mm from v] {};
                \node (y_2) [v:main,position=45:13mm from u] {};

                \node (z_1) [v:main,position=225:5.5mm from y_2] {};
                \node (z_1g) [v:ghost,position=180:0.45mm from z_1] {};
                \node (z_2) [v:main,position=300:4mm from z_1] {};
                \node (z_2g) [v:ghost,position=180:0.45mm from z_2] {};

                \node (z_3) [v:main,position=0:11mm from u] {};
                \node (z_3g) [v:ghost,position=270:0.45mm from z_3] {};
                \node (z_4) [v:main,position=0:4mm from z_3] {};
                \node (z_4g) [v:ghost,position=270:0.45mm from z_4] {};

                \node (z_6) [v:main,position=45:5.5mm from x_2] {};
                \node (z_6g) [v:ghost,position=180:0.45mm from z_6] {};
                \node (z_5) [v:main,position=120:4mm from z_6] {};
                \node (z_5g) [v:ghost,position=180:0.45mm from z_5] {};

                \node (ulabel) [v:ghost,position=180:4.5mm from u] {$u_{H'}$};
                \node (vlabel) [v:ghost,position=0:4.5mm from v] {$v_{H'}$};

                \node (Flabel) [v:ghost,position=105:5mm from u] {$F$};
                \node (Qlabel) [v:ghost,position=275:5mm from v] {$Q$};
                \node (Jlabel) [v:ghost,position=90:4mm from x_1] {\textcolor{PastelOrange}{$J$}};
         
	        \end{pgfonlayer}
			
	        \begin{pgfonlayer}{background}

                \draw[line width=6pt,color=PastelOrange,opacity=0.6,line cap=round] (x_1) to (z_5);
                \draw[line width=6pt,color=PastelOrange,opacity=0.6,line cap=round] (z_5) to (z_6);
                \draw[line width=6pt,color=PastelOrange,opacity=0.6,line cap=round] (z_6) to (x_2);
        
                \draw[e:main,->,bend right=25] (u) to (x_1);
                \draw[e:main,->,bend right=10] (x_1) to (x_2);
                \draw[e:main,->,bend right=25] (x_2) to (v);

                \draw[e:main,->,bend right=25] (v) to (y_1);
                \draw[e:main,->,bend right=10] (y_1) to (y_2);
                \draw[e:main,->,bend right=25] (y_2) to (u);

                \draw[e:main,->,color=CornflowerBlue] (y_2) to (z_1);
                \draw[e:main,->,color=CornflowerBlue] (z_1) to (z_2);
                \draw[e:main,->,color=CornflowerBlue] (z_2) to (z_3);
                \draw[e:main,->,color=CornflowerBlue] (z_3) to (z_4);
                \draw[e:main,->,color=CornflowerBlue] (z_4) to (z_5);
                \draw[e:main,->,color=CornflowerBlue] (z_5) to (z_6);
                \draw[e:main,->,color=CornflowerBlue] (z_6) to (x_2);

                \draw[e:main,->,color=LavenderMagenta] (x_1) to (z_5);
                \draw[e:main,->,color=LavenderMagenta] (z_5g) to (z_6g);
                \draw[e:main,->,color=LavenderMagenta,bend right=90] (z_6) to (z_3);
                \draw[e:main,->,color=LavenderMagenta] (z_3g) to (z_4g);
                \draw[e:main,->,color=LavenderMagenta,bend left=90] (z_4) to (z_1);
                \draw[e:main,->,color=LavenderMagenta] (z_1g) to (z_2g);
                \draw[e:main,->,color=LavenderMagenta] (z_2) to (y_1);
         
	        \end{pgfonlayer}	
			
	        \begin{pgfonlayer}{foreground}
	        \end{pgfonlayer}
   
            \end{tikzpicture}
        };        
 
	\end{pgfonlayer}
			
	\begin{pgfonlayer}{background}
	\end{pgfonlayer}	
			
	\begin{pgfonlayer}{foreground}
	\end{pgfonlayer}
   
\end{tikzpicture}
                \caption{Subcase 2 and 3 of Case 5 in the proof of \cref{thm:mainthm2}.}
                \label{fig:simplesubcases}
        \end{figure}
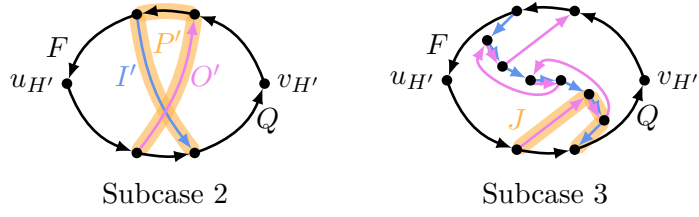

        \paragraph{The paths $\inPnew$ and $\outPnew$ are not internally disjoint and if $\Start{\inPnew}$ and $\End{\outPnew}$ lie on $F \coloneqq \instar{u}\cap\outstar{v}$ then $\Start{\inPnew} <_F  \End{\outPnew}$.}

        Let $x$ be the last vertex of $\V{\inPnew \cap \outPnew}$ along $\leq_{\outPnew}$.
        Note that $x$ is an internal vertex of $\inPnew$ and $\outPnew$ by assumption on $\inPnew$ and $\outPnew$.
        Then $J\coloneqq \inPnew x \outPnew$ is an $H'$-switching path.
        Let $H''$ be obtained by switching onto $J$.
        Let $\inPnew' \coloneqq x \inPnew$ and let $\outPnew'$ be the shortest $\Start{\outPnew}$-$\V{J}$-subpath of $\outPnew$.
        Then, with the paths $Q$, $\inPnew'$, and $\outPnew'$ we are, for $H''$, in the previous case.
        That is, $\End{\outPnew'}$ and $\Start{\inPnew'}$ lie on $F'\coloneqq \instar{u,H''}\cap\outstar{v,H''}$ and we have that $\End{\outPnew'}\leq_{F'} \Start{\inPnew'}$.

        \paragraph{The paths $\inPnew$ and $\outPnew$ are internally disjoint and if $\Start{\inPnew}$ and $\End{\outPnew}$ lie on $F \coloneqq \instar{u}\cap\outstar{v}$ then $\Start{\inPnew} <_F  \End{\outPnew}$.}
            In this case, by~\cref{claim:tidy}, there exists an $\End{\inPnew}$-$\Start{\outPnew}$-path $\Pstar$ such that we can apply \cref{thm:add_chain} using the path $\inPnew \cdot \Pstar \cdot \outPnew$, which is properly laced with $Q$, in order to obtain a witness for $D$ admitting a chain or collarette augmentation.

        With this, Case 5 is complete.
    \end{description}
    Note that as an augmentation of a strongly $2$-connected digraph is again strongly $2$-connected, the constructed sequence consists of only strongly $2$-connected digraphs.
\end{proof}

\paragraph{Acknowledgements:}
We wish to acknowledge William McCuaig.
While none of us has ever met them, their work continues to inspire us and we hope they would like the results presented in this paper.

The majority of research for this paper was conducted during the workshop on \emph{Directed minors and graph structure theory} held in 2022 and organised through the Sparse (Graphs) Coalition.
We want to thank all participants for the friendly and inspiring atmosphere during the workshop.

\bibliographystyle{alphaurl}

\begin{thebibliography}{AKKW16}

    \bibitem[AKKW16]{amiri2016erdos}
    Saeed~Akhoondian Amiri, Ken-ichi Kawarabayashi, Stephan Kreutzer, and Paul Wollan.
    \newblock {T}he {E}rdos-{P}{\'o}sa property for directed graphs.
    \newblock {\em arXiv preprint arXiv:1603.02504}, 2016.
    
    \bibitem[BG69]{BarnetteG1969}
    David~W. Barnette and Branko Gr\"{u}nbaum.
    \newblock On {S}teinitz's theorem concerning convex {$3$}-polytopes and on some properties of planar graphs.
    \newblock In {\em The {M}any {F}acets of {G}raph {T}heory ({P}roc. {C}onf., {W}estern {M}ich. {U}niv., {K}alamazoo, {M}ich., 1968)}, pages 27--40. Springer, Berlin, 1969.
    
    \bibitem[Gee96]{MR2694410}
    James~Ferdinand Geelen.
    \newblock {\em Matchings, matroids and unimodular matrices}.
    \newblock ProQuest LLC, Ann Arbor, MI, 1996.
    \newblock Thesis (Ph.D.)--University of Waterloo (Canada).
    \newblock URL: \url{http://gateway.proquest.com/openurl?url_ver=Z39.88-2004&rft_val_fmt=info:ofi/fmt:kev:mtx:dissertation&res_dat=xri:pqdiss&rft_dat=xri:pqdiss:NN09343}.
    
    \bibitem[GO09]{GeelenO2009}
    Jim Geelen and Sang-il Oum.
    \newblock Circle graph obstructions under pivoting.
    \newblock {\em J. Graph Theory}, 61(1):1--11, 2009.
    \newblock \href {https://doi.org/10.1002/jgt.20363} {\path{doi:10.1002/jgt.20363}}.
    
    \bibitem[Lit75]{Little1975}
    C.~H.~C. Little.
    \newblock A characterization of convertible (0,1)-matrices.
    \newblock {\em J. Combinatorial Theory Ser. B}, 18:187--208, 1975.
    \newblock \href {https://doi.org/10.1016/0095-8956(75)90048-9} {\path{doi:10.1016/0095-8956(75)90048-9}}.
    
    \bibitem[Lov87]{Lovasz1987}
    L\'{a}szl\'{o} Lov\'{a}sz.
    \newblock Matching structure and the matching lattice.
    \newblock {\em J. Combin. Theory Ser. B}, 43(2):187--222, 1987.
    \newblock \href {https://doi.org/10.1016/0095-8956(87)90021-9} {\path{doi:10.1016/0095-8956(87)90021-9}}.
    
    \bibitem[McC01]{mccuaig2001bracegeneration}
    William McCuaig.
    \newblock Brace generation.
    \newblock {\em J. Graph Theory}, 38(3):124--169, 2001.
    \newblock \href {https://doi.org/10.1002/jgt.1029} {\path{doi:10.1002/jgt.1029}}.
    
    \bibitem[Muz17]{muzi2017paths}
    Irene Muzi.
    \newblock {\em Paths and topological minors in directed and undirected graphs}.
    \newblock PhD thesis, Ph. D. Dissertation, Sapienza Universita Di Roma, 2017.
    
    \bibitem[NT07]{NorinT2007}
    Serguei Norine and Robin Thomas.
    \newblock Generating bricks.
    \newblock {\em J. Combin. Theory Ser. B}, 97(5):769--817, 2007.
    \newblock \href {https://doi.org/10.1016/j.jctb.2007.01.002} {\path{doi:10.1016/j.jctb.2007.01.002}}.
    
    \bibitem[Sey80]{seymour1980splitter}
    P.~D. Seymour.
    \newblock Decomposition of regular matroids.
    \newblock {\em J. Combin. Theory Ser. B}, 28(3):305--359, 1980.
    \newblock \href {https://doi.org/10.1016/0095-8956(80)90075-1} {\path{doi:10.1016/0095-8956(80)90075-1}}.
    
    \bibitem[ST87]{SeymourT1987}
    Paul Seymour and Carsten Thomassen.
    \newblock Characterization of even directed graphs.
    \newblock {\em J. Combin. Theory Ser. B}, 42(1):36--45, 1987.
    \newblock \href {https://doi.org/10.1016/0095-8956(87)90061-X} {\path{doi:10.1016/0095-8956(87)90061-X}}.
    
    \bibitem[Tit75]{titov1975constructive}
    V.K. Titov.
    \newblock {\em A constructive description of some classes of graphs}.
    \newblock PhD thesis, Moscow, 1975.
    
    \bibitem[Tut61]{Tutte1961}
    W.~T. Tutte.
    \newblock A theory of {$3$}-connected graphs.
    \newblock {\em Nederl. Akad. Wetensch. Proc. Ser. A 64 = Indag. Math.}, 23:441--455, 1961.
    
    \bibitem[Whi32]{Whitney1932}
    Hassler Whitney.
    \newblock Non-separable and planar graphs.
    \newblock {\em Trans. Amer. Math. Soc.}, 34(2):339--362, 1932.
    \newblock \href {https://doi.org/10.2307/1989545} {\path{doi:10.2307/1989545}}.
    
    \bibitem[Wie20]{wiederrecht2020dtwone}
    Sebastian Wiederrecht.
    \newblock Digraphs of directed treewidth one.
    \newblock {\em Discrete Math.}, 343(12):112124, 9, 2020.
    \newblock \href {https://doi.org/10.1016/j.disc.2020.112124} {\path{doi:10.1016/j.disc.2020.112124}}.
    
    \end{thebibliography}

\end{document}